\documentclass[10pt,a4paper]{article}

\usepackage[section]{placeins}
\usepackage{amsmath}
\usepackage{amsthm}
\usepackage{amssymb}
\usepackage{amsfonts}
\usepackage{lmodern}
\usepackage{bm}
\usepackage{mathtools}
\usepackage[T1]{fontenc}
\usepackage{framed, color}
\usepackage{dsfont}
\usepackage{graphicx}
\usepackage{hyperref}
\usepackage{latexsym}
\usepackage{mathrsfs}
\usepackage{subfig}
\usepackage{tikz}
\usepackage{mathdots}
\usepackage{verbatim}
\usepackage{epstopdf}
\usepackage[a4paper, left=3cm, 
		     right=3cm, top=2.97cm, 
		    ]{geometry}

\theoremstyle{plain}
\newtheorem{theorem}{Theorem}[section]
\newtheorem{proposition}[theorem]{Proposition}
\newtheorem{lemma}[theorem]{Lemma}

\newtheorem{remark}[theorem]{Remark}

\theoremstyle{definition}
\newtheorem{definition}[theorem]{Definition}

\newcommand{\RR}{\mathbb R}
\newcommand{\ZZ}{\mathbb Z}
\newcommand{\CC}{\mathbb C}
\newcommand{\NN}{\mathbb N}

\newcommand{\R}{\mathcal{R}}
\renewcommand{\L}{\mathcal{L}}
\newcommand{\tL}{\tilde{\mathcal{L}}}

\newcommand{\D}{\mathcal{D}}
\newcommand{\X}{\mathcal{X}}
\newcommand{\bu}{\bar{u}}
\newcommand{\bv}{\bar{v}}
\newcommand{\bw}{\bar{w}}

\newcommand{\rstar}{r_*}
\newcommand{\tend}{t_{\text{\textup{end}}}}
\newcommand{\runfile}{\texttt{runproofs.m}}
\renewcommand{\d}{\mbox{d}}

\newcommand{\bydef}{\,\stackrel{\mbox{\tiny\textnormal{\raisebox{0ex}[0ex][0ex]{def}}}}{=}\,} 

\DeclareMathOperator{\bbox}{Box}

\title{Validated integration of semilinear parabolic PDEs}
\author{Jan Bouwe van den Berg \thanks{VU Amsterdam, Department of Mathematics, De Boelelaan 1081, 1081 HV Amsterdam, The Netherlands. \texttt{janbouwe@few.vu.nl}} $\quad$ Maxime Breden \thanks{CMAP, CNRS, \'Ecole polytechnique, Institut Polytechnique de
Paris, 91120 Palaiseau, France. \texttt{maxime.breden@polytechnique.edu}} $\quad$ Ray Sheombarsing \thanks{VU Amsterdam, Department of Mathematics, De Boelelaan 1081, 1081 HV Amsterdam, The Netherlands. \texttt{r.s.s.sheombarsing@outlook.com}}} 
\date{}

\begin{document}

\maketitle

\begin{abstract} 
Integrating evolutionary partial differential equations (PDEs) is an essential ingredient for studying the dynamics of the solutions. Indeed, simulations are at the core of scientific computing, but their mathematical reliability is often difficult to quantify, especially when one is interested in the output of a given simulation, rather than in the asymptotic regime where the discretization parameter tends to zero. In this paper we present a computer-assisted proof methodology to perform rigorous time integration for scalar semilinear parabolic PDEs with periodic boundary conditions. We formulate an equivalent zero-finding problem based on a variation of constants formula in Fourier space. Using Chebyshev interpolation and domain decomposition, we then finish the proof with a Newton--Kantorovich type argument. The final output of this procedure is a proof of existence of an orbit, together with guaranteed error bounds between this orbit and a numerically computed approximation.
We illustrate the versatility of the approach with results for the Fisher equation, the Swift--Hohenberg equation, the Ohta--Kawasaki equation and the Kuramoto--Sivashinsky equation. We expect that this rigorous integrator can form the basis for studying boundary value problems for connecting orbits in partial differential equations.
\end{abstract}

\begin{center}
{\bf \small Keywords} \\ \vspace{.05cm}
{ \small Parabolic PDEs $\cdot$ Initial value problems $\cdot$ Computer-assisted proofs $\cdot$ A posteriori error estimates}
\end{center}

\begin{center}
{\bf \small Mathematics Subject Classification (2020)}  \\ \vspace{.05cm}
{\small 35K15 $\cdot$ 35K30 $\cdot$ 35K58 $\cdot$  37L65 $\cdot$ 65G20 $\cdot$ 65M15 $\cdot$ 65M70} 
\end{center}


\section{Introduction}
\label{sec:introduction}

During the last decades, computer-assisted proofs have become an increasingly effective tool in 
the study of nonlinear ordinary differential equations (ODEs) and dynamical systems in general. 
The rapid progress of the development of computer hardware has made it possible to put numerical 
simulations on a rigorous footing through the construction of theorems whose
hypotheses can be verified with the aid of a computer. Today, there exists a large variety of rigorous 
numerical methods for studying invariant objects in systems of ODEs, such as equilibria, periodic orbits, 
connecting orbits, invariant manifolds, etc. In particular, we mention the prominent software packages 
CAPD~\cite{KapMroWilZgl21}, COSY~\cite{COSY} and Intlab~\cite{Bun20,Intlab}.
Detailed knowledge of invariant objects can provide deep insight into the global structure of a dynamical system,
which for nonlinear systems is typically difficult to obtain solely from pen-and-paper analysis. 

While evolutionary partial differential equations (PDEs) and the associated \emph{infinite dimensional dynamical systems} are much more challenging, computer-assisted proofs are also gradually being developed  for this setting. In particular, going back to the early works~\cite{Nak88,Plu92}, many computer-assisted proofs techniques have been developed for elliptic problems, corresponding to stationary solutions of associated parabolic PDEs), see for instance~\cite{MR2679365,MR3408840,MR2338393,Ois95,TakLiuOis13,BerWil19,Yamamoto,MR1968368}. A complete review of these works is beyond the scope of this paper, but we refer the interested readers to the recent book~\cite{NakPluWat19} and the references therein. In addition to steady states, traveling waves solutions have been studied with computer-assistance~\cite{MR3281845, MR2220064,BerShe20}, as well as periodic solutions of PDEs~\cite{MR2728184,MR3633778,MR3623202,BerBreLesVee21,MR2049869}, even for some ill-posed problems~\cite{CasGamLes18}. More recently, rigorous enclosures for local stable and unstable manifolds in parabolic PDEs have been developed as well~\cite{MirRei19,BerJaqMir22}. We also mention the works~\cite{BucCaoGom22,CheHou22}, where other state-of-the-art computer-assisted proofs for PDEs are described.

In this article, we introduce a rigorous initial value solver for semilinear parabolic PDEs. One main motivation behind this work is to later use this rigorous solver to validate boundary value problems for parabolic PDEs, and to combine it with the above mentioned techniques for local stable and unstable manifolds~\cite{MirRei19,BerJaqMir22}, in order to validate connecting orbits in parabolic PDEs, potentially between two saddle equilibria. However, since rigorous integration of PDEs is already a relatively involved task, we focus in this work on IVPs only. More precisely, in this paper we develop a rigorous computational method for validating solutions of initial value problems for scalar semilinear parabolic equations, with periodic boundary conditions, of the form
\begin{align}
	\begin{cases}
		\dfrac{ \partial u}{ \partial t} = 
		(-1)^{R+1} \dfrac{ \partial^{2R} u}{ \partial x^{2R}} + \displaystyle\sum_{j=0}^{2R-1} \dfrac{\partial^{j} g^{(j)} (u)}{\partial x^j} ,
		& t \in (0,2\tau], \ x \in [0, 2\pi], \\[2ex]
		\dfrac{ \partial^{j} u}{ \partial x^{j}}u(t,0)=\dfrac{ \partial^{j} u}{ \partial x^{j}}u(t,2\pi),  & t \in [0,2\tau], \ j=0,\dots,2R-1, \\[2ex]
		u (0, x) = f ( x ), & x \in [0,2\pi], 
	\end{cases}
	\label{eq:PDE}
\end{align}
where $R\in\NN_{\geq 1}$ is a positive integer, $\tau>0$ determines the integration time, $g^{(j)}:\RR\to\RR$ are given polynomial functions for $j=0,\ldots,2R-1$, and $f: \RR \rightarrow \RR$ is a smooth $2\pi$-periodic function. 

\begin{remark}
The factor $2$ in the time interval $(0,2\tau]$ is of course inconsequential, and only introduced to simplify the expressions we obtain after rescaling time in Section~\ref{sec:functional_setup}.
\end{remark}

To sketch the context and be able to compare with the existing literature, we start with a more conceptional discussion rather than diving into details immediately.
Conceptually, one can think of~\eqref{eq:PDE} as being simply an equation of the form
\begin{align}
\label{eq:H}
\dfrac{ \partial u}{ \partial t} = H(u),
\end{align}
and we assume to have an approximate solution $\bu$ at our disposal. Our strategy for validating this approximate solution starts by using a well chosen splitting of the righthand side as 
\begin{align}
\label{eq:split}
H(u) = \L u + G(u),
\end{align}
where $\L$ is a linear operator which may or may not depend on time. 
When $\L$ is time \emph{independent}, as it essentially is in the current work (although it is taken piecewise constant in time at a later stage), 
we can rewrite~\eqref{eq:H} as a fixed point problem, using the semigroup associated to $\L$ and the fixed point operator $\tilde T$ defined by
\begin{align}
\label{eq:tildeT}
\tilde T(u)(t) = e^{t\L}f + \int_0^t e^{(t-s)\L}G(u(s))\d s,\quad t\in[0,2\tau].
\end{align}
Our validation procedure for proving the existence of an orbit near $\bu$ will be based on the Banach fixed-point theorem. Therefore, we would like $\tilde T$ to be a contraction on a neighborhood of $\bu$, preferably for long integration times $2\tau$.
\begin{remark}
We rely heavily on the structure of equation~\eqref{eq:PDE}, and more precisely on the fact that the highest order term $(-1)^{R+1} \dfrac{ \partial^{2R}}{ \partial x^{2R}}$ is dissipative. Indeed, this ensures that the associated semi-group is contracting on the higher order Fourier modes. Yet, we want our approach to be applicable for equations showcasing interesting dynamics, hence we must also be able to handle the appearance of some unstable modes, which means we allow the semigroup generated by $\L$ to not be contracting for (finitely many) low order modes.

We note that, of course, using the splitting $H(u) = \L u + G(u)$, one could always artificially choose $\L$ so that the associated semigroup is contracting, but that does not mean that $\tilde T$ itself will be contracting
for larger integration times, since this is also influenced by $G(u)$. 
For instance, the scalar ODE $u'=u$ can be split as $u'=-ku + (k+1)u$, with $k>0$, which gives rise to 
\begin{align*}
\tilde T(u)(t) = e^{-tk}f + (k+1)\int_0^t e^{-(t-s)k}u(s) \d s, 
\end{align*}
which is not contracting in $C^0$, except for short times.
\end{remark}
The pivotal choice of $\L$ that we make in this work is meant to help us getting a contracting operator $\tilde T$, even when the semigroup generated by $\L$ is not contracting for low order modes. To that end, we take for $\L$ an approximation of $DH(\bu)$, the Fr\'echet derivative of $H$ at the approximate solution, in order to make $DG(\bu) = \L - DH(\bu)$ small. However, by imposing that $\L$ does not depend on time, in general we cannot hope to get $\L$ arbitrarily close to $DH(\bu)$. Therefore, we further reformulate the problem by first turning it into a zero finding problem
\begin{align*}
F(u)(t) \bydef e^{t\L}f + \int_0^t e^{(t-s)\L} G(u(s)) \d s - u(t) = 0,\quad \text{for all }t\in[0,2\tau],
\end{align*}
which itself is then turned into another fixed point problem, given by a Newton-like operator 
\begin{align}
\label{eq:NK}
T:u\mapsto u -A F(u),
\end{align}
where $A$ is a suitable approximate inverse of $DF(\bu)$. Notice that, for $A=I$, we get $T=\tilde{T}$. Since, as mentioned previously, we already expect $\tilde T$ to be contracting for sufficiently high order Fourier modes, we only need $A$ to act nontrivially on finitely many modes. Once a suitable approximate inverse $A$ is defined, we then want to prove that $T$ is a contraction in a neighborhood of $\bu$, and this is what most of this paper is devoted to. 

Before we proceed, let us again emphasize that the issue of whether $T$ is contracting or not, and the issue of whether or not we are actually able to prove it (with computer assistance), both depend greatly on three main (related) ingredients:
\begin{itemize}
\item the choice of $\L$,
\item the choice of $A$,
\item the choice of the Banach space in which we try to prove the contraction.
\end{itemize}
In light of the above discussion, we now review some of the existing literature regarding rigorous integrators for parabolic PDEs, and describe how they relate to or differ from our proposed approach.

One class of rigorous PDE integrators, which use completely different techniques compared to the ones discussed in this paper, is the one developed in~\cite{MR2049869, MR2788972, MR1838755}. It is not centered around a fixed point argument, but rather based on 
the validated integration of a finite dimensional system of ODEs using Lohner-type algorithms developed in 
\cite{MR1930946} and on the notion of self-consisted bounds introduced in~\cite{MR1838755}. This methodology was further developed in~\cite{MR3167726} and has been used to rigorously study, among other things, globally attracting solutions in the one dimensional Burgers equation~\cite{MR3338669}, heteroclinic connections in the one-dimensional Ohta-Kawasaki model~\cite{Ohta} and very recently to prove chaos in the Kuramoto-Sivashinky equation~\cite{WilZgl20}.

A second family of rigorous PDE integrators, to which our approach belongs, is based on a fixed point reformulation which allows to prove, a posteriori, that a true solution exists in a small neighborhood of a numerically computed approximation. 
Within this class, we first mention the foundational work~\cite{MR2728184}, where numerical solutions are validated using a fixed point operator similar to $\tilde T$ in~\eqref{eq:tildeT}, and where $\L$ is in some sense chosen to be simple, namely $\L=DH(0)$. In the context of~\eqref{eq:PDE}, the approach from~\cite{MR2728184} would amount to taking\footnote{To be more precise, the operator $\L$ used in~\cite{MR2728184} is sometimes slightly modified, the parts that would contribute to eigenvalues with positive real part being moved to $G$.}
\begin{align}
\label{eq:L_ArioliKoch}
\L = (-1)^{R+1} \dfrac{ \partial^{2R}}{ \partial x^{2R}} + \displaystyle\sum_{j=0}^{2R-1} \alpha_j\dfrac{\partial^{j} }{\partial x^j},
\end{align} 
where $\alpha_j\in\RR$ is the derivative of $g^{(j)}$ at $0$. This choice for $\L$ is enticing for making it easier to control the associated semigroup, which is needed when trying to prove that $\tilde T$ is a contraction. This is especially true if the domain and the boundary conditions allow for an explicit diagonalization of such an $\L$, as is the case in~\cite{MR2728184} and in the current paper. However, while this choice of $\L$ simplifies the analysis of $\tilde T$, it also comes at the cost of making the operator $\tilde T$ potentially \emph{less contracting} ($DG(\bu)$ will not be arbitrarily small), and since~\cite{MR2728184} does not use a subsequent Newton-like reformulation, the fixed point operator can generally only be contracting for relatively small values of the integration time $2\tau$. A similar approach, at least at the level of choosing the fixed point reformulation, is used in~\cite{CyrLes22}\footnote{Formally, in~\cite{CyrLes22} Equation~\eqref{eq:H} is first rewritten as $\tilde{F}(u) = \left(\partial_t\right)^{-1}\left( \partial_t u - H(u)\right) = 0$, and then turned into the fixed point problem $\tilde{T}(u) = \left(D\tilde{F}(0)\right)^{-1} \left(D\tilde{F}(0) u -\tilde{F}(u) \right)=u$. However, noticing that $\left(D\tilde{F}(0)\right)^{-1}= \left(\partial_t - DH(0) \right)^{-1} \partial_t$, 
one finds that $\tilde{T}(u)$ is nothing but $\left(\partial_t - DH(0) \right)^{-1}\left(H(u) - DH(0)u\right)$, which corresponds to~\eqref{eq:tildeT} with $\L=DH(0)$.}, although the main point of the latter paper is different and revolves around a Fourier-Chebyshev (fully) spectral approach, see also Remark~\ref{rem:approx}.

On the other end of the spectrum, at least with regards to $\L$, are the techniques introduced in~\cite{TakLesJaqOka22,TakMizKubOis17}. There, in terms of the splitting~\eqref{eq:split}, $\L$ is chosen as a very accurate time dependent approximation of $DH(\bu)$. This means that the mild solution reformulation then reads
\begin{align*}
u(t) = U(t,0)f + \int_0^t U(t,s)G(u(s))\d s,\quad t\in[0,2\tau],
\end{align*}
where $U(t,s)$ is the evolution operator associated to $\L$, i.e. $U(t,s)\varphi$ is the solution at time $t$ of
\begin{align*}
	\begin{cases}
		\dfrac{ \partial u}{ \partial t} = 
		\L(t) u ,
		& t \in (s,2\tau], \ x \in [0, 2\pi], \\[2ex]
		\dfrac{ \partial^{j} u}{ \partial x^{j}}u(t,0)=\dfrac{ \partial^{j} u}{ \partial x^{j}}u(t,2\pi),  & t \in [s,2\tau], \ j=0,\dots,2R-1, \\[2ex]	
		u (s, x ) = \varphi ( x ), & x \in [0,2\pi ].
	\end{cases}
\end{align*}
By allowing $\L$ to depend on time, one can in principle make $DG(\bu)$ arbitrarily small, hence ensuring that the fixed point operator
\begin{align}
\label{eq:nonautonomousevolution}
\tilde T(u)(t) = U(t,0)f + \int_0^t U(t,s)G(u(s))\d s,\quad t\in[0,2\tau],
\end{align}
is contracting in a neighborhood of $\bu$. Nevertheless, when actually having to prove that $\tilde T$ is a contraction in this setting, one then has to explicitly control the evolution operator $U(t,s)$, which can be much more challenging and resource consuming than only having to control the semigroup associated to a time \emph{independent}~$\L$.

Yet another approach, which starts with a different viewpoint but for which the fixed point reformulation ends up being very similar, is the one in~\cite{HasKinNak20}, see also~\cite[Chapter 5]{NakPluWat19}. In the context of~\eqref{eq:PDE}, assuming $R=1$, the approach in~\cite{HasKinNak20} can be interpreted as taking $\L$ to be exactly equal to $DH(\bu)$, and the validation is then based on applying Schauder's fixed point theorem to the operator~\eqref{eq:nonautonomousevolution},
where now $U(t,s)$ is the evolution operator associated to $DH(\bu)$. This choice of $\L$ leads to $DG(\bu) = 0$, hence one should be able to find a small neighborhood of $\bu$ being mapped into itself by $\tilde T$, but doing so now requires controlling the evolution operator associated to $\L=DH(\bu)$. In the end, this is accomplished by introducing some approximation of $DH(\bu)$, see e.g.~\cite{HasKimMinNak19} for the details. We point out that some of the ideas in this approach but also in the other ones, can be traced back to the early work~\cite{Nak91}, in which $\L$ was simply taken to be the Laplacian.

All these works share some similarities regarding the fixed point operator used for the validation, or, at the very least, they can all be reformulated into this common framework. However, we emphasize that they then differ significantly when it comes to how this fixed point operator is analyzed, see also Remark~\ref{rem:approx}. Regarding the choice of $\L$, our own work can be seen as a compromise between the different approaches just mentioned, as we try to take for $\L$ an approximation of $DH(\bu)$ which is \emph{as good as we can while remaining time independent}, thereby avoiding the problem
of having to deal with a solution operator of a non-autonomous system.

\begin{remark}
In all of the approaches discussed above, as well as in the current work, when trying to validate solutions for larger and larger time intervals, it proves more efficient (and it is often mandatory in order to obtain a proof) to split the interval $[0,2\tau]$ into shorter pieces. Of course, one has to glue these subproblems together, and to do so there are two alternatives, which we call here \emph{time stepping} (also sometimes referred to as \emph{concatenation scheme}) and \emph{domain decomposition}. 

The idea of time stepping is to start with some initial condition and to rigorously integrate the associated initial value problem on a small time 
interval. If the proof is successful, we compute a rigorous enclosure for the endpoint of the orbit and try to integrate the \emph{enclosure} forward in 
time. This process is repeated as long as necessary (and possible). In order for long time integration to be feasible, the size of the 
enclosures should be ``managed'' properly. In particular, they should not grow too fast. 
This is where the dissipativity of the system can prove very helpful, if correctly leveraged.

An alternative to time stepping is domain decomposition. The idea of this approach is to still split the time domain into smaller subintervals, but then to validated everything at once, via a single fixed point operator coupling all pieces together. This solution can be computationally more expensive, but has the advantage of being directly adaptable to solving boundary value problems in time, which is not as straightforward for the time stepping approach. This time domain decomposition strategy already proved very powerful in the context of validated integration of ODEs~\cite{BerShe21}.

In this work, we mainly focus on the domain decomposition approach, which will be presented in detail in Section~\ref{sec:domaindecomposition}, but time stepping is also covered in Appendix~\ref{app:timestepping}.
\end{remark}

\begin{remark}
\label{rem:approx}
Up to now, while reviewing the literature concerning rigorous integration of PDEs, we stayed at a rather general level of exposition, and in particular have not yet mentioned discretization choices.  We purposely postponed this discussion until now, in order to first better emphasize the conceptual similarities between these different techniques, but of course the way the approximate solutions are represented (and the often related choices of function spaces that are made) are crucial when actually having to derive the precise estimates needed for the validation.

Regarding the spatial variable, Fourier series and associated sequences spaces are used in~\cite{MR2728184}, whereas results using finite elements and Sobolev spaces can be found in~\cite{HasKinNak20,NakPluWat19}. Regarding the time variable, most of the existing works make use of approximations based on low order 1D finite elements, the notable exception being~\cite{CyrLes22}, where a setup allowing for a spectral approach in both space and time is introduced.
\end{remark}
 
In this work, we make use of a Fourier decomposition in space. Roughly speaking, dissipativity then tells us that the long time behavior of the dynamical system should mostly be governed by only a finite number of Fourier modes, see Section~\ref{sec:functional_setup}. This is why it makes sense to study~\eqref{eq:PDE} via a Fourier expansion in space in the first place, on both paper and in the computer, and it influences the way we define $\L$. Specifically, for the high frequency Fourier modes, we take a time and space average of $DH(\bu)$ as a starting point for defining $\L$, which yields a diagonal operator in Fourier space. We then also incorporate a time independent approximation of $DH(\bu)$ projected onto finitely many (low frequency) Fourier modes. Regarding the time variable, we approximate the solution using polynomial interpolation at Chebyshev nodes, which allows us to easily obtain high order approximations when needed, e.g.\ for accuracy purposes.

\medskip

To summarize, the key features of this work are:
\begin{itemize}
\item A new and deliberate choice of the operator $\L$ used for the Duhamel reformulation~\eqref{eq:tildeT}, specifically: time independent but spatially dependent. This greatly improves the applicability of the validation procedure compared to more naive choices like $\L = (-1)^{R+1} \dfrac{ \partial^{2R}}{ \partial x^{2R}} $ or $\L$ as in~\eqref{eq:L_ArioliKoch}, because we allow $\L$ to depend on the approximate solution $\bu$.
\item The incorporation of the domain decomposition technique, which will facilitate a future extension to boundary value problems in time, and which comes with an interesting twist in the way the approximate inverse $A$ used for the Newton-Kantorovich reformulation~\eqref{eq:NK} is defined (see Section~\ref{sec:setup_DD}).
\item
In the context of domain decomposition, the operator $\L$ can be interpreted as being time dependent in the sense that it is piecewise constant in time.
The Newton-Kantorovich contraction theorem is then applied to a function space with a (time-dependent) weighted norm, where the weight are piecewise constant in time, see~\eqref{e:def_norm_eta}.
\item The usage of Chebyshev interpolation in time, whose order can be adapted to the problem at hand.
\item General estimates and associated code (available at~\cite{integratorcode}) for equations of the form~\eqref{eq:PDE}. That is, the user only has to provide $R\in\NN_{\geq 1}$, $\tau>0$,  polynomial functions $g^{(j)}$ for $j=0,\ldots,2R-1$, and initial data $f$; and the code then automatically generates an approximate solution and tries to validate it using the estimates developed in this paper. This does involve tweaking of computational parameters, which is partly automated and partly manual.
\end{itemize}

\medskip

Below is an example of the type of results that can be obtained with our approach. 

\begin{theorem}
\label{th:SH}
Consider the Swift-Hohenberg equation
	\begin{align}
		\label{eq:SH}
		\begin{cases}
			\dfrac{ \partial u }{ \partial t} = 
			- \left(\dfrac{ \partial^{2} }{ \partial x^{2}} +1\right)^2 u +\alpha u - u^3
			, & (t,x) \in (0,\tend] \times [0,L], \\[2ex]
		\dfrac{ \partial^{j} u}{ \partial x^{j}}u(t,0)=\dfrac{ \partial^{j} u}{ \partial x^{j}}u(t,L),  & t \in [0,\tend], \ j=0,1,2,3, \\[2ex]			
			u (0, x) = f ( x ), & x \in [0,L]
	\end{cases}
	\end{align}
	with $\alpha = 5$, $L=6\pi$, $\tend=3/2$, and $f(x) = 0.4\cos\left(\frac{2\pi x}{L}\right)-0.3\cos\left(\frac{4 \pi x}{L}\right)$. Let $\bu=\bu(t,x)$ be the function represented in Figure~\ref{fig:SH}, and whose precise description in terms of Fourier-Chebyshev coefficients can be downloaded at~\cite{integratorcode}. Then, there exists a smooth solution $u$ of~\eqref{eq:SH} such that
\begin{align}
\label{eq:errorboundexample}
\sup_{t\in [0,\tend]} \sup_{x\in [0,L]} \vert u(t,x)-\bu(t,x) \vert \leq 
4\times 10^{-8}.
\end{align}
\end{theorem}
\begin{figure}[h!]
\centering
\includegraphics[width=0.49\linewidth]{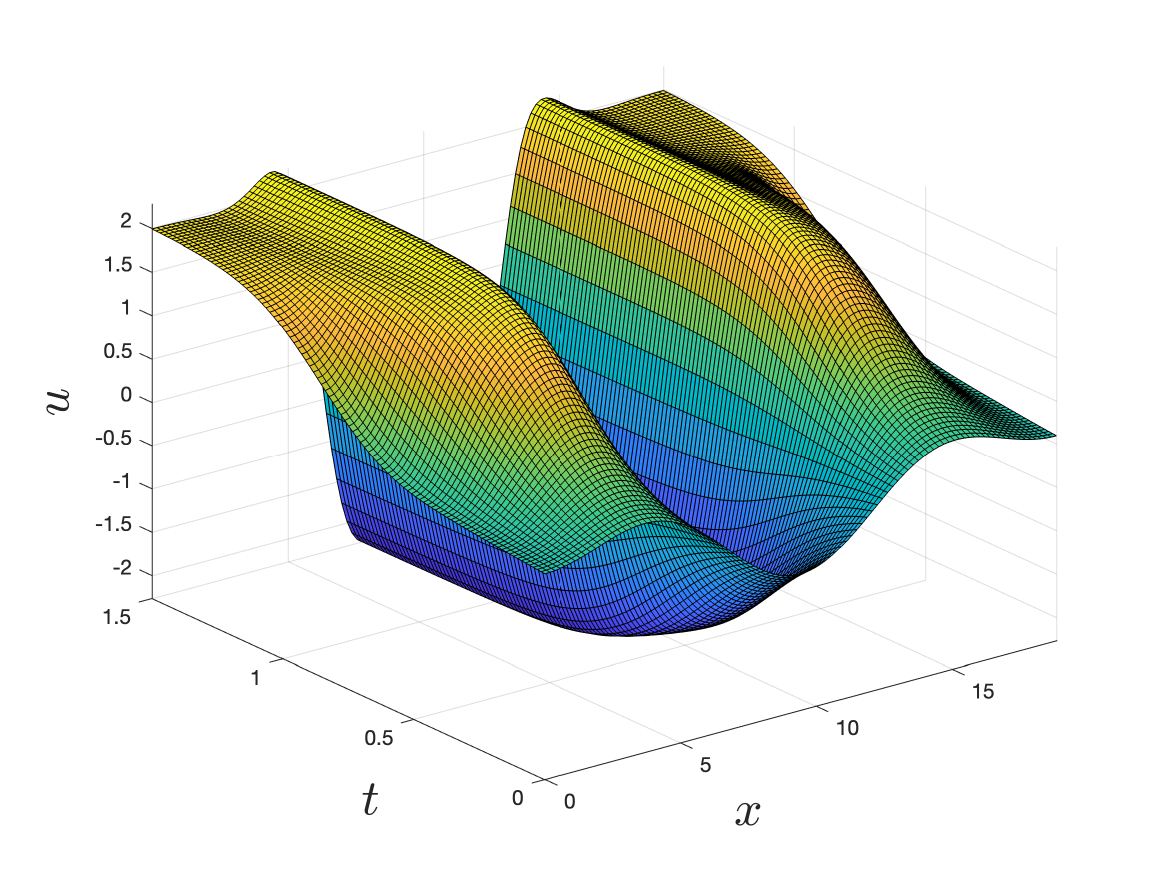}
\hfill
\includegraphics[width=0.49\linewidth]{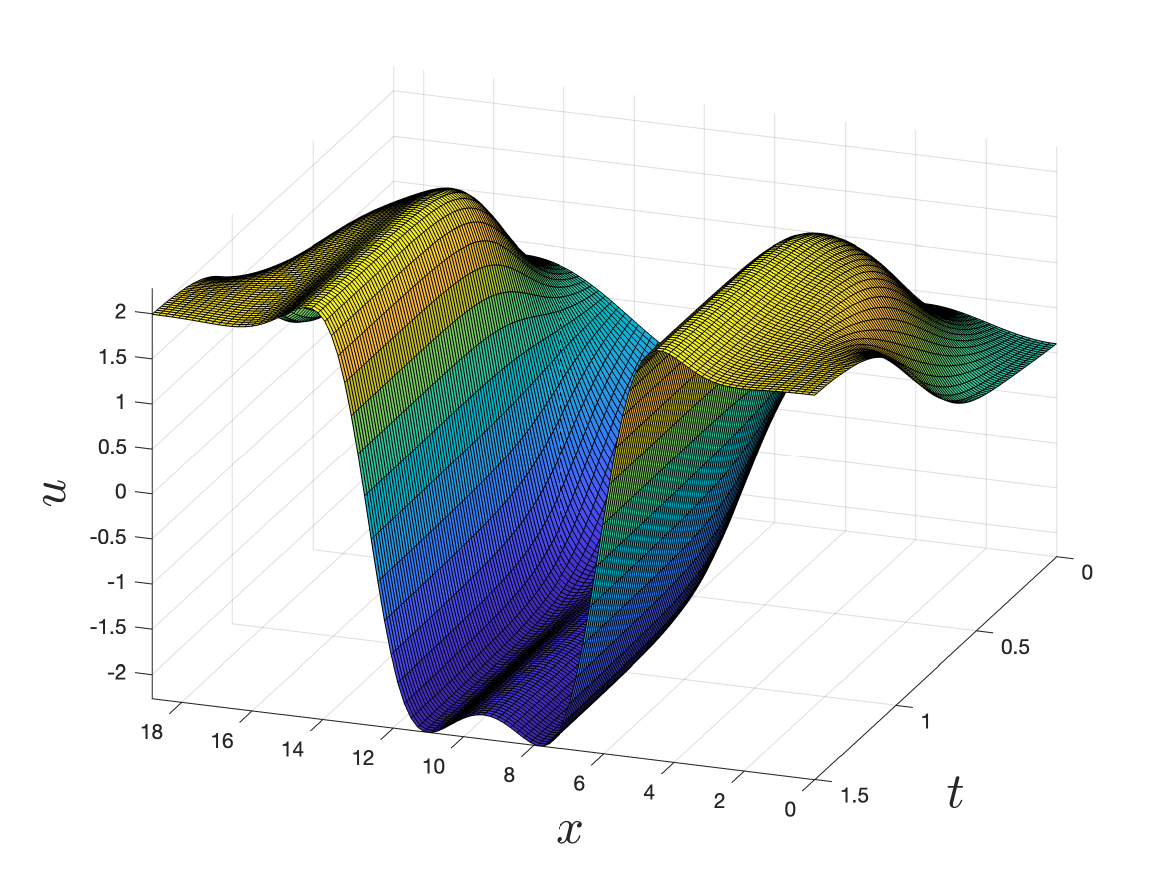}
\caption{The approximate solution $\bu$ of~\eqref{eq:SH}, which has been validated in Theorem~\ref{th:SH}, depicted twice with different views.}
\label{fig:SH}
\end{figure}
\begin{remark}
One of the key ingredients of our proof of Theorem~\ref{th:Fisher} is the selection of a proper function space and associated norm. 
The norm used in the proof is \emph{stronger} than the $C^0$-norm exhibited in the error bound~\eqref{eq:errorboundexample}, where we aimed for a simple statement rather than a sharp estimate. In particular, the solution $u$ obtained in the above theorem is in fact analytic in the space variable $x$, and the norm introduced later in the paper reflects that. Furthermore, the error bound obtained from the proof is not uniform in time, but we report the somewhat weaker resulting uniform bound in~\eqref{eq:errorboundexample}.
We refer to  Section~\ref{sec:preliminaries}, Section~\ref{sec:functional_setup} and Section~\ref{sec:domaindecomposition} for details (see also Remark~\ref{rem:Fisher}).  
\end{remark}

\medskip

Finally, let us mention a few possible extensions. The theory presented in this paper can be extended in a straightforward manner to deal with \emph{systems} of parabolic PDEs, to higher dimensional rectangular spatial domains, to any boundary conditions allowing for smooth Fourier series representation, and to complex valued solutions. Non-constant coefficients in the lower order terms (that is, everywhere except in the $\dfrac{ \partial u}{ \partial t} - (-1)^{R+1} \dfrac{ \partial^{2R} u}{ \partial x^{2R}}$ part) could also be handled. Regarding nonlinearities, we made some choices in~\eqref{eq:PDE} which are not fully general, for the sake of unifying the presentation. For instance, an equation like 
\begin{align}
\label{eq:KSnotnice}
\dfrac{ \partial u }{ \partial t} = 
			- \dfrac{ \partial ^{4} u }{ \partial x^{4} } - \dfrac{ \partial ^{2} u }{ \partial x^{2} }  - \left(\dfrac{ \partial  u }{ \partial x }\right)^2,
\end{align}
which is one version of the Kuramoto-Sivashinsky equation, cannot be written the form~\eqref{eq:PDE}, at least not without changing variables (in this case, the equation satisfied by $ \partial_x u$ can in fact be written in the form~\eqref{eq:PDE}). If one wanted to deal directly with~\eqref{eq:KSnotnice}, or more generally with equations of the form 
\begin{align*}
\dfrac{ \partial u}{ \partial t} = 
		(-1)^{R+1} \dfrac{ \partial^{2R} u}{ \partial x^{2R}} + g\left (u, \dfrac{ \partial u }{ \partial x }, \ldots, \dfrac{ \partial^{2R-1} u }{ \partial x^{2R-1} } \right),
\end{align*}
where $g$ is a (multivariate) polynomial, some of the estimates derived in this paper would have to be adapted, but in principle such equations could also be handled. The extension from initial to boundary value problems is a topic of ongoing research.

\medskip

The paper is organized as follows. In Section~\ref{sec:preliminaries}, we give some definitions and introduce notation related to sequences spaces and Chebyshev interpolation which will be useful throughout the paper, and state the fixed point Theorem~\ref{thm:NewtonKantorovich} that is at the heart of our validation procedure. In Section~\ref{sec:functional_setup}, we give a precise definition of the crucial operator $\L$ that we use in this work, and describe the reformulation of~\eqref{eq:PDE} leading to the fixed point problem to which we want to apply Theorem~\ref{thm:NewtonKantorovich}.
Here we initially restrict attention to a \emph{single} time domain, postponing the domain decomposition to Section~\ref{sec:domaindecomposition}. In Section~\ref{sec:general_estimates}, we collect several technical lemmas, which are then used in Section~\ref{sec:bounds}, where we derive all the estimates required to apply Theorem~\ref{thm:NewtonKantorovich} to the reformulation introduced in Section~\ref{sec:functional_setup}. In Section~\ref{sec:domaindecomposition}, we describe how to modify the setup and the already obtained bounds in order to incorporate the domain decomposition approach in time. Finally, in Section~\ref{sec:examples} we present several examples for the Fisher-KPP equation, the Swift-Hohenberg equation, the Otha-Kawasaki equation and the Kuramoto-Sivashinsky equation, which illustrate the advantages provided by our choice of $\L$ and by the domain decomposition. Appendix~\ref{app:timestepping} contains details about time stepping, Appendix~\ref{app:interp_err} describes new sharp interpolation error estimates, Appendix~\ref{app:quadrature} specifies how we rigorously enclose some integrals, and Appendix~\ref{app:chi} provides an algorithm to explicitly compute some constants introduced in Section~\ref{sec:general_estimates}.

All the computations presented in this paper have been implemented in \textsc{Matlab}, using the phenomenal \textsc{Intlab} package 
\cite{Intlab} for interval arithmetic. The computer-assisted parts of the proofs can be reproduced using the code available at~\cite{integratorcode}.

\section{Preliminaries}
\label{sec:preliminaries}
In this section we introduce notation and provide the necessary background for the tools
used in this paper. 

\subsection{Sequence spaces}
The functional analytic reformulation of (\ref{eq:PDE}) in terms of the Fourier coefficients is posed on a space of continuous functions from $[-1,1]$ into a space of geometrically decaying sequences
\begin{align*}
	\ell^{1}_{\nu} \bydef  \biggl\{ a \in \CC^{\ZZ} : 
				\sum_{n\in\ZZ} \left \vert a_{n} \right \vert \nu^{\vert n\vert} < \infty \biggr\} ,
\end{align*}
endowed with the norm 
$\left \Vert a \right \Vert_{\ell^1_\nu} = 
\sum_{n\in\ZZ}  \left \vert a_{n} \right \vert \nu^{\vert n\vert}$,
where $\nu\geq 1$ is some decay rate to be chosen later. We define 
$\mathcal{X}_{\nu} \bydef C \left( \left[-1,1\right], \ell^{1}_{\nu} \right)$ to be the space of continuous functions from $[-1,1]$ into 
$\ell^{1}_{\nu}$. 

Recall that the Fourier coefficients of the product of two Fourier expansions is given by
the discrete convolution: when $u$ and $v$ are $2\pi$-periodic functions given by
$u(x) = \sum_{n\in\ZZ} \hat u_{n} e^{inx}$
and \
$v(x) = \sum_{n\in\ZZ} \hat v_{n} e^{inx}$,
then
\begin{align*}
	\left( u v \right) (x) = 
	 \sum_{n\in\ZZ} \left(\hat u \ast \hat v\right)_{n} e^{inx},  
\qquad\text{where}\qquad
	\left( \hat u \ast \hat v \right)_{n} \bydef  \sum_{m \in \ZZ} \hat u_m \hat v_{n-m},
	\quad n \in \ZZ,
\end{align*}
and these formal computations are justified as soon as $u$ and $v$ are smooth enough, say Lipschitz-continuous, which will be the case in this work.

\begin{remark}
In order to simplify the notation, in the rest of this paper we will use the same symbol to denote a function and its (discrete) Fourier transform, that is we write
\begin{equation*}
u(x) = \sum_{n\in\ZZ} u_{n} e^{inx}.
\end{equation*}
It should be clear from the context whether $u$ denotes the function $x\mapsto u(x)$ or the sequence of Fourier coefficients $u=\left(u_n\right)_{n\in\ZZ}$.
\end{remark}

We also recall that the discrete convolution gives $\ell^{1}_{\nu}$ a Banach algebra structure.
\begin{lemma}
	\label{prop:BanachAlgebra}
For all $u,v\in\ell^1_\nu$, $u\ast v \in\ell^1_\nu$ and
$\left \Vert u\ast v \right \Vert_{\ell^1_\nu} \leq \left \Vert u \right \Vert_{\ell^1_\nu} \left \Vert v \right \Vert_{\ell^1_\nu}$.
\end{lemma}

\subsection{Chebyshev interpolation}
\label{sec:cheb_interp}
In this section we recall the basics of Chebyshev interpolation. The reader is referred to 
\cite{Cheney, Rivlin, ApproximationTheory} for a comprehensive introduction into the theory of Chebyshev approximation.  

\begin{definition}[Chebyshev points]
	\label{def:chebpoints}
	Let $K \in \NN_{\geq 1}$. The $K$-th order Chebyshev points $\left( t^{K}_{k} \right)_{k=0}^{K}$ 
	are defined by $t^{K}_{k} \bydef  \cos (\frac{\pi k}{K}) $. 
\end{definition}
\begin{remark}
	We shall omit the superscript $K$ from the notation whenever it can be easily inferred
	from the context. We have ordered the Chebyshev points from $1$ to $-1$, i.e., 
	$t_{0} = 1$ and $t_{K} = -1$. In the literature the points $\left(t_{k} \right)_{k=0}^{K}$ are often referred to as the Chebyshev points of the second kind. 
\end{remark}

We will refer to the $K$-th order polynomial which interpolates a continuous function $f: [-1,1] \rightarrow \CC$
at the Chebyshev points $\left( t_{k} \right)_{k=0}^{K}$ as the $K$-th order Chebyshev interpolant of $f$. 
Furthermore, we shall denote the operator which sends a continuous function $f$ to its $K$-th order Chebyshev
interpolant by $P_{K} : C([-1,1],\CC) \rightarrow C([-1,1],\CC)$. 
We note that $f-P_K(f)$ vanishes at $t=\pm 1$ for any $K$ and any $f$, because we have chosen a family of Chebyshev nodes containing the endpoints of the interval $[-1,1]$. We denote the space $C([-1,1],\CC)$ with the $\sup$-norm by $C^0$ throughout.
Our approximate solutions and our analysis makes use of Chebyshev interpolation in time, and we therefore need to control the associated interpolation errors.
\begin{definition}[$C^l$ interpolation error]
\label{def:interp_cste}
For $K\in\NN_{\geq 1}$ and $l\in\{0,1,\ldots,K\}$ we denote by $\sigma_{K,l}$ a real constant such that, for all $f\in C^{l+1}([-1,1],\CC)$,
\begin{align*}
		\left \Vert f - P_{K}(f)  \right \Vert_{C^{0}} \leq 
		\sigma_{K,l} \left \Vert f^{(l+1)} \right \Vert_{C^{0}}.
	\end{align*}
\end{definition}
In practice, we need an explicit value for this constant. One can easily find explicit values in the literature, at least for real-valued functions, but not necessarily optimal ones. In Appendix~\ref{app:interp_err}, we recall some known results about this constant, and provide almost optimal values for $\sigma_{K,0}$ and $K$ small, which is a crucial case for this work.

Another case which we will encounter frequently is when the function to be interpolated is analytic in an open set of the complex plane containing the segment $[-1,1]$. In such a situation, the following error estimate applies.

\begin{definition}
For $\rho\geq 1$, we denote by $\mathcal{E}_{\rho} \subset \CC$ the open ellipse with foci $\pm 1$ such that the length of the semi-major and of the semi-minor axes sum up to $\rho$, or equivalently, the set of all complex numbers $z$ such that $\vert z-1\vert + \vert z+1\vert < \rho+\rho^{-1}$. The set $\mathcal{E}_\rho$ is sometimes referred to as a \emph{Bernstein ellipse}. Notice that $\rho=1$ is a degenerate case: $\mathcal{E}_1=(-1,1)$.
\end{definition}
\begin{theorem}
	\label{thm:analytic_intp_error}
	Let $\rho>1$. Suppose $f: \left[ -1, 1\right] \rightarrow \CC$ can be analytically extended to $\mathcal{E}_{\rho}$, and is bounded on $\mathcal{E}_{\rho}$. Then 
	\begin{align*}
		\left \Vert f - P_{K}(f) \right \Vert_{C^{0}} \leq \frac{4\rho^{-K}}{\rho-1} \sup_{z \in \mathcal{E}_{\rho}} \left \vert f (z) \right \vert,
	\end{align*}
	for any $K \in \NN$. 
\end{theorem}
\noindent The proof can be found in~\cite[Theorem $8.2$]{ApproximationTheory}.

In practice, we compute and represent interpolation polynomials using the Chebyshev basis $\left( T_{k} \right)_{k \in \NN_{0}}$.
\begin{definition}
	\label{def:Tk}
	The Chebyshev polynomials $T_{k} : \left[-1,1\right] \rightarrow \RR$ are defined by 
	the relation $T_{k} \left( \cos \theta \right) = \cos \left( k \theta \right)$, where $k \in \NN_{0}$ and
	$\theta \in \left[0,\pi\right]$. 
\end{definition}
\begin{remark}
	In the literature the polynomials $\left( T_{k} \right)_{k \in \NN_{0}}$ are often referred to as the Chebyshev polynomials of the first kind. The $K$-th order Chebyshev points $\left( t^{K}_{k} \right)_{k=0}^{K}$ 	are the points in $\left[-1,1\right]$ at which $T_{K}$ attains its extrema. Furthermore, note that $T_{k} \left( t^{K}_{j} \right) = \cos ( \frac{ \pi jk}{K} )$ for $k,j \in \NN_{0}$. 
\end{remark}
The polynomials $\left( T_{k} \right)_{k=0}^{K}$
constitute a basis for the space of $K$-th order polynomials $\mathbb{P}_{K}$. Hence, any polynomial $P$ of degree at most $K$, and in particular any interpolation polynomial $P_K(f)$, can be uniquely written as
\begin{align}
\label{eq:cheb_exp}
P = P_0 + 2 \sum_{k=1}^K P_k T_k,
\end{align}
for some Chebyshev coefficients  $P_0,\ldots,P_K \in\CC$.

For any given polynomial written in the Chebyshev basis, we can easily control both its $C^0$ norm and its supremum on a Bernstein ellipse $\mathcal{E}_\rho$, in terms of its Chebyshev coefficients.
\begin{lemma}
\label{lem:upperbounds_norms}
Let $K \in \NN_{\geq 0}$ and $P = P_0 + 2 \sum_{k=1}^K P_k T_k$, then
	\begin{align}
		\label{eq:C0_norm}
		\left \Vert P \right \Vert_{C^{0}} \leq 
		\left \vert P_{0} \right \vert + 2 \sum_{k=1}^{K} \left \vert P_{k} \right \vert,
	\end{align}
	and
	\begin{align}
	\label{eq:Crho_norm}
		\sup_{z\in\mathcal{E}_\rho} \left\vert P(z)\right\vert  \leq 
		\left \vert P_{0} \right \vert + \sum_{k=1}^{K} \left \vert P_{k} \right \vert \left(\rho^k+\rho^{-k}\right). 
		\end{align}
\end{lemma}
\begin{proof}
Any element $z$ of $\mathcal{E}_\rho$ can be written $z=\frac{\omega+\omega^{-1}}{2}$ for some $\omega\in\CC$ such that $1\leq \vert\omega\vert\leq \rho$, and we have the identity
$T_k\bigl( \frac{\omega+\omega^{-1}}{2}\bigr) = \frac{\omega^k+\omega^{-k}}{2}$,
which holds because both sides are analytic and coincide for $\omega$ on the unit circle. Hence
$\sup_{z\in\mathcal{E}_\rho} \left\vert T_k(z)\right\vert = \frac{\rho^k+\rho^{-k}}{2}$,
which yields~\eqref{eq:Crho_norm}.
\end{proof}
\begin{remark}
	Note that \eqref{eq:cheb_exp} is, up to the coordinate transformation $x=\cos\theta$, a Fourier cosine series.
	This is the motivation for using the factor $2$ in front of the coefficients $(P_{k})_{k=1}^{K}$. 
	In particular, with this convention the Chebyshev coefficients of the product of two Chebyshev expansions can be computed directly (i.e., without a rescaling factor) via the discrete convolution. 
\end{remark}

\subsection{A fixed point Theorem}
\label{sec:NewtonKantorovich}

Let $(X^m,\|\cdot\|_{X^m})_{m=1}^M$ be Banach spaces, $X=\Pi_{m=1}^M X^m$ the product space, and $\pi^m:X\to X^m$ the projections onto the components. Let $\rstar = (\rstar^m)_{m=1}^M \in \RR_{>0}^M$ and $\bar{x} \in X$. For any $r\in \RR_{>0}^M$, we define $\bbox(\bar{x},r) = \{ x \in X :  \left\Vert\pi^{m}(x-\bar{x}) \right\Vert_{X^m} \leq r^m \text{ for } 1\leq m\leq M\}$. We consider a map $T \in C^1(\bbox(\bar{x},r),X)$. For $r,\rstar \in \RR_{>0}^M$ we say that $r \leq \rstar$ if $r^m \leq \rstar^m$ for all $1\leq m\leq M$. Finally, we denote partial Fr\'echet derivatives by $D_i$. 

The following statement, based on the Banach fixed point Theorem, provides explicit conditions under which $T$ has a unique fixed point in $\bbox(\bar{x},r)$ (for some explicit $r$). Many similar versions of this theorem have been used in the last decades for computer-assisted proofs, see e.g.~\cite{MR2728184,MR2338393,Gom19,NakPluWat19,Ois95,BerLes15,Yamamoto}. The one we use here, with possibly different radii in each component, originates from~\cite{Ber17}. 

\begin{theorem}
\label{thm:NewtonKantorovich}
Assume that $Y^m \geq 0$, $Z^m_i \geq 0$, $W^m_{ij} \geq 0$ for $1 \leq  i,j,m \leq M$ satisfy
\begin{alignat}{2}
	\left\Vert\pi^m (T(\bar{x})-\bar{x}) \right\Vert_{X^m} &\leq Y^m, \label{e:def_Y}\\
	\left\Vert\pi^m D_i T(\bar{x}) \right\Vert_{B(X^i,X^m)} &\leq Z^m_i, \label{e:def_Z}\\
	\left\Vert\pi^m (D_i T(x)- D_i T(\bar{x})) \right\Vert_{B(X^i,X^m)} &\leq \sum_{j=1}^M W^m_{ij} \left\Vert \pi^j(x-\bar{x}) \right\Vert_{X^j}
	&\quad&\text{for all } x\in \bbox(\bar{x},\rstar). \label{e:def_W}
\end{alignat}
If $r,\eta \in \RR_{>0}^M$ with $r \leq \rstar$ satisfy
\begin{alignat}{1}
	Y^m + \sum_{i=1}^M Z^m_{i} r^i  
	+ \frac{1}{2}\sum_{i,j=1}^M W^m_{ij} r^i r^j  & \leq r^m 
	\label{e:inequalities1} \\
	 \sum_{i=1}^M Z^m_{i} \eta^i
	 + \sum_{i,j=1}^M W^m_{ij} \eta^i r^j  & < \eta^m
	 \label{e:inequalities2} 
\end{alignat}
for $1\leq m\leq M$,
then $T$ has a unique fixed point in $\bbox(\bar{x},r)$.
\end{theorem}
\begin{proof}
Writing
\begin{align*}
T(x) - \bar{x} = T(\bar{x})-\bar{x} + DT(\bar{x})(x-\bar{x}) + \int_0^1 \left[DT\left(\bar{x} + s (x-\bar{x})\right) - DT\left(\bar{x}\right) \right](x-\bar{x})\, \d s,
\end{align*}
the inequalities~\eqref{e:inequalities1} imply that $T$ maps $\bbox(\bar{x},r)$ into itself. Then, writing
\begin{align*}
DT(x) = DT(\bar{x}) + \left[DT\left(x\right) - DT\left(\bar{x}\right) \right],
\end{align*}
and considering the weighted maximum norm
\begin{equation}
\label{e:def_norm_eta}
  \left\Vert x \right\Vert_X = \max_{1\leq m \leq M} \frac{\left\Vert\pi^m x \right\Vert_{X^m}}{\eta^m}, 
\end{equation}
we get, for all $x$ in $\bbox(\bar{x},\rstar)$,
\begin{align*}
\left\Vert D T(x) \right\Vert_{B(X,X)} \leq \max_{1\leq m\leq M} \frac{1}{\eta^m} \sum_{i=1}^M \biggl( Z^m_{i}  + \sum_{j=1}^M W^m_{ij}  r^j \biggr) \eta^i, 
\end{align*}
and \eqref{e:inequalities2} yields that $T$ is a contraction on 
$\bbox(\bar{x},r) \subset X$ (which is a closed set for $\left\Vert \cdot \right\Vert_X$).
\end{proof}

\begin{remark}
If $M=1$, we recover a more \emph{classical} form of this statement, since $\eta$ can be factored out in~\eqref{e:inequalities2} and therefore plays no role. In that case, $\left\Vert \cdot \right\Vert_X = \left\Vert \cdot \right\Vert_{X^1}$, and the set $\bbox(\bar{x},r)$ is simply a ball for this norm.

There are various ways one could get similar but arguably simpler versions of Theorem~\ref{thm:NewtonKantorovich}. For example, by imposing $\eta = r$, i.e. replacing the conditions~\eqref{e:inequalities1} and~\eqref{e:inequalities2} by
\begin{alignat}{1}
	Y^m + \sum_{i=1}^M Z^m_{i} r^i  
	+ \frac{1}{2}\sum_{i,j=1}^M W^m_{ij} r^i r^j  & \leq r^m  ,
	\label{e:inequalities1bis} \\
	 \sum_{i=1}^M Z^m_{i} r^i
	 + \sum_{i,j=1}^M W^m_{ij} r^i r^j  & < r^m ,
	 \label{e:inequalities2bis} 
\end{alignat}
for $1\leq m\leq M$, or by simply asking that
\begin{align}
\label{e:inequalitiesno12} 
Y^m + \sum_{i=1}^M Z^m_{i} r^i  
	+ \sum_{i,j=1}^M W^m_{ij} r^i r^j  & \leq r^m ,
\end{align}
for $1\leq m\leq M$, which then implies~\eqref{e:inequalities1bis} and~\eqref{e:inequalities2bis}. The downside of these simplifications is that they lead to conditions on the constants $Y^m$, $Z^m_i$ and $W^m_{i,j}$ that are more stringent than~\eqref{e:inequalities1}-\eqref{e:inequalities2}. 

For many computer-assisted proofs, most of the effort goes into actually getting as sharp as possible constants $Y^m$, $Z^m_i$ and $W^m_{i,j}$ satisfying~\eqref{e:def_Y}-\eqref{e:def_W}, and this work is no exception. Once these constants are obtained, we feel that it is somewhat wasteful to use simpler but more stringent forms of Theorem~\ref{thm:NewtonKantorovich}. When the proof is not ``tight'', it is likely that~\eqref{e:inequalities1}-\eqref{e:inequalities2}, \eqref{e:inequalities1bis}-\eqref{e:inequalities2bis} and~\eqref{e:inequalitiesno12} can all be satisfied. However, for the examples presented in Theorem~\ref{th:Fisher}, Theorem~\ref{th:SH_less_precise} and Theorem~\ref{th:OK}, for which we did successfully verify~\eqref{e:inequalities1}-\eqref{e:inequalities2}, there is no $r$ such that either~\eqref{e:inequalitiesno12} or even~\eqref{e:inequalities1bis}-\eqref{e:inequalities2bis} holds, which means the proof would have failed if we had tried to use these simplified version of Theorem~\ref{thm:NewtonKantorovich}.
\end{remark}

In the sequel, we use this theorem with $M=1$ in the case where the time interval $[0,2\tau]$ is not split, and with $M>1$ when using domain decomposition. In Section~\ref{sec:functional_setup}, we introduce in full detail the setup in which we apply this theorem to validate solutions of~\eqref{eq:PDE}. The remainder of the paper is then devoted to the derivation of (computable) bounds $Y^m$, $Z^m_i$ and $W^m_{ij}$ satisfying~\eqref{e:def_Y},~\eqref{e:def_Z} and~\eqref{e:def_W} for this problem.

\section{Functional analytic setup}
\label{sec:functional_setup}

In this section we construct a fixed point map whose fixed points correspond to solutions of \eqref{eq:PDE}. Here we consider a \emph{single} time domain. The generalization to domain decomposition is presented in Section~\ref{sec:domaindecomposition}.
First, we 
recast \eqref{eq:PDE} into an infinite set of coupled ODEs on the sequence space $\ell^{1}_{\nu}$ by using a Fourier transformation
in the spatial variable. We then reformulate the system of ODEs as an equivalent zero finding problem on $\mathcal{X}_{\nu}$ by using the variation of 
constants formula. Next, we perform a finite dimensional reduction 
by approximating a finite number of (time varying) Fourier modes with the aid of Chebyshev interpolation. This reduction is
used to set up a Newton-like map $T$ based at an approximate (numerically computed) zero, which is the one on which Theorem~\ref{thm:NewtonKantorovich} will be applied.

\subsection{An equivalent zero-finding problem}
\label{sec:zeroproblem}
In this section we set up a zero finding problem for \eqref{eq:PDE} by using the variation of constants formula. 
First, rewrite \eqref{eq:PDE} as 
\begin{align}
\label{eq:PDE_with_L}
	\dfrac{ \partial u}{ \partial t} -\L u  &=  (-1)^{R+1} \dfrac{ \partial^{2R} u}{ \partial x^{2R}} + \displaystyle\sum_{j=0}^{2R-1} \dfrac{\partial^{j} g^{(j)} (u)}{\partial x^j} -\L u,
\end{align}
where $\L$ is the linear operator which is going to be the generator of the semigroup used in the Duhamel formula. The choice of $\L$ is actually crucial in our approach: we want $\L$ to make the linearization of the right-hand side of~\eqref{eq:PDE_with_L} around the approximation solution as a small as possible, but we also want to keep $\L$ relatively simple, so that we can easily get explicit estimates on the associated semigroup. The precise definition of $\L$ is given in Section~\ref{sec:defL}.

Next, let $\left( u_{n} \left( t \right) \right)_{n \in \ZZ}$ and $\left( f_{n} \right)_{n \in \ZZ}$ denote the Fourier coefficients of $u\left(t, \cdot \right)$ and $f$, respectively, that is
\begin{align*}
u(t,x) = \sum_{n=-\infty}^{\infty} u_{n}(t) e^{inx}, \quad
f(x) = \sum_{n=-\infty}^{\infty} f_{n} e^{inx},
\end{align*}
for all $t \in [0,2\tau]$ and $x \in \left[0,2\pi\right]$. Slightly abusing the notations, we also use $( g^{(j)}_{n}(u) \left( t \right) )_{n \in \ZZ}$ to denote the Fourier coefficients of $g^{(j)}(u)\left(t, \cdot \right)$, i.e.
\begin{align*}
	g^{(j)} (u)(t,x) = 
	\sum_{n=-\infty}^{\infty} g^{(j)}_{n}(u)(t) e^{inx},
\end{align*}
for all $j=0,\ldots,2R-1$. 

Since the approximation of the solutions will be made using Chebyshev interpolation in time, we rescale the time domain $[0,2\tau]$ to $[-1,1]$, on which the theory of Chebyshev approximations is developed (see Section \ref{sec:cheb_interp}). 
Substitution of the above Fourier expansions into~\eqref{eq:PDE_with_L} yields an infinite dimensional system of ODEs on $[-1,1]$ for 
the Fourier coefficients $\left(u_{n}\right)_{n \in \ZZ}$:
\begin{align}	
	\label{eq:semiflow}
	\begin{cases}
		\dfrac{ d u }{ d t }(t) - \tau \L u(t)  = \tau \biggl((-1)^{R+1} \D^{2R} u(t) + \displaystyle\sum_{j=0}^{2R-1} \D^{j} g^{(j)} (u(t)) -\L u(t) \biggr), & t \in [-1,1], \\[2ex]
		u \left( -1 \right) = f, 
	\end{cases}
\end{align}
where $\D$ is the Fourier transform of $\dfrac{\partial}{\partial x}$ :
\begin{equation*}
\left(\D u\right)_n \bydef  in u_n,
\end{equation*}
and where we abuse notation and denote the Fourier transform of $\L$ still by $\L$.

Finally, integration of \eqref{eq:semiflow} with the aid of variation of constants yields the following map.
\begin{definition}
	\label{Fisher:zeroMap}
	The zero finding map $F: \mathcal{X}_{\nu} \rightarrow \mathcal{X}_{\nu}$ for \eqref{eq:PDE} is defined by 
	\begin{align*}
		 F \left(u\right)\left(t\right) & \bydef 
		e^{ \tau(t+1) \L}f + \tau \int_{-1}^{t} e^{ \tau(t-s)\L } \gamma(u(s))
		 \mbox{d} s - u ( t ),
	\end{align*}
	where, for any $v\in\ell^1_\nu$,
	\begin{equation}\label{eq:defgamma}
		\gamma(v) \bydef (-1)^{R+1} \D^{2R} v + \displaystyle\sum_{j=0}^{2R-1} \D^{j} g^{(j)} (v) -\L v .
	\end{equation}	
\end{definition}
\begin{remark}
While we have not defined $\L$ precisely yet (this will be done in Section~\ref{sec:defL}), we are assuming that $\L$ generates a $C^0$ semi-group, with smoothing properties so that $F$ is well defined, even if 
$\gamma$ does not necessarily map $\ell^{1}_{\nu}$ to itself.
\end{remark}

\subsection{Finite dimensional reduction}
\label{sec:finitereduction}
In this section we introduce a finite dimensional reduction of $F$. 
To accomplish this we will need to truncate the phase space 
$\mathcal{X}_{\nu}$ and to discretize time. 

\begin{definition}[Truncation of phase space]
	Let $N \in \NN$ be a truncation parameter. 
	The projection $\Pi_{N} : \mathcal{X}_{\nu} \rightarrow \mathcal{X}_{\nu}$ is defined by
	\begin{align*}
		\left( \Pi_{N} ( u ) \right)_{n} \bydef  
		\begin{cases}
			u_{n}, &  0 \leq \vert n\vert \leq N, \\
			0, & \vert n\vert > N. 
		\end{cases}
	\end{align*} 
	Furthermore, we set $\Pi_{\infty} \bydef  I - \Pi_{N}$, where $I$ is the identity on $\mathcal{X}_{\nu}$. We also introduce the subspaces $\X^N_\nu=\Pi_N\left(\X_\nu\right)$ and $\X^\infty_\nu=\Pi_\infty\left(\X_\nu\right)$, for which we have
\begin{equation*}
\X_\nu = \X^N_\nu \oplus \X^\infty_\nu.
\end{equation*}
\end{definition}
\begin{remark}
	Henceforth we shall identify $\Pi_{N} \left( u \right)$ with the vector of functions 
	\begin{align*}
		\begin{bmatrix}
			u_{-N} \\
			\vdots \\
			u_{N}
		\end{bmatrix} \in C \left( \left[-1,1\right], \CC^{2N+1} \right).
	\end{align*}
\end{remark}

\begin{definition}[Time discretization]
\label{def:PimN}
Let $K \in \NN$. 
The Fourier-Chebyshev projection $\Pi_{KN} : \mathcal{X}_{\nu} \rightarrow C \left( \left[-1,1\right], \CC^{2N+1} \right)$ is defined by 
\begin{align*}
	\Pi_{KN} ( u ) \bydef  
	\begin{bmatrix}
		P_{K} ( u_{-N} ) \\
		\vdots \\
		P_{K} ( u_{N} )
	\end{bmatrix},
\end{align*}
where $P_{K} : C\left([-1,1],\CC\right) \rightarrow C\left([-1,1],\CC\right)$ is the operator which sends a continuous function to its Chebyshev interpolant
(see Section \ref{sec:cheb_interp}). We sometimes abuse notation by applying $P_K$ to a vector of functions, meaning we apply $P_K$ to each component, e.g. $\Pi_{KN}(u) = P_K \Pi_N(u)$. Furthermore, we set $\Pi_{\infty N} \bydef  \Pi_{N} - \Pi_{KN}$. We note that $\Pi_{\infty N} (u) $ vanishes at $t=\pm 1$.
\end{definition}

The truncation of phase space and discretization of time give rise to the decomposition 
\begin{align}
\label{eq:Xnu_decompo}
	\mathcal{X}_{\nu} = \mathcal{X}^{KN}_{\nu} \oplus \mathcal{X}^{\infty N}_{\nu} \oplus \mathcal{X}^{\infty}_{\nu},
\end{align}
where 
\begin{align*}
	\mathcal{X}^{KN}_{\nu} \bydef  \Pi_{KN} \left( \mathcal{X}_{\nu} \right), \quad 
	\mathcal{X}^{\infty N}_{\nu} \bydef  \Pi_{\infty N} \left( \mathcal{X}_{\nu} \right), \quad 
	\mathcal{X}^{\infty}_{\nu} \bydef  \Pi_{\infty} \left( \mathcal{X}_{\nu} \right).
\end{align*}
We equip $\mathcal{X}^{KN}_{\nu}$ with the norm 
\begin{align}
	\label{eq:norm_Xmn}
	\left \Vert u \right \Vert_{\mathcal{X}^{KN}_{\nu}} &\bydef  
	\Biggl\Vert
		\biggl[ \left \vert u_{n0} \right \vert + 2 \sum_{k=1}^{K} \left \vert u_{nk} \right \vert \biggr]_{n=-N}^{N}
	\Biggr\Vert_{\ell^1_\nu}
	=
	\sum_{\vert n\vert \leq N} \biggl(\left \vert u_{n0} \right \vert + 2 \sum_{k=1}^{K} \left \vert u_{nk} \right \vert\biggr)\nu^{\vert n\vert}.
\end{align}
There are several reasons for choosing this particular
norm over the more obvious supremum norm.
First of all, the norm in \eqref{eq:norm_Xmn} is numerically easy to compute, whereas the
computation of a supremum norm is relatively complicated. Furthermore, with this norm it is easy to compute
operator norms which amounts to computing weighted $\ell^{1}$-norms of finite dimensional matrices. 
This is to be contrasted with the use of a supremum norm, where the analysis of linear operators is much more complicated. 
Lastly, it follows from \eqref{eq:C0_norm} that the norm in \eqref{eq:norm_Xmn} is stronger than the supremum norm, i.e.,  
\begin{align*}
	\sup_{t \in [-1,1]} \left \Vert u(t) \right \Vert_{\ell^1_\nu}
	 \leq
\left \Vert \Vert u \Vert_{C^0} \right \Vert_{\ell^1_\nu}	
	\leq
	\left \Vert u \right \Vert_{\mathcal{X}^{KN}_{\nu}}, \quad \text{for all } u \in \mathcal{X}^{KN}_{\nu},
\end{align*}
thereby allowing one to relate the two norms in a straightforward manner. Here and throughout the paper, when we apply the $C^0$-norm to an element of $\X_\nu$, it should be understood as applying the $C^0$-norm to each Fourier mode.

The subspaces $\mathcal{X}^{\infty N}_{\nu}$ and $\mathcal{X}^{\infty}_{\nu}$ 
are both endowed with the $\left\Vert \cdot\right\Vert_{\ell^1_{\nu}\left(C^0\right)} = \left\Vert \left\Vert \cdot \right\Vert_{C^0} \right\Vert_{\ell^1_\nu}$ norm, i.e., 
\begin{align*}
	 \left\Vert u \right\Vert_{\ell^1_{\nu}\left(C^0\right)} \bydef  \sum_{n\in\ZZ} \left\Vert u_n \right\Vert_{C^0} \nu^{\vert n\vert}.
\end{align*}
Furthermore, the full space $\mathcal{X}_{\nu}$ is equipped with the norm 
\begin{align}
	\label{eq:normXnu}
	\left \Vert u \right \Vert_{\mathcal{X}_{\nu}} \bydef  
	\left \Vert \Pi_{KN} u \right \Vert_{\mathcal{X}^{KN}_{\nu}}+  
		\epsilon_{\infty N}^{-1} \left \Vert \Pi_{\infty N} u \right \Vert_{\mathcal{X}^{\infty N}_{\nu}} +
		\epsilon_{\infty}^{-1} \left \Vert \Pi_{\infty }u \right \Vert_{\mathcal{X}^{\infty}_{\nu}}
\end{align}
where $\epsilon_{\infty N}, \epsilon_{\infty}>0$ are weights whose purpose is to provide some control over the truncation errors in phase space and the interpolation errors in time. The choice of these weights is different for each example in Section~\ref{sec:examples} (the specific values can be found in the code). Over the whole space, the norm $\left\Vert\cdot\right\Vert_{\X_\nu}$ still controls the $C^0$ norm:
\begin{lemma}
\label{lem:norms}
Let $\vartheta_\epsilon =  \max(1,\epsilon_{\infty N}, \epsilon_{\infty})$. Then,
\begin{equation*}
\left\Vert u \right\Vert_{\ell^1_{\nu}\left(C^0\right)} \leq \vartheta_\epsilon
	\left \Vert u \right \Vert_{\mathcal{X}_{\nu}}, \quad \text{for all } u \in \mathcal{X}_{\nu}.
\end{equation*}
\end{lemma}

Finally, we define a finite dimensional reduction of $F$.
\begin{definition}[Finite dimensional reduction of $F$]
	The finite dimensional reduction $F_{KN}: \mathcal{X}^{KN}_{\nu} \rightarrow \mathcal{X}^{KN}_{\nu}$ of $F$
	is defined by 
$		F_{KN} \bydef  \Pi_{KN} \circ F \vert_{ \mathcal{X}^{KN}_{\nu} } $.
\end{definition}

\subsection{A posteriori analysis}
\label{sec:Newtonlike}
In this section we construct a Newton-like map for $F$ by using the finite dimensional 
reduction~$F_{KN}$. To this end, suppose we have computed the following: 
\begin{itemize}
	\item [$(i)$] An approximate zero $\bu$ in $\mathcal{X}^{KN}_{\nu} \simeq\CC^{(2N+1)\left(K+1\right)}$ of $F_{KN}$. 
	\item[$(ii)$] The derivative $DF_{KN} \left( \bu \right)$.
	\item [$(iii)$] An approximate \textit{injective} inverse $A_{KN}$ of $DF_{KN} \left( \bu \right)$.
\end{itemize}
\begin{remark}
	One can check that $A_{KN}$ is injective by verifying that the bound in~\eqref{eq:Z0} is strictly smaller than $1$. In fact, our computer-assisted proof can only be successful when this
	inequality is satisfied, cf.~\eqref{e:inequalities2}. Therefore, if the computer-assisted
	proof is successful, we may a posteriori conclude that $A_{KN}$ is injective without any further ado.  
\end{remark}

We will use the finite dimensional data to construct an approximate inverse of $DF \left( \bu \right)$. 
We anticipate that   
\begin{align*}
	\Pi_{\infty N}DF \left( \bu \right) \approx - \Pi_{\infty N} \qquad\text{and}\qquad
	\Pi_{\infty} DF \left( \bu\right) \approx - \Pi_{\infty}
\end{align*}
in a small neighborhood of $\bu$ provided $K$ and $N$ are sufficiently large. These observations motivate the following 
definitions: 
\begin{definition}[Approximation of $DF \left( \bu \right)$]
\label{def:approx_DF}
	The approximate derivative $\widehat{DF}: \mathcal{X}_{\nu} \rightarrow \mathcal{X}_{\nu}$ of $F$ at $\bu$ is defined by 
	\begin{align*}
		\widehat{DF} \bydef  DF_{KN} \left( \bu \right) \oplus \left( -\Pi_{\infty N} \right) \oplus \left( -\Pi_{\infty}\right). 
	\end{align*}
\end{definition}

\begin{definition}[Approximate inverse of $DF \left( \bu \right)$]
The approximate inverse $A: \mathcal{X}_{\nu} \rightarrow \mathcal{X}_{\nu}$  of $DF \left( \bu \right)$ is defined by 
\begin{align*}
	A \bydef  A_{KN} \oplus \left( -\Pi_{\infty N} \right) \oplus \left( -\Pi_{\infty} \right). 
\end{align*}
\end{definition}

Next, we define a Newton-like operator $T$ for $F$ based at $\bu$:
\begin{definition}[Newton-like operator for $F$]
The Newton-like operator $T: \mathcal{X}_{\nu} \rightarrow \mathcal{X}_{\nu}$ for $F$ based at $\bu$ is defined by 
\begin{align*}
	T \bydef  I - AF.
\end{align*}
\end{definition}

The idea is to seek fixed points of $T$ in a small neighborhood of $\bu$. To be more precise, let $B_{r,\epsilon}(0)$
denote the \emph{closed} ball of radius $r>0$ centered at $0$ in $\mathcal{X}_{\nu}$, i.e.,  
\begin{align}
\label{eq:def_ball}
	B_{r,\epsilon}(0) = 
	\left \{ h \in \mathcal{X}_{\nu} : 
	\left \Vert \Pi_{KN} \left( h \right) \right \Vert_{\mathcal{X}^{KN}_{\nu}} + \epsilon_{\infty N}^{-1}
	\left \Vert \Pi_{\infty N} \left( h \right) \right \Vert_{\mathcal{X}^{\infty N}_{\nu}} + \epsilon_{\infty}^{-1}
	\left \Vert \Pi_{\infty} \left( h \right) \right \Vert_{\mathcal{X}^{\infty}_{\nu}} \leq r
\right \}.
\end{align}
We shall prove the existence of a fixed point $u^*$ of $T$ in $B_{r,\epsilon} \left( \bu \right) = \bu + B_{r,\epsilon}(0)$, where $r>0$ is an unknown radius to be determined, by using Theorem~\ref{thm:NewtonKantorovich}, with $M=1$ and $X = \X_\nu$.

\begin{remark}
This fixed point $u^*$ of $T$ corresponds to a zero of $F$, since $A$ is injective. We have thus found the solution of the initial value problem~\eqref{eq:PDE}.
In view of Lemma~\ref{lem:norms}, the error bound on 
the distance between the solution $u^*$ and its approximation $\bu$ is controlled by 
\[
	\max_{\substack{ t\in[0,2\tau] , x \in [0,2\pi] }} |u^*(t,x)- \bu(t,x)| 
	\leq \|u^*- \bu\|_{\ell^1_{\nu}(C^0)} \leq \vartheta_{\epsilon} r.
\]

More precisely, the solution obtained via the computer assisted argument is a mild solution, continuous in time with values in $\ell^1_\nu$, and it is well known that parabolic problems of the form~\eqref{eq:PDE} have a unique mild solution which is in fact a classical solution~\cite{Lun12,Paz12}. However, the maximal existence time $t_{\text{max}}$ of this unique solution could in principle be smaller than~$2\tau$. Our proof additionally shows the maximal existence time $t_{\text{max}}$ is in fact larger than~$2\tau$, and provides more detailed information on the regularity of the solution (for $\nu >1$), namely that is it analytic in the space variable on a strip of width $\ln \nu$.
\end{remark}

\subsection{Construction of $\L$}
\label{sec:defL}

For reasons already outlined in the introduction, and that will be made more concrete later on, we want $\L$ to be a time \emph{independent}
approximation of 
\begin{align}\label{eq:linearization}
		h \mapsto (-1)^{(R+1)} \mathcal{D}^{2R}h + \sum_{j=0}^{2R-1} 
		\mathcal{D}^{j} \left( \left( g^{(j)} \right)'( \bar u(s) ) \ast h \right), \quad h \in \ell^{1}_{\nu},
\end{align}
where $\left( g^{(j)} \right)'$ simply denotes the derivative of the polynomial function $g^{(j)}$.
Although we choose $\L$ to be uniform in $s$, it does approximate the linear operator~\eqref{eq:linearization}, 
hence we have that, with $\gamma$ defined in \eqref{eq:defgamma},
\begin{equation*}
s\mapsto D\gamma(\bu(s)) \approx 0.
\end{equation*}
To construct $\L$, we start by introducing vectors $\bv^{(j)}\in\Pi^N\left(\ell^1_\nu\right)$, $j=0,\ldots,2R-1$, which in practice should be (constant in time) approximations of $s\mapsto \left(g^{(j)}\right)'(\bu)(s)$. Then, we define $\tL$ on $\X_{\nu} = \Pi_N \X_{\nu} \oplus \Pi_\infty \X_{\nu}$ by
\begin{equation*}
\left\{
\begin{aligned}
&\tL \Pi_N u \bydef   (-1)^{R+1} \D^{2R}\Pi_N u + \displaystyle\sum_{j=0}^{2R-1} \D^{j} \Pi_N\bigl(\bv^{(j)}\ast \Pi_N u\bigr) \\
&\left(\tL \Pi_\infty u\right)_n \bydef  \biggl(-n^{2R} + \displaystyle\sum_{j=0}^{2R-1} (in)^j \bv^{(j)}_0 \biggr) u_n \qquad \text{for all } \vert n\vert > N.
\end{aligned}
\right.
\end{equation*}
Notice that $\tL$ leaves $\X^N_{\nu}$ invariant, and that it does not depend on time, therefore its restriction to $\X^N_{\nu}$ can be represented as a finite dimensional matrix, that we denote by $\tL_N$.

Finally, we consider a diagonal matrix $\Lambda_N= \text{diag}\left(\lambda_{-N},\ldots,\lambda_{N}\right)$ and an invertible matrix $Q$ such that
\begin{equation*}
\tL_N \approx Q \Lambda_N Q^{-1},
\end{equation*}
and define
\begin{equation*}
\L_N \bydef Q \Lambda_N Q^{-1},
\end{equation*}
and $\L$ on $\X_{\nu} = \Pi_N \X_{\nu} \oplus \Pi_\infty \X_{\nu}$ by
\begin{equation*}
\left\{
\begin{aligned}
&\L \Pi_N u \bydef  \L_N\Pi_N u \\
&\left(\L \Pi_\infty u\right)_n \bydef  \biggl(-n^{2R} + \displaystyle\sum_{j=0}^{2R-1} (in)^j \bv^{(j)}_0 \biggr) u_n \quad \text{for all } \vert n\vert > N.
\end{aligned}
\right.
\end{equation*}


For further use, we denote, for all $\vert n\vert >N$,
\begin{equation}
\label{eq:eigenvalues}
\lambda_n \bydef -n^{2R} + \displaystyle\sum_{j=0}^{2R-1} (in)^j \bv^{(j)}_0,
\end{equation}
and 
\begin{equation*}
\R_N \bydef \L_N-\tL_N.
\end{equation*}
Notice that with this notation we can write
\begin{equation}
\label{eq:Dgamma}
\Pi_N D\gamma(\bu)u = \sum_{j=0}^{2R-1} \D^j \Pi_N\left[\left(\left(g^{(j)}\right)'(\bu) - \bv^{(j)}\right) \ast \Pi_N u + \left(g^{(j)}\right)'(\bu) \ast \Pi_\infty u \right] - \R_N \Pi_N u.
\end{equation}

Now that $A$ and $\L$ have been explicitly constructed, we are ready to derive the estimates needed to prove that $T$ has a locally unique fixed point around $\bu$.

\section{General estimates}
\label{sec:general_estimates}

We compile here a list of lemmas which will be used extensively in the computation of the $Y$, $Z$ and $W$ bounds needed for Theorem~\ref{thm:NewtonKantorovich}. These are essentially technical estimates, the reader more interested in an overview of the proof is encouraged to jump directly to Section~\ref{sec:bounds}, and to only refer to this section when needed. 

\subsection{An $\ell^1_\nu$ dual estimate}

\begin{lemma}
\label{lem:Upsilon}
For a given $(2N+1)\times (2N+1)$ matrix $B$ with complex entries, define $\Upsilon(B)\in \CC^{2N+1}$ by
\begin{equation*}
\Upsilon_n(B) = \max_{\vert m\vert \leq N} \frac{\vert B_{n,m}\vert }{\nu^{\vert m\vert}}, \quad \vert n\vert \leq N.
\end{equation*}
Then, for any $v\in \Pi_N\ell^1_\nu \cong \CC^{2N+1}$,
\begin{align*}
\left\vert Bv \right\vert \leq \Upsilon(B) \left\Vert v \right\Vert_{\ell^1_\nu},
\end{align*}
where the absolute values apply component-wise.
\end{lemma}

\subsection{Controlling convolution products}


\begin{lemma}
\label{lem:convo_with_h_split}
Let $a,b,h\in\X_{\nu}$. Define, for all $n$ in $\ZZ$,
\begin{equation*}
\Phi^\epsilon_n(a,b) \bydef  \max \biggl\{ \epsilon_{\infty N} \max_{\vert n-m\vert \leq N} \frac{\left\Vert a_m \right\Vert_{C^0}}{\nu^{\vert n-m\vert}},\ \epsilon_{\infty} \sup_{\vert n-m\vert > N} \frac{\left\Vert b_m \right\Vert_{C^0}}{\nu^{\vert n-m\vert}} \biggr\},
\end{equation*}
\begin{equation*}
\check\Phi^\epsilon_n(a,b) \bydef  \max \biggl\{ \max(1,\epsilon_{\infty N}) \max_{\vert n-m\vert \leq N} \frac{\left\Vert a_m \right\Vert_{C^0}}{\nu^{\vert n-m\vert}},\ \epsilon_{\infty} \sup_{\vert n-m\vert > N} \frac{\left\Vert b_m \right\Vert_{C^0}}{\nu^{\vert n-m\vert}} \biggr\},
\end{equation*}
and
\begin{align*}
\tilde\Phi^\epsilon(a,b) \bydef \max \biggl\{ &\max(1,\epsilon_{\infty N}) \max_{\vert m\vert\leq N} \sum_{\vert n\vert \leq N} \left\Vert a_{n-m} \right\Vert_{C^0} \nu^{\vert n\vert -\vert m\vert},\\ 
&\hspace*{3cm}\epsilon_{\infty}\sup_{\vert m\vert> N} \sum_{\vert n\vert \leq N} \left\Vert b_{n-m} \right\Vert_{C^0} \nu^{\vert n\vert- \vert m\vert} \biggr\}.
\end{align*}
Then, for all $n$ in $\ZZ$,
\begin{align*}
\left\Vert \left(a \ast \Pi_{\infty N} h + b \ast \Pi_{\infty} h \right)_n \right\Vert_{C^0} \leq \Phi^\epsilon_n(a,b) \left\Vert h \right\Vert_{\X_{\nu}},
\end{align*}
\begin{align*}
\left\Vert \left(a \ast \Pi_{N} h + b \ast \Pi_{\infty} h \right)_n \right\Vert_{C^0} \leq \check\Phi^\epsilon_n(a,b) \left\Vert h \right\Vert_{\X_{\nu}},
\end{align*}
and
\begin{align*}
\left\Vert \Pi_N \left[a \ast \Pi_N h + b \ast \Pi_{\infty} h \right] \right\Vert_{\ell^1_\nu(C^0)} \leq \tilde\Phi^\epsilon(a,b) \left\Vert h \right\Vert_{\X_{\nu}}.
\end{align*}
\end{lemma}
\begin{proof}
Let $h^{\infty N} = \Pi_{\infty N} h$ and $h^{\infty} = \Pi_{\infty} h$. For all $s\in[-1,1]$, by using~\eqref{eq:normXnu},
\begin{align*}
\left\vert \left(a \ast h^{\infty N} + b \ast h^{\infty}\right)_n(s) \right\vert & \leq  \sum_{m\in\ZZ} \left( \left\vert a_m(s) h^{\infty N}_{n-m}(s) \right\vert + \left\vert b_m(s) h^{\infty}_{n-m}(s) \right\vert \right) \\
&\leq \sum_{\vert n-m\vert \leq N}  \frac{\left\Vert a_m \right\Vert_{C^0}}{\nu^{\vert n-m\vert}} \left\Vert h^{\infty N}_{n-m} \right\Vert_{C^0} \nu^{\vert n-m\vert} \\
& \quad + \sum_{\vert n-m\vert > N}  \frac{\left\Vert b_m \right\Vert_{C^0}}{\nu^{\vert n-m\vert}} \left\Vert h^{\infty}_{n-m} \right\Vert_{C^0} \nu^{\vert n-m\vert}   \\
&\leq \max_{\vert n-m\vert \leq N} \frac{\left\Vert a_m \right\Vert_{C^0}}{\nu^{\vert n-m\vert}} \left \Vert \Pi_{\infty N} h \right \Vert_{\mathcal{X}^{\infty N}_{\nu}} + \sup_{\vert n-m\vert > N} \frac{\left\Vert b_m \right\Vert_{C^0}}{\nu^{\vert n-m\vert}} \left \Vert \Pi_{\infty }h \right \Vert_{\mathcal{X}^{\infty}_{\nu}} \\
&\leq \Phi^\epsilon_n(a,b) \left\Vert h \right\Vert_{\X_{\nu}}. 
\end{align*}
The second estimate is obtained in the same way, using additionally that 
\begin{align*}
\sum_{\vert m\vert \leq N} \left\Vert h_{m} \right\Vert_{C^0} \nu^{\vert m\vert} &\leq \left \Vert \Pi_{KN} h \right \Vert_{\mathcal{X}^{KN}_{\nu}} + \left \Vert \Pi_{\infty N} h \right \Vert_{\mathcal{X}^{\infty N}_{\nu}} \\
& \leq \max(1,\epsilon_{\infty N} ) \left(\left \Vert \Pi_{KN} h \right \Vert_{\mathcal{X}^{KN}_{\nu}} + \epsilon_{\infty N}^{-1} \left \Vert \Pi_{\infty N} h \right \Vert_{\mathcal{X}^{\infty N}_{\nu}} \right).
\end{align*}
For the final estimate, introducing $h^{N} = \Pi_{N} h$, we proceed in a similar fashion:
\begin{align*}
\left\Vert \Pi_N \left[a \ast h^{N} + b \ast h^{\infty} \right] \right\Vert_{\ell^1_\nu(C^0)} & \leq  \sum_{\vert n\vert \leq N} \sum_{m\in\ZZ} \left( \left\Vert a_{n-m} \right\Vert_{C^0} \left\Vert h^{N}_{m} \right\Vert_{C^0} + \left\Vert b_{n-m} \right\Vert_{C^0} \left\Vert h^{\infty}_{m} \right\Vert_{C^0} \right) \nu^{\vert n\vert} \\
& \leq  \sum_{\vert m\vert\leq N} \left\Vert h_{m} \right\Vert_{C^0} \nu^{\vert m\vert} \sum_{\vert n\vert \leq N} \left\Vert a_{n-m} \right\Vert_{C^0} \nu^{\vert n\vert -\vert m\vert} \\
&\qquad + \sum_{\vert m\vert> N} \left\Vert h_{m} \right\Vert_{C^0} \nu^{\vert m\vert} \sum_{\vert n\vert \leq N} \left\Vert b_{n-m} \right\Vert_{C^0} \nu^{\vert n\vert -\vert m\vert} \\
& \leq  \left( \left\Vert \Pi_{KN} h \right\Vert_{\X^{KN}_\nu} + \epsilon_{\infty N}^{-1}\left\Vert \Pi_{\infty N} h \right\Vert_{\X^{\infty N}_\nu} \right) \\
&\qquad \times \max(1,\epsilon_{\infty N})\max_{\vert m\vert\leq N} \sum_{\vert n\vert \leq N} \left\Vert a_{n-m} \right\Vert_{C^0} \nu^{\vert n\vert -\vert m\vert} \\
&\qquad\qquad + \epsilon_{\infty}^{-1}\left\Vert \Pi_{\infty} h \right\Vert_{\X^{\infty}_\nu}  \, \epsilon_{\infty}\sup_{\vert m\vert> N} \sum_{\vert n\vert \leq N} \left\Vert b_{n-m} \right\Vert_{C^0} \nu^{\vert n\vert -\vert m\vert}. \qedhere
\end{align*}
\end{proof}

\begin{remark}
As soon as $b$ only has a finite number of non-zero Fourier modes, the suprema in $\Phi^\epsilon_n(a,b)$, $\check\Phi^\epsilon_n(a,b)$ and $\tilde\Phi^\epsilon(a,b)$ become maxima that can be computed.
\end{remark}

\begin{lemma}
	\label{lem:Theta}
Let $r\geq 0$ and $p$ be a polynomial, written in the canonical basis
\begin{equation*}
p(x) = \sum_{k=0}^{N_p} p_k x^k.
\end{equation*}
If we denote by $\vert p\vert$ the polynomial
\begin{equation*}
\vert p\vert (x) \bydef \sum_{k=0}^{N_p} \vert p_k\vert x^k,
\end{equation*}
then for all $\Vert v \Vert_{\mathcal{X}_{\nu}} \leq 1$ and  $\Vert h \Vert_{\mathcal{X}_{\nu}} \leq 1$ we have 
\begin{equation*}
\left \Vert  p \left( \bu  +  r v\right) \ast v \ast h \right \Vert_{\ell^1_\nu(C^0)} \leq  \vartheta_\epsilon^2  \, \vert p \vert \left(\left\Vert  \bu \right\Vert_{\ell^1_\nu(C^0)} +  \vartheta_\epsilon r \right).
\end{equation*}
\begin{proof} 
We use the Banach algebra properties of $\ell^1_\nu$ and $\ell^1_\nu(C^0)$ to get
\begin{align*}
\left \Vert  p \left( \bu  +  r v\right) \ast v \ast h \right \Vert_{\ell^1_\nu(C^0)} 
&\leq   \left \Vert  p \left( \bu  +  r v\right) \right \Vert_{\ell^1_\nu(C^0)} 
\Vert v \Vert_{\ell^1_\nu(C^0)} \Vert h\Vert_{\ell^1_\nu(C^0)} 
\\
&\leq  \vartheta_\epsilon^2  \sum_{k=0}^{N_p} \vert p_k\vert \left \Vert  \bigl( \left\Vert\bu\right\Vert_{C^0} +  r \left\Vert v\right\Vert_{C^0}\bigr)^k \right \Vert_{\ell^1_\nu} \\
&\leq  \vartheta_\epsilon^2  \sum_{k=0}^{N_p} \vert p_k\vert \left\Vert \bigl(  \left\Vert\bu\right\Vert_{C^0} +  r \left\Vert v\right  \Vert_{C^0} \bigr) \right \Vert_{\ell^1_\nu}^k \\
&\leq  \vartheta_\epsilon^2  \sum_{k=0}^{N_p} \vert p_k\vert \bigl(   \left\Vert\bu\right\Vert_{\ell^1_\nu(C^0)} +   \vartheta_\epsilon r \bigr)^k. \qedhere
\end{align*}
\end{proof}
\end{lemma}

\subsection{Interpolation errors}

\begin{lemma}
\label{lem:interp_exp}
Let $\lambda\in\CC$, $\tau>0$ and $\rho>1$,
	\begin{align*}
		\left \Vert \left( I - P_{K} \right) \left( t \mapsto e^{ \tau(t+1) \lambda} \right) \right \Vert_{C^{0}} &\leq
		\frac{4 \rho^{-K}}{\rho -1} 
		\exp\left( \tau \Bigl(\Re(\lambda) + \frac{1}{2}\sqrt{\Re(\lambda)^2(\rho+\rho^{-1})^2 + \Im(\lambda)^2 (\rho-\rho^{-1})^2} \Bigr)\right),
	\end{align*}
where $\Re(\lambda)$ and $\Im(\lambda)$ denote the real part and the imaginary part of $\lambda$, respectively.
\end{lemma}
\begin{proof}
From Theorem~\ref{thm:analytic_intp_error} we get that
\begin{align*}
\left \Vert \left( I - P_{K} \right) \left( t \mapsto e^{ \tau(t+1) \lambda} \right) \right \Vert_{C^{0}} &\leq \frac{ 4 \rho^{-K} }{ \rho -1 } \sup_{z \in \mathcal{E}_{\rho}} e^{\tau\Re\left(\lambda(z+1)\right)}.
\end{align*}
Any $z\in\mathcal{E}_{\rho}$ can be written $z=\frac{1}{2}\left(re^{i\theta}+r^{-1}e^{-i\theta}\right)$, for some $(r,\theta)\in[1,\rho]\times[0,2\pi]$, and we then have
\begin{align*}
\Re\left(\lambda\left( z+1 \right)\right) &= \Re(\lambda)\left(1+\frac{1}{2}(r+r^{-1})\cos\theta\right) - \Im(\lambda)\frac{1}{2}(r-r^{-1})\sin\theta \\
&= \Re(\lambda) + \frac{1}{2}\sqrt{\Re(\lambda)^2(r+r^{-1})^2 + \Im(\lambda)^2 (r-r^{-1})^2}\cos(\theta-\theta_{r,\lambda}),
\end{align*}
for some $\theta_{r,\lambda}\in[0,2\pi]$, hence
	\begin{equation*}
		 \sup_{z \in \mathcal{E}_{\rho}} \Re\left(\lambda(z+1)\right) =  \Re(\lambda) + \frac{1}{2}\sqrt{\Re(\lambda)^2(\rho+\rho^{-1})^2 + \Im(\lambda)^2 (\rho-\rho^{-1})^2}. \hfill  \qedhere 
	\end{equation*}
\end{proof}

\begin{remark}
For a given $\lambda$, the quality of the above estimate depends on the choice of $\rho$, therefore in practice we approximately optimize over $\rho$.
\end{remark}

\begin{lemma}
\label{lem:interpolation_error_N}
Let $\varphi_N\in C^0\left([-1,1],\CC^{2N+1}\right)$, seen as an element of $\Pi_N\X_\nu$, $j\in\NN$, $\tau>0$, $\sigma_{K,0}>0$ satisfying Definition~\ref{def:interp_cste}, and consider the $(2N+1)\times (2N+1)$ diagonal matrix 
\begin{equation}
\label{eq:DN}
D_N(t) \bydef \Re(\Lambda_N)^{-1}\left(e^{\tau(1+t)\Re(\Lambda_N)}-I_{2N+1}\right) = \tau \int_{-1}^{t} \left\vert e^{\tau(t-s)\Lambda_N}\right\vert \mbox{d} s,
\end{equation}
where $\Re$ and $\vert\cdot\vert$ applied to matrices must be understood component-wise. Let $\D_N = \Pi_N \D \Pi_N$. Then 
\begin{align*}
&\left\Vert (I-P_K)\left(t\mapsto \tau \int_{-1}^t e^{\tau(t-s)\L_N}\D^j\varphi_N(s) \d s \right) \right\Vert_{\mathcal{X}^{\infty N}_{\nu}}  \\
&\qquad \leq \tau\sigma_{K,0}\left(\left\Vert \vert \D_N^j \vert +\vert Q\vert \vert \Lambda_N \vert D_N(1) \vert Q^{-1}\vert\vert \D_N^j \vert\right\Vert_{B(\ell^1_\nu,\ell^1_\nu)} \right) \left\Vert  \varphi_N \right\Vert_{\ell^1_\nu\left(C^0\right)} .
\end{align*}
\end{lemma}
\begin{proof}
Using Definition~\ref{def:interp_cste}, we get
\begin{align*}
&\left\Vert (I-P_K)\left(t\mapsto \tau \int_{-1}^t e^{\tau(t-s)\L_N}\D^j\varphi_N(s) \d s \right) \right\Vert_{\mathcal{X}^{\infty N}_{\nu}}  \\
&\qquad = \left\Vert \left\Vert  (I-P_K)\left(t\mapsto \tau \int_{-1}^t e^{\tau(t-s)\L_N}\D^j\varphi_N(s) \d s \right) \right\Vert_{C^0} \right\Vert_{\ell^1_\nu} \\
&\qquad \leq \tau\sigma_{K,0}\left\Vert \left\Vert  \frac{d}{dt}\left(t\mapsto \int_{-1}^t e^{\tau(t-s)\L_N}\D^j\varphi_N(s) \d s \right) \right\Vert_{C^0} \right\Vert_{\ell^1_\nu} \\
&\qquad = \tau\sigma_{K,0}\left\Vert \left\Vert \D^j\varphi_N(t) + \tau\L_N \int_{-1}^t e^{\tau(t-s)\L_N}\D^j\varphi_N(s) \d s  \right\Vert_{C^0} \right\Vert_{\ell^1_\nu} \\
&\qquad = \tau\sigma_{K,0}\left\Vert \left\Vert \D^j\varphi_N(t) + \tau Q \Lambda_N \int_{-1}^t e^{\tau(t-s)\Lambda_N} Q^{-1} \D^j\varphi_N(s) \d s  \right\Vert_{C^0} \right\Vert_{\ell^1_\nu} \\
&\qquad \leq \tau\sigma_{K,0}\left\Vert \vert \D_N^j\vert \left\Vert  \varphi_N \right\Vert_{C^0} + \vert Q\vert \vert\Lambda_N\vert D_N(1) \vert Q^{-1}\vert \vert \D_N^j\vert \left\Vert  \varphi_N \right\Vert_{C^0} \right\Vert_{\ell^1_\nu}. \qedhere
\end{align*}
\end{proof}

\begin{lemma}
\label{lem:interpolation_error_Cinfty}
Let $\varphi:[-1,1]\to\CC$ be an infinitely differentiable function, $\lambda\in\CC$, $\tau>0$, $q\in\NN$ and
\begin{equation*}
	\tilde C \bydef \sigma_{K,q} \left(\left \vert \tau\lambda \right \vert^{q+1}
			\frac{ e^{2\tau \Re(\lambda)} - 1}{\Re(\lambda)} \left \Vert \varphi \right \Vert_{C^{0}}  + 
			\tau\sum_{i=0}^{q} \left \vert \tau\lambda \right \vert^{i} \left \Vert \varphi^{(q-i)} \right \Vert_{C^{0}} \right),
	\end{equation*}
	where $\varphi^{(i)}$ is the $i$-th derivative of $\varphi$ and $\sigma_{K,q}$ satisfies Definition~\ref{def:interp_cste}. Then
	\begin{align*}
\left\Vert (I-P_{K}) \left(t \mapsto \tau \int_{-1}^{t} e^{ \tau(t-s)\lambda} \varphi(s) \mbox{d} s\right)\right\Vert_{C^0} \leq \tilde C.
	\end{align*}
Assume further that $\varphi$ can be analytically extended onto $\mathcal{E}_\rho$ for some $\rho>1$, and let
\begin{align*}
\zeta(x) \bydef
\left\{
	\begin{aligned}
	&\dfrac{ e^{x} -1 }{ x}  \qquad & x\neq 0,\\
	&1 \qquad & x = 0,
	\end{aligned}
\right.
\end{align*}
and
\begin{align*}
\check C \bydef \frac{2\tau(\rho+\rho^{-1}+2)}{\rho^K(\rho-1)} \zeta\left(\tau\left(\Re(\lambda) + \frac{1}{2}\sqrt{\Re(\lambda)^2(\rho+\rho^{-1})^2 + \Im(\lambda)^2 (\rho-\rho^{-1})^2}\right)\right) \sup_{z\in\mathcal{E}_\rho} \left\vert \varphi(z)\right\vert.
\end{align*}
Then 
\begin{align*}
\left\Vert (I-P_{K}) \left(t \mapsto \tau \int_{-1}^{t} e^{ \tau(t-s)\lambda} \varphi(s) \mbox{d} s\right)\right\Vert_{C^0} \leq \check C.
	\end{align*}
\end{lemma}
\begin{proof}
The first estimate is a direct consequence of the interpolation error estimate described in Definition~\ref{def:interp_cste}, and of the fact that
\begin{align*}
&\left\Vert \frac{d^{q+1}}{dt^{q+1}} \left(t \mapsto \tau \int_{-1}^{t} e^{ \tau(t-s)\lambda } \varphi(s) \mbox{d} s\right) \right\Vert_{C^0}  \\
&\qquad\qquad\qquad \leq  \left( \left \vert \tau\lambda \right \vert^{q+1}
\frac{e^{2\tau\Re(\lambda)} - 1}{\Re(\lambda)}  \left \Vert \varphi \right \Vert_{C^{0}}  + \tau\sum_{i=0}^{q} \left\vert \tau\lambda \right \vert^{i} \left \Vert \varphi^{(q-i)} \right \Vert_{C^{0}} \right) .
\end{align*}
The second estimate follows from Theorem~\ref{thm:analytic_intp_error}. More precisely, for any $z\in\mathcal{E}_\rho$ we consider $\gamma(\tilde{s})=-1+\tilde{s}(z+1)$, $\tilde{s} \in[0,1]$, so that $\gamma$ parametrizes the segment $[-1,z]$, and we estimate
\begin{align*}
	 	\tau \left \vert \int_{\gamma} e^{\tau(z-y) \lambda} \varphi(y) \mbox{d}y \right \vert 
		&\leq
		\tau
		\left(\sup_{z \in \mathcal{E}_{\rho}} \vert \varphi(z) \vert \right)
		 \left( \frac{ \rho + \rho^{-1} }{2} + 1 \right)
		\int_{0}^{1} \left \vert e^{\tau(z+1)(1-s) \lambda}  \right \vert \text{d}s
		\\[2ex] &= 
		\left(\sup_{z \in \mathcal{E}_{\rho}} \vert \varphi(z) \vert  \right)
		\left( \frac{ \rho + \rho^{-1} }{2} + 1 \right)
		\left( \frac{ e^{\tau \Re( (z+1) \lambda) } - 1}{ \Re( (z+1) \lambda)} \right).
	 \end{align*} 
Since the map $\zeta$ is increasing on $\RR$,
	 \begin{align*}
	 	\sup_{z \in \mathcal{E}_{\rho} } \tau \left \vert \int_{\gamma} e^{\tau(z-y) \lambda} \varphi(y) \mbox{d}y \right \vert 
		\leq 
		\left(\sup_{z \in \mathcal{E}_{\rho}} \vert \varphi(z) \vert \right)
		\left( \frac{ \rho + \rho^{-1} }{2} + 1 \right) 
		\tau \zeta \left ( \tau\sup_{z \in \mathcal{E}_{\rho}} \Re( (z+1) \lambda)  \right), 
	 \end{align*}
	 where, as already established in the proof of Lemma~\ref{lem:interp_exp},
	 \begin{align*}
		\sup_{z \in \mathcal{E}_{\rho}} \Re( (z+1) \lambda) = 
	 	 \Re(\lambda) + \frac{1}{2} \sqrt{ \Re(\lambda)^{2} ( \rho + \rho^{-1} )^{2} + \Im(\lambda)^{2}( \rho - \rho^{-1} )^{2} }.
	 \end{align*}
From Theorem 2.7, we get
\begin{align*}
\left\Vert (I-P_{K}) \left(t \mapsto \tau \int_{-1}^{t} e^{ \tau(t-s)\lambda} \varphi(s) \mbox{d} s\right)\right\Vert_{C^0} \leq \frac{4}{\rho^K(\rho-1)} \sup_{z \in \mathcal{E}_{\rho} } \tau \left \vert \int_{\gamma} e^{\tau(z-y) \lambda} \varphi(y) \mbox{d}y \right \vert,
	\end{align*}
	which finishes the proof.
\end{proof}

\begin{remark}
\label{rem:interp_error}
Assume we need to control an interpolation error of the form $(I-P_K)g$, where the function $g$ is known, and one of the above lemmas is applicable. Rather than directly using the appropriate lemma, one may get a sharper estimate by first splitting the error as follows
\begin{align*}
\left\Vert (I-P_K)g \right\Vert_{C^0} \leq \left\Vert (P_{K_0}-P_K)g \right\Vert_{C^0} + \left\Vert (I-P_{K_0})g \right\Vert_{C^0},
\end{align*}
where $K_0$ is larger than $K$. The rational behind this splitting is that, by taking $K_0$ large enough, $P_{K_0}g$ becomes an extremely accurate approximation of $g$, and the first term $\left\Vert (P_{K_0}-P_K)g \right\Vert_{C^0}$ is essentially the interpolation error one is interested in. The main difference is now that $(P_{K_0}-P_K)g$ is merely a polynomial, hence its $C^0$ norm can very easily be estimated via~\eqref{eq:C0_norm}, and we only have to use one of the above lemmas to control the second term $\left\Vert (I-P_{K_0})g \right\Vert_{C^0}$, which is expected to be much smaller since $K_0$ is larger than $K$.
\end{remark}

\subsection{Truncation errors in Fourier space}
 
\begin{lemma}[Truncation-error]
	\label{lem:tailEstimate}
	Define the linear operators $K^{(j)}_{\infty}: \mathcal{X}^{\infty}_{\nu} \rightarrow \mathcal{X}^{\infty}_{\nu}$, $j=0,\ldots,2R-1$, by 
	\begin{align}
		\label{eq:Kj_inf}
		\left[ K^{(j)}_{\infty} \left( u \right) \left( t \right) \right]_{n} \bydef 
		\begin{cases}
			0, & 0 \leq \vert n\vert \leq N, \\[2ex]
			\displaystyle
			\tau \int_{-1}^{t} e^{\tau\lambda_{n} \left( t - s\right) } (in)^j u_{n} \left( s \right) \normalfont \mbox{d}s, & \vert n\vert > N.
		\end{cases}
	\end{align}
Then $K^{(j)}_{\infty}$ is bounded and 
	\begin{align*}
		\left \Vert K^{(j)}_{\infty} \right \Vert_{ \mathcal{B} \left( \mathcal{X}^{\infty}_{\nu}, \mathcal{X}^{\infty}_{\nu} \right) } \leq \chi^{(j)}_N \bydef  \sup_{n> N} n^j
		\frac{ 1- e^{2\tau \Re(\lambda_{n})} }{ -\Re(\lambda_{n})}.
	\end{align*}
	\begin{proof}
		Let $u \in \mathcal{X}^{\infty}_{\nu}$. Then, using that $\Re(\lambda_n)=\Re(\lambda_{-n})$ for all $\vert n\vert >N$,
		\begin{align*}
			\left \Vert K^{(j)}_{\infty} \left( u \right) \right \Vert_{\X^\infty_\nu} &\leq 
			 \tau \sum_{n=N+1}^{\infty} 	\left(\left\Vert u_n \right\Vert_{C^0}+ \left\Vert u_{-n} \right\Vert_{C^0}\right) n^j	\left\Vert t\mapsto \int_{-1}^{t} e^{\tau \Re(\lambda_{n}) \left( t - s\right) } \mbox{d}s \right\Vert_{C^0}  \nu^{n}  \\
			&=  \sum_{n=N+1}^{\infty} 	\left(\left\Vert u_n \right\Vert_{C^0}+ \left\Vert u_{-n} \right\Vert_{C^0}\right) n^j	\frac{1-e^{2\tau\Re(\lambda_n)}}{-\Re(\lambda_n)}  \nu^{n}  \\
			&\leq \chi^{(j)}_N \left\Vert u \right\Vert_{\ell^1_\nu(C^0)}. \qedhere
		\end{align*} 
	\end{proof}
\end{lemma}
\begin{remark}
Note that 
	$\chi_N^{(j)} \sim N^{-2R+j}$ as $N\to\infty$, by \eqref{eq:eigenvalues}. An easy way to get an explicit value for $\chi_N^{(j)}$ is given in Appendix~\ref{app:chi}.
\end{remark}


\section{Bounds}
\label{sec:bounds}

In this section, we derive computable bounds $Y$, $Z$ and $W$ satisfying assumptions~\eqref{e:def_Y}-\eqref{e:def_W} of Theorem~\ref{thm:NewtonKantorovich}, for the operator $T$, the space $X=\X_\nu$ and the approximate solution $\bar u$ introduced in Section~\ref{sec:functional_setup}. Since we use here the case $M=1$, we drop the indices $i,j$ and the exponent $m$ from $Y$, $Z$ and $W$.

\subsection{Y-bound}
\label{sec:Ybounds}

In this subsection, we deal with the residual bound $Y$ satisfying assumption~\eqref{e:def_Y} of Theorem~\ref{thm:NewtonKantorovich}. First, observe that 
\begin{align*}
	T (\bu) - \bu &= 
	- A F (\bu) \\[2ex] &= 
	\left( - A_{KN} F_{KN} (\bu) \right) 
	\oplus \Pi_{\infty N} F (\bu)  
	\oplus \Pi_{\infty} F (\bu). 
\end{align*}
We are going to estimate each of the three terms separately, i.e. we derive bounds $Y_{KN}$, $Y_{\infty N}$ and $Y_{\infty}$ such that
\begin{align*}
\left\Vert A_{KN} F_{KN} (\bu) \right\Vert_{\X^{KN}_\nu} \leq Y_{KN}, \quad
\left\Vert \Pi_{\infty N} F (\bu) \right\Vert_{\X^{\infty N}_\nu} \leq Y_{\infty N}, \quad
\left\Vert \Pi_{\infty} F (\bu) \right\Vert_{\X^{\infty}_\nu} \leq Y_{\infty},
\end{align*}
so that
\begin{align*}
Y \bydef Y_{KN} + \epsilon_{\infty N}^{-1} Y_{\infty N} + \epsilon_{\infty }^{-1} Y_{\infty},
\end{align*}
satisfies~\eqref{e:def_Y}.

\subsubsection{Finite dimensional projection}

A bound for the finite dimensional part can be obtained by simply evaluating 
\begin{align*}
	Y_{KN} \bydef  \left \Vert A_{KN} F_{KN} (\bu) \right \Vert_{\mathcal{X}^{KN}_{\nu}}.
\end{align*}
In order to make this bound computable, we need to be able to rigorously enclose each entry of $F_{KN} (\bu)$, that is, to get enclosure for the vectors
\begin{align*}
\Pi_N F(\bu)(t_k) =  e^{ \tau(t_k+1) \L}\Pi_N f + \tau \Pi_N\left(\int_{-1}^{t_k} e^{ \tau(t_k-s)\L } 
		\gamma(\bu(s)) \mbox{d} s\right) - \bu \left( t_k \right),
\end{align*}
for all $0\leq k \leq K$. 
\begin{remark}
\label{rem:coeffs2values}
There are at least two convenient ways of representing (each component of) an element $\bu\in \X^{KN}_\nu$, namely via its values at the Chebyshev nodes, or via its coefficients in the Chebyshev basis. The Chebyshev coefficients representation simplifies all the norm computations, whereas the representation with the values at the nodes is the one we naturally end up with when computing the projection of $\Pi_NF(\bu)$ onto $\X^{KN}_\nu$. In practice, we therefore chose to implement $F_{KN}$ as going from Chebyshev coefficients to values at Chebyshev nodes. 
\end{remark}
Recalling that $\L$ leaves both $\mathcal{X}^{N}_{\nu}$ and $\mathcal{X}^{\infty}_{\nu}$ invariant, we have
\begin{align*}
 \Pi_N F(\bu)(t_k) =  e^{ \tau(t_k+1) \L_N}\Pi_N f + \tau \int_{-1}^{t_k} e^{ \tau(t_k-s)\L_N } \Pi_N \gamma(\bu(s)) \mbox{d} s - \bu \left( t_k \right).
\end{align*}
Therefore, we need to be able to rigorously enclose integrals of the form 
\begin{equation*}
\int_{-1}^{t_k} e^{ \tau(t_k-s)\L_N } 
		\varphi(s) \mbox{d} s,
\end{equation*}
where $\varphi(s)\in\CC^{2N-1}$ and each component of $\varphi$ is a polynomial in $s$. Here we use the fact that
\begin{equation*}
\L_N = Q \Lambda_N Q^{-1},
\end{equation*}
where $\Lambda_N= \text{diag}\left(\lambda_{-N},\ldots,\lambda_{N}\right)$, to write
\begin{equation*}
\int_{-1}^{t_k} e^{ \tau(t_k-s)\L_N } 
		\varphi(s) \mbox{d} s = Q \int_{-1}^{t_k} e^{ \tau(t_k-s)\Lambda_N } 
		Q^{-1}\varphi(s) \mbox{d} s,
\end{equation*}
where $Q^{-1}\varphi$ is another vector of polynomials, and thus the whole estimate boils down to rigorously enclosing integrals of the form
\begin{equation*}
\int_{-1}^{t_k} e^{ \tau(t_k-s)\lambda_n } 
		\varphi(s) \mbox{d} s,
\end{equation*}
where $\varphi$ is a scalar polynomial. We explain in detail in Appendix~\ref{app:quadrature} how this can be done.

\begin{remark}
We expect to get $Y_{KN}$ small provided $\bu$ is an accurate approximate solution of the (truncated) problem. In practice, if the obtained value of $Y_{KN}$ is not small enough for~\eqref{e:inequalities1} to hold, we should try to increase $K$ (or to split the time interval, see Section~\ref{sec:domaindecomposition}).
\end{remark}

\subsubsection{Interpolation error estimates}
\label{sec:Ybounds_inftyN}

We now turn our attention to the computation of the $Y_{\infty N}$ bound. We need to estimate
\begin{align*}
\left\Vert \Pi_{\infty N}  \left(T (\bu) - \bu\right) \right\Vert_{\mathcal{X}^{\infty N}_{\nu}} &= \left\Vert \Pi_{\infty N} F(\bu) \right\Vert_{\mathcal{X}^{\infty N}_{\nu}} \\
&= \left\Vert \left\Vert \Pi_N F(\bu) - \Pi_{KN} F(\bu) \right\Vert_{C^0}  \right\Vert_{\ell^1_\nu},
\end{align*}
i.e., to control the interpolation error for $\Pi_N F(\bu)$. Since $\Pi_N F(\bu)$ is a quantity that we can compute explicitly, we apply the idea presented in Remark~\ref{rem:interp_error}, i.e. we consider $K_0\geq K$ and estimate
\begin{align*}
& \left\Vert  \Pi_N F(\bu) - \Pi_{KN} F(\bu)  \right\Vert_{\ell^1_\nu(C^0)} \\
& \qquad\qquad \leq  \left\Vert \Pi_N F(\bu) - \Pi_{K_0N} F(\bu) \right\Vert_{\ell^1_\nu(C^0)} +  \left\Vert \Pi_{K_0N} F(\bu) - \Pi_{KN} F(\bu) \right\Vert_{\ell^1_\nu(C^0)}. 
\end{align*}
We then estimate $\left\Vert \Pi_{K_0N} F(\bu) - \Pi_{KN} F(\bu) \right\Vert_{\ell^1_\nu(C^0)}$ using~\eqref{eq:C0_norm}, and $\left\Vert \Pi_N F(\bu) - \Pi_{K_0N} F(\bu) \right\Vert_{\ell^1_\nu(C^0)}$ by combining Lemma~\ref{lem:interp_exp} and Lemma~\ref{lem:interpolation_error_Cinfty} to obtain the following statement.
\begin{proposition}
	\label{prop:Y_infN}
	Let $K_0 \geq K$,
	\begin{equation*}
	\tilde f \bydef Q^{-1}\Pi_N f,\quad \tilde{c}^{Y} (s) \bydef Q^{-1}
		\Pi_N \gamma(\bu(s)).
	\end{equation*}
	 For each $n\in\ZZ$ such that $\vert n\vert \leq N$, consider also $0 \leq q_n \leq K_0$, $\rho_n,\check\rho_n>1$,
	\begin{align*}
	\Omega_n = \frac{4 \rho_n^{-K_0}}{\rho_n -1} 
		\exp\left( \tau \left(\Re(\lambda_n) + \frac{1}{2}\sqrt{\Re(\lambda_n)^2(\rho_n+\rho_n^{-1})^2 + \Im(\lambda_n)^2 (\rho_n-\rho_n^{-1})^2} \right)\right) \vert \tilde f_n\vert,
	\end{align*}	
	\begin{equation*}
	\tilde C_n = \sigma_{K_0,q_n} \left(\left \vert \tau\lambda_n \right \vert^{q_n+1}
			\frac{ e^{2\tau \Re(\lambda_n)} - 1}{\Re(\lambda_n)} \left \Vert \tilde{c}^{Y}_{n} \right \Vert_{C^{0}}  + 
			\tau\sum_{i=0}^{q_n} \left \vert \tau\lambda_n \right \vert^{i} \left \Vert \left(\tilde{c}^{Y}_{n}\right)^{(q_n-i)} \right \Vert_{C^{0}} \right),
	\end{equation*}
	and
	\begin{align*}
	&\check C_n = \\
	&\quad\frac{2\tau(\check\rho_n+ \check\rho_n^{-1}+2)}{\check\rho_n^{K_0}(\check\rho_n-1)} \exp\left(\tau\left(\Re(\lambda_n) + \frac{1}{2}\sqrt{\Re(\lambda_n)^2(\check\rho_n+\check\rho_n^{-1})^2 + \Im(\lambda_n)^2 (\check\rho_n-\check\rho_n^{-1})^2}\right)\right) \sup_{z\in\mathcal{E}_{\check\rho_n}} \left\vert \tilde{c}^{Y}_{n}\right\vert.
	\end{align*}
Then
	\begin{align*}
	 & \left\Vert \Pi_{N} F(\bu) - \Pi_{K_0N} F(\bu) \right\Vert_{\ell^1_\nu(C^0)} \leq \left\Vert \vert Q \vert \left(\Omega + \min\left(\tilde C,\check C\right)\right) \right\Vert_{\ell^1_\nu},
	\end{align*}
	where $\Omega = \left(\Omega_n\right)_{-N\leq n\leq N}$, $\tilde C = \left(\tilde C_n\right)_{-N\leq n\leq N}$, $\check C = \left(\check C_n\right)_{-N\leq n\leq N}$, and $\min\left(\tilde C,\check C\right)$ has to be understood component-wise.
\end{proposition}
 
	\begin{proof}
		First note that 
		\begin{align*}
			& \Pi_N F(\bu) - \Pi_{K_0N} F(\bu) \nonumber\\
			 &=
			\left(I - P_{K_0}\right) \left( t \mapsto e^{ \tau(t+1) \L_N}\Pi_N f + \tau \int_{-1}^{t} e^{ \tau(t-s)\L_N } 
		 \Pi_N\gamma(\bu(s)) \mbox{d} s \right)  \nonumber\\
		&= Q\left(I - P_{K_0}\right) \left( t \mapsto e^{ \tau(t+1) \Lambda_N} \right) \tilde f  + 
		Q\left(I - P_{K_0}\right) \left( t \mapsto \tau \int_{-1}^{t} e^{ \tau(t-s)\Lambda_N } \tilde{c}^{Y}(s) \mbox{d} s \right). 
		\end{align*}
Therefore,
\begin{align}
		\label{eq:Y_infN}
			&\left\vert \Pi_N F(\bu) - \Pi_{K_0N} F(\bu) \right\vert \nonumber\\
		&\leq  \vert Q \vert \left( \left\vert \left(I - P_{K_0}\right) \left( t \mapsto e^{ \tau(t+1) \Lambda_N}\right) \right\vert \left\vert \tilde f \right\vert + \left\vert\left(I - P_{K_0}\right) \left( t \mapsto \tau \int_{-1}^{t} e^{ \tau(t-s)\Lambda_N } \tilde{c}^{Y}(s) \mbox{d} s \right)\right\vert\right),
		\end{align}
where the absolute values have to be understood component-wise. We then apply, component-wise, Lemma~\ref{lem:interp_exp} and Lemma~\ref{lem:interpolation_error_Cinfty} to the first and the second term respectively.
	\end{proof}	
\begin{remark}
	For each component, it is not clear a priori which choice of $q_n$ leads to the smallest $\tilde C_n$, so in practice we just compute $\tilde C_n$ for all $q_n$ between $1$ and $K_0$ and then keep the smallest result. Similarly, we approximately optimize for each $\rho_n$ and $\check\rho_n$. Finally, in practice we replace the $C^0$ norms and the supremum over Bernstein ellipses by the easy-to-evaluate upper-bounds provided in Lemma~\ref{lem:upperbounds_norms}.
\end{remark}

\subsubsection{Truncation in phase space}

Next, we compute a bound for the residual associated to the truncation in phase space. First note that
\begin{equation*}
 \Pi_{\infty} \left(T(\bu)-\bu\right) = \Pi_{\infty} F(\bu).
 \end{equation*} 
Therefore, for each $\vert n\vert > N$ we have
\begin{align*}
F_n (\bu)(t) = e^{\tau\lambda_n(t+1)} f_{n} + \tau \displaystyle \int_{-1}^{t} e^{\lambda_n \tau\left( t - s \right) } \gamma_n(\bu(s)) \mbox{d} s,
\end{align*}
which we estimate by
\begin{align*}
\left\Vert F_n\left(\bu\right) \right\Vert_{C^0} &\leq e^{2\tau\Re(\lambda_n)^+}\vert f_n\vert +  \left\Vert \gamma_{n} (\bu)\right\Vert_{C^0} \frac{e^{2\tau\Re(\lambda_n)}-1}{\Re(\lambda_n)},
\end{align*}
where $\Re(\lambda_n)^+\bydef  \max(\Re(\lambda_n),0)$. We then get the $Y_\infty$ bound by taking the $\ell^1_\nu$ norm of the right-hand side. Notice that $\gamma (\bu)$ only has finitely many non-zero coefficients, hence this $\ell^1_\nu$ norm amounts to a finite computation. Assuming we have a bound on the $\ell^1_\nu$ norm of the tail of $f$, this contributes to the $\X^{\infty}_\nu$ bound on the tail of $F$ after estimating $e^{2\tau\Re(\lambda_n)^+}$ uniformly for $|n|>N$.

\subsection{Z-bounds}
\label{sec:Zbounds}

In this subsection, we now focus on the bound $Z$ satisfying assumption~\eqref{e:def_Z} of Theorem~\ref{thm:NewtonKantorovich}.
In order to estimate the norm of $DT(\bu)$, we take an arbitrary $h$ in $B_{1,\epsilon}\left(0\right)$ (see~\eqref{eq:def_ball}), and we again derive three separate estimates, namely $Z_{KN}$, $Z_{\infty N}$ and $Z_{\infty}$ such that
\begin{align*}
\left\Vert \Pi_{K N} DT(\bu)h \right\Vert_{\X^{KN}_\nu} \leq Z_{KN}, \quad
\left\Vert \Pi_{\infty N} DT(\bu)h \right\Vert_{\X^{\infty N}_\nu} \leq Z_{\infty N}, \quad
\left\Vert \Pi_{\infty} DT(\bu)h \right\Vert_{\X^{\infty}_\nu} \leq Z_{\infty},
\end{align*}
so that
\begin{align}
\label{eq:Z_3parts}
Z \bydef Z_{KN} + \epsilon_{\infty N}^{-1} Z_{\infty N} + \epsilon_{\infty }^{-1} Z_{\infty},
\end{align}
satisfies~\eqref{e:def_Z}. For each of these bounds, it will be helpful, recalling the Definition~\ref{def:approx_DF} of $\widehat{DF}$, to rewrite $DT(\bu)$ as 
\begin{align*}
	DT \left( \bu\right)h  &= 
	\left( I - A \widehat{DF} \right) - A \left(DF \left( \bu \right) - \widehat{DF} \right).
\end{align*}

\begin{remark}
From~\eqref{e:inequalities1} (with $M=1$), we see that we must have $Z<1$ in order to apply Theorem~\ref{thm:NewtonKantorovich}. When deriving the estimates $Z_{KN}$, $Z_{\infty N}$ and $Z_{\infty}$, we are going to see how the choices we made when defining $\L$ are instrumental in achieving that goal. A careful choice of the weights $\epsilon_{\infty N}$ and $\epsilon_{\infty}$ is also helpful in order to make $Z$ small. However, unlike what~\eqref{eq:Z_3parts} might suggest, we cannot simply take  these weights arbitrarily large, as they will also influence the bounds $Z_{KN}$, $Z_{\infty N}$ and $Z_{\infty}$ themselves (as well as the forthcoming bound $W$). Since this dependency is somewhat intricate and difficult to study a priori, in practice we have to experiment with different values of these weights in order to select suitable ones.
\end{remark}

\subsubsection{Finite dimensional projection}

\label{sec:Zbounds_KN}

By definition of $A$ and $\widehat{DF}$, 
\begin{align}\label{eq:Z0}
		\left \Vert \Pi_{KN} \left( I - A \widehat{DF} \right)h \right \Vert_{\mathcal{X}^{KN}_{\nu}} &\leq
		\left \Vert I_{KN} - A_{KN}DF_{KN} \left( \bu \right) \right \Vert_{
		\mathcal{B} \left( \mathcal{X}^{KN}_{\nu},\mathcal{X}^{KN}_{\nu} \right)},
	\end{align}
which is just a finite computation. In order to obtain the bound $Z_{KN}$, the main task is therefore to enclose
\begin{equation*}
 \Pi_{KN} \left(DF(\bu)-\widehat{DF}\right) h =\Pi_{KN} DF(\bu) \left(\Pi_{\infty N} + \Pi_\infty\right)(h),
\end{equation*}
for all $h \in B_{1,\epsilon}(0)$. We then have to multiply the result by $\vert A_{KN}\vert$ in order to get a bound on 
\begin{equation*}
\left\Vert \Pi_{KN} A \left( DF \left( \bu \right) - \widehat{DF} \right)h \right\Vert_{\mathcal{X}^{KN}_{\nu}}.
\end{equation*}
%
We compute
\begin{align*}
\Pi_{N} \left[ DF(\bu) \left(\Pi_{\infty N} + \Pi_\infty\right)(h)\right] (t_k) 
&= \tau\int_{-1}^{t_k} e^{\tau(t_k-s)\L_N}\Pi_N D\gamma(\bu(s))\left(\Pi_{\infty N} + \Pi_\infty\right)(h)(s)\d s  \\
&= \tau Q \int_{-1}^{t_k} e^{\tau(t_k-s)\Lambda_N}Q^{-1} \Pi_N D\gamma(\bu(s))\left(\Pi_{\infty N} + \Pi_\infty\right)(h)(s) \mbox{d} s,
\end{align*}
where, according to~\eqref{eq:Dgamma},
\begin{align*}
&\Pi_N D\gamma(\bu(s))\left(\Pi_{\infty N} + \Pi_\infty\right)(h)(s) = \\
&\qquad \sum_{j=0}^{2R-1} \D^j \Pi_N\left[\left(\left(g^{(j)}\right)'(\bu(s)) - \bv^{(j)}\right) \ast \Pi_{\infty N}h(s) + \left(g^{(j)}\right)'(\bu(s)) \ast \Pi_{\infty}h(s) \right] - \R_N \Pi_{\infty N} h(s).
\end{align*}
By Lemma~\ref{lem:convo_with_h_split} we have that
\begin{align*}
&\left\Vert \Pi_N\left[\left(\left(g^{(j)}\right)'(\bu(s)) - \bv^{(j)}\right) \ast \Pi_{\infty N}h(s) + \left(g^{(j)}\right)'(\bu(s)) \ast \Pi_{\infty}h(s) \right] \right\Vert_{C^0} \\ &\qquad\qquad\qquad\qquad\qquad\qquad\qquad\qquad\qquad \leq  \Pi_N\Phi^\epsilon\left(\left(g^{(j)}\right)'(\bu) - \bv^{(j)},\left(g^{(j)}\right)'(\bu) \right),
\end{align*}
and therefore, recalling that $D_N(t)$ is defined in~\eqref{eq:DN}, and using Lemma~\ref{lem:Upsilon}, we end up with
\begin{align}
\label{eq:Z1_KN}
&\left\vert \Pi_{N} \left[ DF(\bu) \left(\Pi_{\infty N} + \Pi_\infty\right)(h)\right] (t_k) \right\vert \nonumber\\
&\qquad \leq \tau\vert Q \vert  \int_{-1}^{t_k} e^{\tau(t_k-s)\Re(\Lambda_N)} \mbox{d} s \nonumber\\
&\qquad\qquad \times \left(\vert Q^{-1}\vert \sum_{j=0}^{2R-1}\vert \D^j\vert  \Pi_N\Phi^\epsilon\left(\left(g^{(j)}\right)'(\bu) - \bv^{(j)}, \left(g^{(j)}\right)'(\bu)\right) + \left\vert Q^{-1}\R_N \right\vert \left\Vert \Pi_{\infty N} h \right\Vert_{C^0} \right) \nonumber\\
&\qquad \leq \vert Q \vert\, D_N(t_k)  \left(\vert Q^{-1} \vert \sum_{j=0}^{2R-1} \vert \D^j\vert \Pi_N\Phi^\epsilon\left(\left(g^{(j)}\right)'(\bu) - \bv^{(j)}, \left(g^{(j)}\right)'(\bu)\right)  + \epsilon_{\infty N} \Upsilon\left(Q^{-1}\R_N \right) \right).
\end{align}

\begin{remark}
A careful inspection of this bound reveals why we can hope it to be small, so that $Z_{KN} <1$. First, the last term $\epsilon_{\infty N} \Upsilon\left(Q^{-1}\R_N \right)$ should be negligible, as $\R_N$ should be very small: it simply accounts for the fact that we do not necessarily diagonalize $\tL_N$ exactly when defining $\L$, see Section~\ref{sec:defL}. More importantly, the term $\Pi_N\Phi^\epsilon\left(\left(g^{(j)}\right)'(\bu) - \bv^{(j)}, \left(g^{(j)}\right)'(\bu)\right)$ is expected to be small because we take $\bv^{(j)}$ as a constant-in-time approximation of $\left(g^{(j)}\right)'(\bu)$. Notice that, if we had taken $\L$ to be simply $(-1)^{R+1} \D^{2R}$, then there would be no $-\bv$ terms, and the only way of making $\Pi_N\Phi^\epsilon\left(\left(g^{(j)}\right)'(\bu), \left(g^{(j)}\right)'(\bu)\right)$ small would have been to take $\epsilon_{\infty N}$ small, which would then be detrimental for the $\epsilon_{\infty N}^{-1}$ factor in $Z$, see~\eqref{eq:Z_3parts}. We also point out that this estimate will benefit from the domain decomposition procedure (presented in details in Section~\ref{sec:domaindecomposition}): by splitting the time interval into smaller subintervals, and by choosing the $\bv^{(j)}$ piece-wise constant, we will be able to get a more accurate approximation of $\left(g^{(j)}\right)'(\bu)$. 
\end{remark}

\subsubsection{Interpolation error estimates}

Next, regarding the projection onto $\mathcal{X}^{\infty N}_{\nu}$ we first notice that 
\begin{align*}
\Pi_{\infty N} \left( I - A \widehat{DF} \right) =0.
\end{align*}
Therefore, we have to bound, for any $h \in B_{1,\epsilon}(0)$,
	\begin{align*}
		& \left \Vert \Pi_{\infty N} A \left( DF \left( \bu \right) - \widehat{DF} \right) h \right \Vert_{\mathcal{X}^{\infty N}_{\nu}} \nonumber\\
		& \qquad\qquad\qquad = \left \Vert \Pi_{\infty N} \left( DF \left( \bu \right)h + h \right) \right \Vert_{\mathcal{X}^{\infty N}_{\nu}} \nonumber\\
		&  \qquad\qquad\qquad = \left\Vert \left\Vert \left(I-P_K\right) \left(t\mapsto \tau Q\int_{-1}^t e^{\tau(t-s)\Lambda_N} Q^{-1}  \Pi_N D\gamma(\bu(s))h(s) \mbox{d} s \right) \right\Vert_{C^0} \right\Vert_{\ell^1_{\nu}}, 
%
%
	\end{align*}
	which we do by using Lemma~\ref{lem:interpolation_error_N}, and then Lemma~\ref{lem:convo_with_h_split}, namely
\begin{align}
	\label{eq:Z1inftyN}
	& \left \Vert \Pi_{\infty N} A \left( DF \left( \bu \right) - \widehat{DF} \right) h \right \Vert_{\mathcal{X}^{\infty N}_{\nu}} \nonumber\\
	& \leq \tau\sigma_{K,0} \Bigg(\sum_{j=0}^{2R-1}\left\Vert\vert\D_N^j\vert + \vert Q\vert \vert \Lambda_N\vert\, D_N(1) \,\vert Q^{-1}\vert \vert\D_N^j\vert \right\Vert_{B(\ell^1_{\nu},\ell^1_{\nu})}  \nonumber \\
	&\qquad\qquad\quad \times \left\Vert \Pi_N \left[\left(\left(g^{(j)}\right)'(\bu) - \bv^{(j)}\right) \ast \Pi_N h + \left(g^{(j)}\right)'(\bu) \ast \Pi_\infty h \right] \right\Vert_{\ell^{1}_{\nu}(C^0)} \nonumber \\
	&\qquad\qquad\quad + \left\Vert I_N + \vert Q\vert \vert \Lambda_N\vert\, D_N(1) \,\vert Q^{-1}\vert \right\Vert_{B(\ell^1_{\nu},\ell^1_{\nu})} \left\Vert\R_N \Pi_N h \right\Vert_{\ell^1_\nu(C^0)}\Bigg)  \nonumber \\
	& \leq \tau\sigma_{K,0} \Bigg(\sum_{j=0}^{2R-1}\left\Vert\vert\D_N^j\vert + \vert Q\vert \vert \Lambda_N\vert\, D_N(1) \,\vert Q^{-1}\vert \vert\D^j_N\vert \right\Vert_{B(\ell^1_{\nu},\ell^1_{\nu})} \tilde\Phi^\epsilon\left(\left(g^{(j)}\right)'(\bu) - \bv^{(j)},\left(g^{(j)}\right)'(\bu)\right) \nonumber \\
	&\qquad\qquad\quad +  \max(1,\epsilon_{\infty N})\left\Vert I_N + \vert Q\vert \vert \Lambda_N\vert\, D_N(1) \,\vert Q^{-1}\vert \right\Vert_{B(\ell^1_{\nu},\ell^1_{\nu})} \left\Vert\R_N \right\Vert_{B(\ell^1_{\nu},\ell^1_{\nu})} \Bigg),
\end{align}
which gives us the $Z_{\infty N}$ estimate.

\begin{remark}
As was the case for $Z_{KN}$, we see here that the definition of $\L$ and the corresponding choice of $\bv$ helps to make $Z_{\infty N}$ smaller. Also, we see here the influence of $K$: the larger we take~$K$, the smaller the constant $\sigma_{K,0}$ becomes. In Section~\ref{sec:domaindecomposition}, we will introduce another option to make  $Z_{\infty N}$ smaller, namely to subdivide the time interval.
\end{remark}

\subsubsection{Truncation in phase space}
\label{sec:Zinf}

Again, we first notice that 
\begin{align*}
\Pi_{\infty} \left( I - A \widehat{DF} \right) =0.
\end{align*}
Therefore, our last remaining task regarding the $Z$ bound is to estimate, still for any $h \in B_{1,\epsilon}(0)$,
	\begin{align*}
		\left \Vert \Pi_{\infty} A \left( DF \left( \bu \right) - \widehat{DF} \right) h \right \Vert_{\mathcal{X}^\infty_{\nu}} =
		\left \Vert \Pi_{\infty} \left( DF \left( \bu \right)h + h \right) \right \Vert_{\mathcal{X}^\infty_{\nu}}.
	\end{align*}
Introducing $\tilde g^{(j)}\in\X_{\nu}$, $j=0,\ldots,2R-1$, defined by
\begin{equation*}
\tilde g^{(j)}_n(s) \bydef \left\{
\begin{aligned}
&\left(\left(g^{(j)}\right)'(\bu(s))-\bv^{(j)}\right)_0 \quad & n=0 \\
&\left(\left(g^{(j)}\right)'(\bu(s))\right)_n \quad & n\neq 0, \\
\end{aligned}
\right.
\end{equation*}
we have for all $\vert n\vert >N$ and $t\in [-1,1]$,
\begin{align*}
 ( DF \left( \bu \right)h + h)_n(t)  &= \tau  \int_{-1}^t e^{\tau(t-s)\lambda_n} \left(\sum_{j=0}^{2R-1}(in)^j \left(\tilde g^{(j)}(s)\ast h(s)\right)_n \right) \mbox{d} s   .
 \end{align*}
By Lemmas~\ref{lem:tailEstimate} and~\ref{lem:norms} and the Banach algebra property, we infer that
\begin{align*}
\left\Vert \Pi_{\infty} (DF \left( \bu \right)h + h) \right\Vert_{\X^\infty_{\nu}} 
&\leq  \sum_{j=0}^{2R-1}\chi_N^{(j)} \left\Vert \tilde g^{(j)} 
\ast 
h \right\Vert_{\ell^1_\nu(C^0)}
\leq  \vartheta_\epsilon  \sum_{j=0}^{2R-1}\chi_N^{(j)} \left\Vert \left\Vert \tilde g^{(j)} \right\Vert_{C^0}\right\Vert_{\ell^1_\nu},
\end{align*}
which gives us the $Z_\infty$ bound.

\begin{remark}
Looking at the constants $\chi_N^{(j)}$ (see Lemma~\ref{lem:tailEstimate}), we see that $Z_\infty$ can be made as small as we want by taking $N$ large enough.
\end{remark}

\subsection{W-bound}
\label{sec:Wbounds}

In this subsection, we finally derive a computable bound $W$ satisfying assumption~\eqref{e:def_W} of Theorem~\ref{thm:NewtonKantorovich}.
In order to estimate the norm of $DT(u) - DT(\bu)$, for any $u$ in $B_{\rstar,\epsilon}(\bu)$, we take arbitrary $h$ and $v$ in $B_{1,\epsilon}\left(0\right)$, and we again derive three separate estimates, namely $W_{KN}$, $W_{\infty N}$ and $W_{\infty}$ such that
\begin{align*}
\left\Vert \Pi_{K N} \left(DT(\bu+\rstar v) - DT(\bu)\right)h \right\Vert_{\X^{KN}_\nu} & \leq W_{KN}\, \rstar\left\Vert v\right\Vert_{\X_\nu}, \\
\left\Vert \Pi_{\infty N} \left(DT(\bu+\rstar v) - DT(\bu)\right)h \right\Vert_{\X^{\infty N}_\nu} &  \leq W_{\infty N}\, \rstar\left\Vert v\right\Vert_{\X_\nu}, \\
\left\Vert \Pi_{\infty} \left(DT(\bu+\rstar v) - DT(\bu)\right)h \right\Vert_{\X^{\infty}_\nu} & \leq W_{\infty}\, \rstar\left\Vert v\right\Vert_{\X_\nu},
\end{align*}
so that
\begin{align*}
W \bydef W_{KN} + \epsilon_{\infty N}^{-1} W_{\infty N} + \epsilon_{\infty }^{-1} W_{\infty}
\end{align*}
satisfies~\eqref{e:def_W}.

We start from the fact that $DT(\bu+\rstar v) - DT(\bu) = -A \left(DF(\bu+\rstar v) - DF(\bu)\right)$, which will essentially be controlled by a bound on the second derivative $D^2F$. We compute 
\begin{align}
	& \left[ \left(DF \left( \bu + \rstar v \right) - DF \left( \bu \right) \right)h \right](t) \nonumber \\[1ex] 
	 & \qquad = 
	\tau \int_{-1}^{t} e^{\tau \left( t - s\right) \L } 
	\biggl( D\gamma \left( \bu\left(s\right) + \rstar v\left(s\right) \right) - D\gamma \left( \bu \left(s\right) \right) \biggr) h(s) \ \mbox{d}s
	\nonumber \\[1ex] 
	& \qquad = 	
	\tau\rstar \int_{-1}^{t} e^{\tau \left( t - s\right) \L} 
	 \int_{0}^{1} D^{2}\gamma \left( \bu \left(s\right) + \tilde{s} \rstar v\left(s\right) \right) 
		\left[ h\left(s\right), v\left(s\right) \right] \ \mbox{d} \tilde{s} \ \mbox{d}s
	\nonumber \\[1ex]
	& \qquad = 	
	\tau\rstar \int_{-1}^{t} e^{\tau \left( t - s\right) \L} \sum_{j=0}^{2R-1} \D^j
	 \int_{0}^{1}  \left(g^{(j)}\right)'' \! \left( \bu \left(s\right) + \tilde{s} \rstar v\left(s\right)\right) \ast v(s) \ast h(s)  \ \mbox{d} \tilde{s} \ \mbox{d}s .
	\label{eq:D2F}
\end{align}

\subsubsection{Finite dimensional projection}
\label{sec:Wbounds_KN}

Let us start with estimating the finite dimensional projection of~\eqref{eq:D2F}:
\begin{align}
\label{eq:Z2_KN}
&\left\vert \Pi_N \tau\rstar\left(\int_{-1}^{t_k} e^{\tau \left( t_k - s\right) \L} 
	 \sum_{j=0}^{2R-1} \D^j
	 \int_{0}^{1}  \left(g^{(j)}\right)'' \left( \bu \left(s\right) + \tilde{s} \rstar v\left(s\right)\right) \ast v(s) \ast h(s)  \ \mbox{d} \tilde{s} \ \mbox{d}s \right) \right\vert  \nonumber\\
	 &\quad \leq \rstar\sum_{j=0}^{2R-1}\tau\left\vert Q \int_{-1}^{t_k} e^{\tau \left( t_k - s\right) \Lambda_N} Q^{-1}\D^j
	 \int_{0}^{1} \Pi_N \left( \left(g^{(j)}\right)'' \left( \bu \left(s\right) + \tilde{s} \rstar v\left(s\right)\right) \ast v(s) \ast h(s)\right) \ \mbox{d} \tilde{s} \ \mbox{d}s \right\vert  \nonumber\\
	 &\quad \leq \rstar\sum_{j=0}^{2R-1} \vert Q\vert\, D_N(t_k) \,\vert Q^{-1}\vert\  \vert \D^j\vert \,
	  \Pi_{N} \left(
			 \left \vert \left( g^{(j)} \right)'' \right \vert \left( \Vert \bar u \Vert_{C^{0}} + \rstar \Vert v \Vert_{C^{0}} \right) 
			 \ast \Vert v \Vert_{C^{0}} \ast \Vert h \Vert_{C^{0}}
		\right),
\end{align}
where $\left \vert \left( g^{(j)} \right)'' \right \vert$ is a polynomial defined as in Lemma~\ref{lem:Theta}.
We then consider the operators $\Xi^{(j)} : \CC^{2N+1}\to\X^{KN}_{\nu} $ defined by
\begin{equation}\label{e:Xi}
\Xi^{(j)} \bydef
\begin{pmatrix}
\vert Q\vert\, D_N(t_0) \,\vert Q^{-1}\vert\  \vert \D^j\vert \\ 
\vdots \\ 
\vert Q\vert\, D_N(t_K) \,\vert Q^{-1}\vert\  \vert \D^j\vert
\end{pmatrix},
\end{equation}
and introduce
\begin{equation*}
B^{(j)} = \begin{pmatrix}
B^{(j)}_0 \\ \vdots \\ B^{(j)}_K
\end{pmatrix}
= \vert A_{KN} \vert \,   \Xi^{(j)}  
\end{equation*}
with $B^{(j)}_k : \CC^{2N+1} \to \CC^{2N+1}$.
\begin{remark}
As discussed in Remark~\ref{rem:coeffs2values}, $F_{KN}$ goes from Chebyshev coefficients to values at Chebyshev nodes, and therefore so does $DF_{KN}(\bar{u})$. Since $A_{KN}$ acts as an approximate inverse of $DF_{KN}(\bar{u})$, we construct it to go from values at Chebyshev nodes back to Chebyshev coefficients, which is why we get different Chebyshev coefficients $B^{(j)}_k$ in $B^{(j)}$, whereas we had values at Chebyshev nodes $t_k$ in $\Xi^{(j)}$.
\end{remark}
We now use Lemma~\ref{lem:norms} and Lemma~\ref{lem:Theta} to obtain
\begin{align*}
&\left\Vert \Pi_{KN} \left(A\left[ \left(DF \left( \bu + \rstar v \right) - DF \left( \bu \right) \right)h \right]\right)  \right\Vert_{\X^{KN}_\nu}  \\
&\leq \rstar\sum_{j=0}^{2R-1} \left\Vert B^{(j)} \Pi_N \left(
			 \left \vert \left( g^{(j)} \right)'' \right \vert \left( \Vert \bar u \Vert_{C^{0}} + \rstar \Vert v \Vert_{C^{0}} \right)  \ast \Vert v\Vert_{C^0} \ast \Vert h\Vert_{C^0}\right)\right\Vert_{\X^{KN}_\nu} \\
&\leq \rstar\sum_{j=0}^{2R-1} \left(\left\Vert B^{(j)}_0 \right\Vert_{B(\ell^1_\nu,\ell^1_\nu)} + 2\sum_{k=1}^K \left\Vert B^{(j)}_k \right\Vert_{B(\ell^1_\nu,\ell^1_\nu)}\right) \\
&\hspace*{3cm} \times \left\Vert \Pi_N \left( \left \vert \left( g^{(j)} \right)'' \right \vert \left( \Vert \bar u \Vert_{C^{0}} + \rstar \Vert v \Vert_{C^{0}} \right) \ast \Vert v\Vert_{C^0} \ast \Vert h\Vert_{C^0}\right) \right\Vert_{\ell^1_\nu} \\
&\leq \rstar\left\Vert v\right\Vert_{\X_\nu} \vartheta_\epsilon^2 \sum_{j=0}^{2R-1} \left(\left\Vert B^{(j)}_0 \right\Vert_{B(\ell^1_\nu,\ell^1_\nu)} + 2\sum_{k=1}^K \left\Vert B^{(j)}_k \right\Vert_{B(\ell^1_\nu,\ell^1_\nu)}\right) \left\vert \left(g^{(j)}\right)'' \right\vert \left( \left\Vert \bu\right\Vert_{\ell^1_\nu(C^0)} + \vartheta_\epsilon \rstar\right),
\end{align*}
which, after removing the $\rstar \Vert v \Vert_{\X_\nu}$ factor, gives the $W_{KN}$ bound.

\subsubsection{Interpolation error estimates}

Next, we estimate
\begin{equation*}
\left\Vert \Pi_{\infty N} A \left(DF \left( \bu + \rstar v \right) - DF \left( \bu \right) \right)h  \right\Vert_{\X^{\infty N}_\nu} = \left\Vert \Pi_{\infty N} \left(DF \left( \bu + \rstar v \right) - DF \left( \bu \right) \right)h  \right\Vert_{\X^{\infty N}_\nu},
\end{equation*}
by using Lemma~\ref{lem:interpolation_error_N} on~\eqref{eq:D2F}, and then once again Lemma~\ref{lem:norms} and Lemma~\ref{lem:Theta}. This yields
\begin{align*}
& \left\Vert \Pi_{\infty N} A \left(DF \left( \bu + \rstar v \right) - DF \left( \bu \right) \right)h  \right\Vert_{\X^{\infty N}_\nu} \\
& \leq \rstar\tau\sigma_{K,0}\sum_{j=0}^{2R-1}\left\Vert \vert \D^j\vert +\vert Q\vert \vert \Lambda_N \vert D_N(1) \vert Q^{-1}\vert \vert \D^j\vert \right\Vert_{B(\ell^1_\nu,\ell^1_\nu)} \\
&\qquad\qquad\qquad\quad \times \left\Vert \left\Vert  \Pi_N \int_{0}^{1}  \left(g^{(j)}\right)'' \left( \bu \left(t\right) + \tilde{s} r v\left(t\right)\right) \ast v(t) \ast h(t) \ \mbox{d} \tilde{s}  \right\Vert_{C^0} \right\Vert_{\ell^1_\nu} \\
& \leq \rstar\left\Vert v\right\Vert_{\X_\nu}\vartheta_\epsilon^2  \tau\sigma_{K,0}\sum_{j=0}^{2R-1}\left\Vert \vert \D^j\vert +\vert Q\vert \vert \Lambda_N \vert D_N(1) \vert Q^{-1}\vert \vert \D^j\vert \right\Vert_{B(\ell^1_\nu,\ell^1_\nu)}  \left\vert \left(g^{(j)}\right)'' \right\vert \left( \left\Vert \bu\right\Vert_{\ell^1_\nu(C^0)} + \vartheta_\epsilon \rstar\right),
\end{align*}
which, after removing the $\rstar \Vert v \Vert_{\X_\nu}$ factor, gives the $W_{\infty N}$ bound.

\subsubsection{Truncation in phase space}

Finally, for the tail part, we use once again Lemmas~\ref{lem:tailEstimate},~\ref{lem:norms} and~\ref{lem:Theta}, to get
\begin{align}
\label{eq:Z2infty}
\left\Vert \Pi_{\infty} A \left(DF \left( \bu + \rstar v \right) - DF \left( \bu \right) \right)h  \right\Vert_{\X^\infty_\nu} &= \left\Vert \Pi_{\infty} \left(DF \left( \bu + \rstar v \right) - DF \left( \bu \right) \right)h  \right\Vert_{\X^\infty_\nu} \nonumber\\
&\leq \rstar\left\Vert v\right\Vert_{\X_\nu} \vartheta_\epsilon^2  \sum_{j=0}^{2R-1}\chi^{(j)}_N \left\vert \left(g^{(j)}\right)'' \right\vert \left( \left\Vert \bu\right\Vert_{\ell^1_\nu(C^0)} + \vartheta_\epsilon \rstar\right),
\end{align}
which, after removing the $\rstar \Vert v\Vert_{\X_\nu}$ factor, gives the $W_{\infty}$ bound.

\section{Domain decomposition}
\label{sec:domaindecomposition}

In this section, we introduce a generalization of the setup introduced in Section~\ref{sec:functional_setup}, which is more efficient for handling longer integration times. Indeed, as one would expect, most of the estimates derived in Section~\ref{sec:bounds} get worse when the integration time $\tau$ increases. This happens in very explicit ways, as the $Y_{\infty N}$, $Z_{\infty N}$ and $W_{\infty N}$ estimates all contain a multiplicative factor $\tau$, but also in slightly more hidden but no less impactful manners, for instance in the $Z_{KN}$ and $Z_{\infty N}$ estimates, which depend on how well the functions $\left(g^{(j)}\right)'(\bu)$ are approximated by the time-independent $\bv^{(j)}$.

A natural option to compensate for these growing factors when $\tau$ increases would be to also increase the degree $K$ of Chebyshev interpolation used. Indeed, several of the above-mentioned bounds are proportional to the interpolation error constant $\sigma_{K,0}$, namely $Z_{\infty N}$ and $W_{\infty N}$, whereas $Y_{\infty N}$ behaves even more favorably when $K$ increases, as $C^l$ or even analytic interpolation error estimates can be used there (as opposed to the $C^1$ interpolation error estimate corresponding to $\sigma_{K,0}$). However, it should already be noted that increasing $K$ does not improve the factors related to the difference between $\left(g^{(j)}\right)'(\bu)$ and $\bv^{(j)}$ in the $Z$ estimates.

Another possibility is to split the time domain into multiple smaller subdomains, and to use several ``copies'' of the zero finding problem introduced previously, which are only weakly coupled via ``matching conditions'' at the boundary between successive subintervals. If we consider $M$ subintervals, assumed for the moment to have equal length to simplify the discussion, this essentially means replacing $\tau$ by $\frac{\tau}{M}$ in the estimates mentioned above, and several estimates then scale like $\frac{\tau}{M}\sigma_{K,0}$. Since $\sigma_{K,0}$ behaves roughly like $\frac{\ln(K)}{K}$ (see Theorem~\ref{thm:interp_error}), we already see that increasing $M$ is slightly more efficient than increasing $K$ for these estimates (see~\cite{BreLes18} for a more throughout discussion on this specific aspect). Besides, for the critical estimate $Z$, increasing $M$ instead of $K$ has another substantial benefit related to our set-up with the time independent $\L$. Indeed, we are now allowed to consider a different $\L$ on each (now smaller) subdomain.
In other words, we can approximate $\left(g^{(j)}\right)'(\bu)$ by $\bv^{(j)}$ that is piece-wise constant in time rather than constant over the whole time interval, resulting in a much smaller difference $\left(g^{(j)}\right)'(\bu)-\bv^{(j)}$, and therefore in a smaller $Z$ estimate. For the sake of simplifying the discussion, we completely ignored the impact of the matching conditions on the estimates here, but this will be discussed in detail in Section~\ref{sec:adaptedbounds}.

A key aspect that we wish to emphasize is that, thanks to these matching conditions, we still validate the entire solution at once. This means the domain decomposition setup can then be used to solve boundary value problems in time, for instance between local unstable and stable manifolds of objects of interest. This will be the subject of a future work.

An alternative approach for long time integration is time stepping: we also decompose the time domain, but then validate the solution on each subdomain successively, while rigorously propagating the errors/enclosures obtained at each step. While this direction is not the main focus of our work, the approach we developed can also handle time-stepping, which we describe in more detail in Appendix~\ref{app:timestepping}.

In this section, we describe the modifications required for the domain decomposition approach. One key aspect is that the approximate inverse $A$ should be modified judiciously. Once this is done, most of the computations presented in Section~\ref{sec:bounds} can be re-used, and we therefore mainly describe the new estimates that have to be incorporated.

\subsection{How to adapt the setup}
\label{sec:setup_DD}
 
We start again from~\eqref{eq:PDE_with_L}, for $t\in[0,2\tau]$, but instead of dealing with the entire time interval at once, we consider $M\in\NN_{\geq 2}$, and $\tau_1,\tau_2,\ldots,\tau_M>0$ such that $\sum_{m=1}^M \tau_m = \tau$. Splitting $t\in[0,2\tau]$ into $M$ subintervals (each of length $2\tau_m$), and then rescaling to $[-1,1]$ on each subinterval, the initial value problem \eqref{eq:semiflow} is then equivalent to 
\begin{align}	
	\label{eq:domaindecomp}
	\begin{cases}
		\dfrac{ d u^{(m)} }{ d t }(t) - \tau_m \L^{(m)} u^{(m)}(t)  = \tau_m \left((-1)^{R+1} \D^{2R} u^{(m)}(t) + \displaystyle\sum_{j=0}^{2R-1} \D^{j} g^{(j)} (u^{(m)})(t) -\L^{(m)} u^{(m)}(t) \right), \\[2ex]
		\hfill t \in [-1,1], \\[2ex]
		u^{(m)} \left( -1 \right) = u^{(m-1)}(1), 
	\end{cases}
\end{align}
for all $m=1,\ldots,M$, where $u^{(0)}(1) \bydef f$. A natural space in which to study this domain decomposition problem is $\X_\nu^M\bydef \left(\mathcal{X}_{\nu}\right)^M$, with the max-norm
\begin{align*}
\left\Vert u \right\Vert_{\X_\nu^M} \bydef \max_{1\leq m\leq M} \left\Vert u^{(m)} \right\Vert_{\X_\nu}.
\end{align*}
\begin{definition}
	\label{eq:zeroMap_domaindecomp}
	The zero finding map $F: \mathcal{X}_{\nu}^M \rightarrow \mathcal{X}_{\nu}^M$ for \eqref{eq:PDE} is defined by 
	\begin{align*}
		 F^{(m)} \left(u\right)\left(t\right) &\bydef
		e^{ \tau_m(t+1) \L^{(m)}}u^{(m-1)}(1) + \tau_m \int_{-1}^{t} e^{ \tau_m(t-s)\L^{(m)} }  \gamma^{(m)}(u^{(m)}(s))
		 \mbox{d} s - u^{(m)} \left( t \right),
	\end{align*}
for all $m=1,\ldots,M$,	where
	\begin{equation*}
		\gamma^{(m)}(u^{(m)}) = (-1)^{R+1} \D^{2R} u^{(m)} + \displaystyle\sum_{j=0}^{2R-1} \D^{j} g^{(j)} (u^{(m)}) -\L^{(m)} u^{(m)}.
	\end{equation*}	
Each $\L^{(m)}$ is defined, on its corresponding subinterval, as $\L$ was defined in Section~\ref{sec:defL}.
\end{definition}

One can now define a finite dimensional reduction of $F$ and a Newton-like operator $T$ in almost exactly 
the same way as in Section \ref{sec:Newtonlike}. For the sake of completeness, we present the essential details. 
To keep the coupling between the successive subdomains more easily tractable, we work with a single choice of $K$ and $N$ for all subdomains.

Before introducing the approximate inverse, let us also consider the linear operator $G: \left(\mathcal{X}_{\nu}\right)^M \rightarrow \left(\mathcal{X}_{\nu}\right)^M$ defined by 
\begin{equation*}
G^{(m)} \left(u\right)\left(t\right) \bydef  \left\{
\begin{aligned}
& - u^{(m)} \left( t \right) &\qquad m=1,\\
&e^{ (t+1)\tau_m\L^{(m)}}u^{(m-1)}(1) - u^{(m)} \left( t \right) &\qquad m=2,\ldots,M.
\end{aligned}
\right.
\end{equation*}
Notice that $G$ is invertible, and that
\begin{align*}
&\left( G^{-1}(u) \right)^{(m)}(t) =  \\
&\qquad\qquad \left\{
\begin{aligned}
& - u^{(m)} \left( t \right) &\qquad m=1,\\
& - u^{(m)}(t)-e^{(t+1)\tau_m\L^{(m)}}\sum_{l=1}^{m-1}\left(\prod_{j=l+1}^{m-1}e^{2\tau_j\L^{(j)}}\right)u^{(l)}(1) &\qquad m=2,\ldots,M.
\end{aligned}
\right.
\end{align*}
with the convention $\prod_{j=l+1}^{m-1}e^{2\tau_j\L^{(j)}} =I$ when $l=m-1$.
 
We extend the projections
	$\Pi_{KN}, \Pi_{\infty N}$ and $\Pi_{\infty}$ from the single domain 
	to the multiple domains setting in the natural way. This gives rise to the subspaces 
	\begin{align*}
		(\mathcal{X}^{KN}_{\nu})^{M} =  \Pi_{KN} \left( \mathcal{X}_{\nu}^{M} \right), \quad
		(\mathcal{X}^{\infty N}_{\nu})^{M} =  \Pi_{\infty N} \left( \mathcal{X}_{\nu}^{M} \right), \quad
		(\mathcal{X}^{\infty}_{\nu})^{M} =  \Pi_{\infty} \left( \mathcal{X}_{\nu}^{M} \right)
	\end{align*}
	and to the decomposition 
	\begin{align*}
		\mathcal{X}^{M}_{\nu} = (\mathcal{X}^{KN}_{\nu})^{M}  \oplus (\mathcal{X}^{\infty N}_{\nu})^{M} \oplus (\mathcal{X}^{\infty}_{\nu})^{M}. 
	\end{align*}
In order to introduce an appropriate approximate inverse, it will be convenient to use a block-decomposition of linear operators acting on $\X_\nu^M$, based on the above decomposition.

\begin{definition}[Approximation of $DF \left( \bu \right)$]
	Given $\bu$ in $\left(\X^{KN}_{\nu}\right)^M$, we introduce an approximate derivative $\widehat{DF}: \X^M_{\nu} \rightarrow \X^M_{\nu}$ of $F$ at $\bu$ defined by 
\begin{align*}
\widehat{DF} \bydef 
\renewcommand{\arraystretch}{3}
\left(
\begin{array}{c|c|c}
DF^M_{KN} \left( \bu \right) & 0 & 0 \\ \hline
\Pi_{\infty N} G \Pi_{KN} & -I^M_{\infty N} & 0 \\ \hline
0 & 0 & \Pi_{\infty} G \Pi_{\infty}
\end{array}
\right).
\end{align*}
\end{definition}

\begin{remark}
\label{rem:def_approx_inv}
Note that $\Pi_{KN} G \Pi_{\infty N} = 0$ and $\Pi_{\infty N} G \Pi_{\infty N} = -I^M_{\infty N}$, so that this approximate derivative $\widehat{DF}$ is simply obtained by always keeping the ``important'' part of $DF$, namely $G$, and by neglecting the others terms except on the finite dimensional subspace $\left(\X^{KN}_{\nu}\right)^M$. 

It is tempting be somewhat lazy, go one step further and also neglect the $\Pi_{\infty N} G \Pi_{KN}$ term, so that $\widehat{DF}$ and then its inverse $A$ become block-diagonal, which makes subsequent estimates easier. However, in practice we want the approximate derivative, and especially the approximate inverse of the derivative, to become more and more accurate when $N$ and $K$ increase, so that, at least in principle, we always get a contraction for $T$ by taking the finite dimensional projection large enough. If the $\Pi_{\infty N} G \Pi_{KN}$ term is not included, the consequence is that quantities like
\begin{equation*}
\left\Vert \left(I-P_K\right) \left(t\mapsto e^{(t+1)\tau_m\lambda_n^{(m)}} \right) \right\Vert_{C^0}
\end{equation*}
for all $\vert n\vert \leq N$  would have to be estimated in the new $Z^1_{\infty N}$ bound, and would have to be small enough for us to prove that $T$ is a contraction. Because these exponentials get stiffer and stiffer when $N$ increases, this would introduce a very undesirable dependency between $K$ and $N$, somewhat reminiscent of the CFL condition, where $K$ (the time discretization parameter) has to be large enough with respect to $N$ (the space discretization parameter).

Including the $\Pi_{\infty N} G \Pi_{KN}$ term in $\widehat{DF}$, and adapting $A$ accordingly, allows us to alleviate this limitation.
\end{remark}

\begin{definition}[Approximate inverse of $DF \left( \bu \right)$]
\label{def:A_n}
Given $\bu$ in $\left(\X^{KN}_{\nu}\right)^M$, we introduce an approximate inverse $A: \X^M_{\nu} \rightarrow \X^M_{\nu}$  of $DF \left( \bu \right)$, defined by 
\begin{align*}
A \bydef 
\renewcommand{\arraystretch}{3}
\left(
\begin{array}{c|c|c}
A_{KN} & 0 & 0 \\ \hline
\Pi_{\infty N} G \Pi_{KN} A_{KN} & -I^M_{\infty N} & 0 \\ \hline
0 & 0 & \Pi_{\infty} G^{-1} \Pi_{\infty}
\end{array}
\right),
\end{align*}
where $A_{KN}$ is a numerically computed approximate inverse of $DF^M_{KN}(\bu)$.
\end{definition}

Up to the fact that $A_{KN}$ is not the exact inverse of $DF^M_{KN}(\bu)$, $A$ is the inverse of $\widehat{DF}$. 
We can now again consider a Newton-like operator $T = I-AF$, this time mapping the product space $\X^M_\nu$ into itself, and try to validate a posteriori a numerically computed approximate solution $\bu$ by using Theorem~\ref{thm:NewtonKantorovich}.

On each copy of $\X_\nu$ in $\X_\nu^M$, we use the norm defined in~\eqref{eq:normXnu}, possibly with different weights $\epsilon_{\infty N}$ and $\epsilon_{\infty}$, which are thus denoted by $\epsilon_{\infty N}^{(m)}$ and $\epsilon_{\infty}^{(m)}$. For any $u\in\X^M_\nu$ and $r=(r^{(1)},\ldots,r^{(m)}) \in \RR^M_{>0}$, we define 
\begin{align*}
B^M_{r,\epsilon}(u) = \prod_{m=1}^M B_{r^{(m)},\epsilon^{(m)}}\left(u^{(m)}\right),
\end{align*}
which will play the role of $\bbox(u,r)$ when applying Theorem~\ref{thm:NewtonKantorovich}.

\subsection{Estimates on $G$ and $G^{-1}$}

We collect here basic estimates on $G$ and $G^{-1}$ that will be used several times in the next subsection, where we describe the extra terms appearing in the bounds $Y$, $Z$ and $W$ when domain decomposition is used.

\begin{lemma}
\label{lem:Pi_inftyNG}
Let $\bw \in\left(\X_\nu^{KN}\right)^M$. Then,
\begin{align*}
\left\Vert \Pi_{\infty N} G^{(1)} \left(\bw\right) \right\Vert_{\ell^1_\nu(C^0)} = 0
\end{align*}
and, for $m=2,\ldots,M$,
\begin{align*}
\left\Vert \Pi_{\infty N} G^{(m)} \left(\bw\right) \right\Vert_{\ell^1_\nu(C^0)} \leq  \left\Vert \vert Q\vert \left\Vert \left(I-P_K\right) \left(t\mapsto e^{(t+1)\tau_m\Lambda_N^{(m)}} \right) \right\Vert_{C^0} \vert Q^{-1}\vert\ \left\vert  \bw^{(m-1)}(1) \right\vert \right\Vert_{\ell^1_{\nu}} .
\end{align*}
\end{lemma}
\begin{proof}
By definition of $G$, we have
\begin{equation*}
\Pi_{\infty N} G^{(m)} \left(\bw\right) = \left\{
\begin{aligned}
& 0 &\qquad m=1,\\
&\Pi_{\infty N} \left(t\mapsto e^{ (t+1)\tau_m\L^{(m)}}\bw^{(m-1)}(1) \right) &\qquad m=2,\ldots,M.
\end{aligned}
\right.
\end{equation*}
For all $m=2,\ldots,M$, we then estimate
\begin{align*}
&\left\Vert \left\Vert \left(I-P_K\right) \left(t\mapsto e^{(t+1)\tau_m\L_N^{(m)}}\bw^{(m-1)}(1) \right) \right\Vert_{C^0} \right\Vert_{\ell^1_{\nu}}  \\
&\qquad \leq \left\Vert \left\Vert \left(I-P_K\right) \left(t\mapsto e^{(t+1)\tau_m\L_N^{(m)}} \right) \right\Vert_{C^0}  \left\vert  \bw^{(m-1)}(1) \right\vert \right\Vert_{\ell^1_{\nu}} \\
&\qquad \leq \left\Vert \vert Q\vert \left\Vert \left(I-P_K\right) \left(t\mapsto e^{(t+1)\tau_m\Lambda_N^{(m)}} \right) \right\Vert_{C^0} \vert Q^{-1}\vert\ \left\vert  \bw^{(m-1)}(1) \right\vert \right\Vert_{\ell^1_{\nu}}.
\end{align*}
\end{proof}

While these interpolation errors for stiff exponentials were exactly the terms we were worried about when we discussed the definition of $\widehat{DF}$ and $A$ in Remark~\ref{rem:def_approx_inv}, they are less worrisome here because they get multiplied by $\left\vert  \bw^{(m-1)}(1) \right\vert$, which are not specified here but which will always be small when we apply Lemma~\ref{lem:Pi_inftyNG} in Section~\ref{sec:adaptedbounds} repeatedly.

\begin{remark}
\label{rem:Pi_inftyNG}
In order to make this estimate computable, we have to get estimates on
\begin{equation*}
\left\Vert \left(I-P_K\right) \left(t\mapsto e^{(t+1)\tau_m\lambda_n^{(m)}} \right) \right\Vert_{C^0}
\end{equation*}
for all $\vert n\vert \leq N$. This is accomplished using Lemma~\ref{lem:interp_exp} together with Remark~\ref{rem:interp_error}. Similarly to what we did in Section~\ref{sec:Ybounds_inftyN}, we estimate this error by splitting it into an almost exact part, using an interpolation polynomial of high degree $K_0$, and an interpolation error estimate:
\begin{align*}
&\left\Vert \left(I-P_K\right) \left(t\mapsto e^{(t+1)\tau_m\lambda_n^{(m)}} \right) \right\Vert_{C^0} \leq \\
&\qquad\qquad \left\Vert \left(P_{K_0}-P_K\right) \left(t\mapsto e^{(t+1)\tau_m\lambda_n^{(m)}} \right) \right\Vert_{C^0} + \left\Vert \left(I-P_{K_0}\right) \left(t\mapsto e^{(t+1)\tau_m\lambda_n^{(m)}} \right) \right\Vert_{C^0},
\end{align*}
where the first term can be easily estimated via~\eqref{eq:C0_norm}, and for the second one we use again Lemma~\ref{lem:interp_exp}, 
for some well chosen $\rho>1$ (in practice a different $\rho$ is taken for each $m$ and $n$).
\end{remark}

\begin{lemma}
\label{lem:GinvPi_infty}
Let $u\in\X_\nu^M$. For $m=1,\ldots,M$ and $\vert n\vert > N$,
\begin{align*}
\left\Vert \left(G^{-1}(u)\right)^{(m)}_n \right\Vert_{C^0} \leq  \left\Vert u_n^{(m)} \right\Vert_{C^0} + \sum_{l=1}^{m-1}   \exp\left(2\tau_m\Re\left(\lambda^{(m)}_n\right)^+ + 2\sum_{j=l+1}^{m-1}\tau_j\Re\left(\lambda_n^{(j)}\right)\right)\left\vert u^{(l)}_n(1) \right\vert,
\end{align*}
where we again use the convention that empty sums are equal to $0$. Moreover,
\begin{align*}
\left\Vert \Pi_\infty \left(G^{-1}  (u)\right)^{(m)} \right\Vert_{\ell^1_\nu(C^0)} \leq \left\Vert \Pi_\infty u^{(m)} \right\Vert_{\ell^1_\nu(C^0)} + \sum_{l=1}^{m-1} e^{\mu_{m,l}} \left\Vert \Pi_\infty u^{(l)} \right\Vert_{\ell^1_\nu(C^0)},
\end{align*}
where
\begin{equation*}
\mu_{m,l} \bydef  \sup_{n>N} \left(2\tau_m \Re\left(\lambda_n^{(m)}\right)^+ + 2\sum_{j=l+1}^{m-1}\tau_j\Re\left(\lambda_n^{(j)}\right)\right).
\end{equation*}
\end{lemma}
\begin{proof}
The proof of Lemma~\ref{lem:GinvPi_infty} follows directly from the definition of $G^{-1}$ and the triangle inequality.
\end{proof}

\subsection{How to adapt the bounds}
\label{sec:adaptedbounds}

In this subsection, we derive computable bounds $Y$, $Z$ and $W$ satisfying assumptions~\eqref{e:def_Y}-\eqref{e:def_W} of Theorem~\ref{thm:NewtonKantorovich}, for the operator $T$, the space $X=\X_\nu^M$ and an approximate solution $\bar u$ in $\left(\X_\nu^{KN}\right)^M$. The dependencies in $F$ on the various subdomains are mostly uncoupled, except for the matching conditions at the boundaries, which means many of the estimations derived in Section~\ref{sec:bounds} can be re-used. Throughout this section we will comment sparsely on the bookkeeping involved to obtain the ``component'' bounds $Y^{(m)}$, $Z_i^{(m)}$ and $W^{(m)}_{ij}$, which are needed for using Theorem~\ref{thm:NewtonKantorovich} with $M>1$.

As was done in Section~\ref{sec:bounds}, each bound $Y^{(m)}$, $Z_i^{(m)}$ and $W^{(m)}_{ij}$ will be split into three parts corresponding to the decomposition~\eqref{eq:Xnu_decompo} of $\X_\nu$, i.e.
\begin{align*}
Y^{(m)} := Y_{KN}^{(m)} + \left(\epsilon^{(m)}_{\infty N}\right)^{-1} Y^{(m)}_{\infty N} + \left(\epsilon^{(m)}_{\infty }\right)^{-1} Y^{(m)}_{\infty}, \quad \text{for all } 1\leq m\leq M,
\end{align*}
and similarly for $Z^{(m)}_i$ and $W^{(m)}_{ij}$.

\begin{remark}
Similarly to what we did in Section~\ref{sec:bounds}, it will be convenient to consider arbitrary $h$ and $v$ in $B^M_{1,\epsilon}(0)$ in order to derive the bounds $Z^{(m)}_i$ and $W^{(m)}_{ij}$. That is we are going to look for $Z^{(m)}_i$ such that
\begin{align*}
\left\Vert\pi^{(m)} DT(\bu) h\right\Vert_{\X_\nu} \leq \sum_{i=1}^M Z^{(m)}_i \left\Vert h^{(i)} \right\Vert_{\X_\nu},
\end{align*}
and for $W^{(m)}_{ij}$ such that, writing an arbitrary $u\in B^M_{\rstar,\epsilon}(\bu)$ as $\bu+\rstar v$,
\begin{align*}
\left\Vert\pi^{(m)} \left(DT(\bu+\rstar v) - DT(\bu)\right) h\right\Vert_{\X_\nu} \leq \rstar\sum_{i=1}^M\sum_{j=1}^M W^{(m)}_{ij} \left\Vert h^{(i)} \right\Vert_{\X_\nu}  \left\Vert v^{(j)} \right\Vert_{\X_\nu}.
\end{align*}
\end{remark}

\subsubsection{Modifications for $Y_{KN}$}

Each entry of $F_{KN}(\bu)$ can still be computed (or more precisely enclosed), and so the generalization is straightforward:  we simply consider 
\begin{align*}
	Y^{(m)}_{KN} \bydef \left \Vert  \pi^{(m)}  \left(A_{KN} F_{KN} (\bu) \right)\right \Vert_{\mathcal{X}^{KN}_{\nu}},\quad m=1,\ldots,M,
\end{align*}
with the $F$ and $A_{KN}$ of Section~\ref{sec:setup_DD}.

\subsubsection{Modifications for $Y_{\infty N}$}
 
On each subdomain $1\leq m\leq M$, we have to estimate 
$\left\Vert \Pi_{\infty N} F^{(m)}(\bu) \right\Vert_{\X_\nu}$,
for which we can reproduce verbatim the analysis of Section~\ref{sec:Ybounds_inftyN}. However, we also get an additional term
\begin{equation*}
\Pi_{\infty N} G \Pi_{KN} A_{KN}F_{KN}(\bu)
\end{equation*}
to estimate. We therefore consider $\bw = A_{KN}F_{KN}(\bu) \in \left(\X_\nu^{KN}\right)^M$, and use Lemma~\ref{lem:Pi_inftyNG} together with Remark~\ref{rem:Pi_inftyNG} to get an upper bound for this quantity.

\subsubsection{Modifications for $Y_{\infty}$}

We have to bound
\begin{equation*}
\pi^{(m)}\Pi_\infty \left(T(\bu)-\bu\right) = \pi^{(m)}G^{-1} \left(\Pi_\infty F(\bu)\right),
\end{equation*}
and to that aim we use the first part of Lemma~\ref{lem:GinvPi_infty}, together with
\begin{align}
\label{eq:Fbu1}
\left\vert F^{(m)}_n\left(\bu\right) (1)\right\vert = \left\vert e^{2\tau_m\lambda_n^{(m)}}  \bu^{(m-1)}_n(1)  + \tau_m  \int_{-1}^1 e^{(1-s)\tau_m\lambda_n^{(m)}} \gamma_n^{(m)}(\bu(s)) \mbox{d} s \right\vert,
\end{align} 
and
\begin{align}
\label{eq:Fbu}
\left\Vert F^{(m)}_n\left(\bu\right) \right\Vert_{C^0} &\leq e^{2\tau_m\Re\left(\lambda_n^{(m)}\right)^+} \left\vert \bu^{(m-1)}_n(1) \right\vert  +  \left\Vert \gamma_{n}^{(m)} \left( \bu \right)\right\Vert_{C^0} \frac{e^{2\tau_m\Re(\lambda_n^{(m)})}-1}{\Re(\lambda_n^{(m)})},
\end{align}
for each $m=1,\ldots,M$ and $n>N$. Notice that, except maybe for $m=1$, $\bu^{(m-1)}_n=0$ because we only consider $\vert n\vert >N$, and that there are only finitely many $n$ for which $\gamma_{n}^{(m)} \left( \bu \right)$, and hence $F^{(m)}_n\left(\bu\right)$ , are non-zero. Therefore we can compute the associated $\ell^1_\nu$ norm, which gives us the bounds $Y^{(m)}_{\infty}$ for $1\leq m\leq M$.

\subsubsection{Modifications for $Z_{KN}$}
\label{sec:modZKN}

Up to using the new definition of $A$ and $\widehat{DF}$ from Section~\ref{sec:setup_DD}, there is no significant change to the $\Pi_{KN} (I - A \widehat{DF})$ part of that estimate, which remains a finite computation.
In view of~\eqref{e:def_Z} we just need to extract the different components $(Z_{KN})^{(m)}_{i}$, which is essentially a bookkeeping task. Let $A_{KN}^{(m,m')}$ be the block of the matrix $A_{KN}$ linking the $m'$-th to the $m$-th domain, and similarly for $DF^M_{KN}(\bar{u})^{(m,m')}$, then 
we need to compute 
$\| I_{KN} -\sum_{m'=1}^M A_{KN}^{(m,m')} DF^M_{KN}(\bar{u})^{(m',i)} \|_{B(\X_{\nu}^{KN},\X_{\nu}^{KN})}$ as the first part of the bound $(Z_{KN})^{(m)}_{i}$. 

Regarding the $\Pi_{KN} \left(DF(\bu)-\widehat{DF}\right) h$ part, we can reproduce verbatim the analysis of Section~\ref{sec:Zbounds_KN} on each subdomain separately. There are no new terms coming from the matching conditions, since 
$ \Pi_{\infty N} h^{(m)}(1) = 0$
for $1\leq m\leq M$.
When determining the contribution of this term to the bound $(Z_{KN})^{(m)}_{i}$,
the only coupling between the domains is through left-multiplication by $A_{KN}^{M}$.

\subsubsection{Modifications for $Z_{\infty N}$}

Regarding the $\Pi_{\infty N}\left(I - A \widehat{DF}\right)$ part of that estimate, we again get an extra term of the form $\Pi_{\infty N} G \left(\bw\right)$, this time with
\begin{align*}
\bw = \left(I^M_{KN} - A_{KN}^M DF^M_{KN}(\bu)\right)h_{KN},
\end{align*}
for $h \in B^M_{1,\epsilon}(0)$. We can then easily estimate each $\vert \bw^{(m)}(1)\vert$, 
and subsequently use Lemma~\ref{lem:Pi_inftyNG}.
The decomposition into components $(Z_{\infty N})^{(m)}_{i}$ is a bookkeeping exercise as already outlined in Section~\ref{sec:modZKN}.

Regarding the $\Pi_{\infty N} A \left(DF \left( \bu \right) - \widehat{DF} \right)$ part of that estimate,
we also pick up another term of the form $\Pi_{\infty N} G \left(\bw\right)$, this time with
\begin{equation*}
\bw = A_{KN}\Pi_{KN} \left(DF(\bu)-\widehat{DF}\right) h,
\end{equation*}
for $h \in B^M_{1,\epsilon}(0)$. The $\Pi_{KN} \left(DF(\bu)-\widehat{DF}\right) h$ part was already estimated in~\eqref{eq:Z1_KN}, since all the extra terms related to the coupling appearing in $DF(\bu)$ are also included in $\widehat{DF}$ and therefore cancel out. We then simply have to multiply this by $\vert A_{KN}\vert$, evaluate the value at 1, use Lemma~\ref{lem:Pi_inftyNG}, and do diligent bookkeeping to 
decompose the contributions to $(Z_{\infty N})^{(m)}_{i}$.

\subsubsection{Modifications for $Z_\infty$}
We have to estimate, for all $m=1,\ldots,M$,
\begin{equation*}
\left \Vert \pi^{(m)}\Pi_{\infty} A \left( DF \left( \bu \right) - \widehat{DF} \right) h \right \Vert_{\X^\infty_{\nu}} = \left \Vert \pi^{(m)}G^{-1} \left(\Pi_{\infty} w\right) \right\Vert_{\X^\infty_{\nu}}
\end{equation*}
where
\begin{align*}
w = DF(\bu)(h) - G(h),
\end{align*}
for $h \in B^M_{1,\epsilon}(0)$. For all $\vert n\vert >N$, and $1\leq m\leq M$, with the notation $\tilde g^{(m,j)}$ inherited from Section~\ref{sec:Zinf},
\begin{align*}
w^{(m)}_n(t) & = \tau_m  \int_{-1}^t e^{\tau_m(t-s)\lambda^{(m)}_n} \left(\sum_{j=0}^{2R-1}(in)^j \left(\tilde g^{(m,j)}(s)\ast h^{(m)}(s)\right)_n \right) \mbox{d} s,
\end{align*}
hence
\begin{align*}
\left\Vert \Pi_\infty w^{(m)} \right\Vert_{\ell^1_\nu(C^0)} \leq  \vartheta_\epsilon  \left\Vert h^{(m)} \right\Vert_{\X_\nu} \sum_{j=0}^{2R-1}\chi_N^{(m,j)} \left\Vert \tilde g^{(m,j)} \right\Vert_{\ell^1_\nu(C^0)},
\qquad\text{for } m=1,\ldots,M,
\end{align*}
which can be combined with the second part of Lemma~\ref{lem:GinvPi_infty} to obtain the  $(Z_{\infty})^{(m)}_{i}$ estimates.

\subsubsection{Modifications for $W_{KN}$}

As was the case for the $Y$ and the $Z$ bounds, there is nothing to change in the finite part of the $W$ bound, we can just re-use, component-wise, the bounds from Section~\ref{sec:Wbounds_KN}.

\subsubsection{Modifications for $W_{\infty N}$}

We get a final extra term of the form $\Pi_{\infty N} G \left(\bw\right)$, this time with
\begin{equation*}
\bw = A_{KN}\Pi_{KN} \left(DF(\bu+\rstar v)-DF(\bu)\right) h,
\end{equation*}
for $h,v \in B^M_{1,\epsilon}(0)$. The $\Pi_{KN} \left(DF(\bu+\rstar v)-DF(\bu)\right) h$ part was already estimated in~\eqref{eq:Z2_KN}. 
Starting back from there, and also generalizing the notation $\Xi^{(j)}$ from~\eqref{e:Xi} to multiple domains:
\[
\Xi^{(m,j)} \bydef
\begin{pmatrix}
\vert Q^{(m)}\vert\, D^{(m)}_N(t_0) \,\vert (Q^{(m)})^{-1}\vert\  \vert \D^j\vert \\ 
\vdots \\ 
\vert Q^{(m)}\vert\, D^{(m)}_N(t_K) \,\vert (Q^{(m)})^{-1}\vert\  \vert \D^j\vert 
\end{pmatrix},
\]
we get, for all $m=1,\ldots,M$,
\begin{align*}
& \left\vert \Pi_{KN} \left(DF^{(m)}(\bu+\rstar v)-DF^{(m)}(\bu)\right) h \right\vert \leq \\
&\qquad\qquad  \rstar\sum_{j=0}^{2R-1} \Xi^{(m,j)} \,
	  \Pi_N \left( \left \vert \left( g^{(j)} \right)'' \right \vert \left( \Vert \bar u^{(m)} \Vert_{C^{0}} + \rstar \Vert v^{(m)} \Vert_{C^{0}} \right) \ast \Vert v^{(m)}\Vert_{C^0} \ast \Vert h^{(m)}\Vert_{C^0}\right).
\end{align*}
We then define the ``evaluation at $t=1$'' operator $E_1 : \X^{KN}_{\nu} \to \CC^{2N+1}$ by $(E_1 u)_n = u_{n0}+2\sum_{k=1}^K u_{nk}$, and use Lemmas~\ref{lem:norms}, \ref{lem:Upsilon} and~\ref{lem:Theta} to obtain the $(W_{\infty N})^{(m)}_{ij}$ estimates.
\begin{align}\label{eq:verymessy}
  &|\bw^{(m)}(1)|  \leq  \nonumber\\
  &\quad  \rstar \vartheta_\epsilon^2 \sum_{m'=1}^{M} \Vert v^{(m')}\Vert_{\X_\nu} \Vert h^{(m')}\Vert_{\X_\nu} \sum_{j=0}^{2R-1} \Upsilon \left(
   | E_1 A^{(m,m')}_{KN} | \, \Xi^{(m',j)} \right) \,
   \left\vert \left(g^{(j)}\right)'' \right\vert \left( \left\Vert \bu^{(m')} \right\Vert_{\ell^1_\nu(C^0)} + \vartheta_\epsilon r^*\right) ,
\end{align}
where $A^{(m,m')}_{KN}$ is the block of the matrix $A_{KN}$ linking the $m'$-th to the $m$-th domain. Finally, we apply Lemma~\ref{lem:Pi_inftyNG} to obtain the $(W_{\infty N})^{(m)}_{ij}$ estimates. 

\subsubsection{Modifications for $W_{\infty}$} 

We have to estimate, for all $m=1,\ldots,M$,
\begin{equation*}
\left \Vert \pi^{(m)}\Pi_{\infty} A \left( DF \left( \bu +\rstar v \right) - DF \left( \bu\right) \right) h \right \Vert_{\X^\infty_{\nu}} = \left \Vert \pi^{(m)}G^{-1} \left( \Pi_{\infty} z \right) \right\Vert_{\X^\infty_{\nu}},
\end{equation*}
where
\begin{equation*}
z=\left( DF \left( \bu +\rstar v \right) - DF \left( \bu\right) \right) h.
\end{equation*}
Such a term was already estimated in~\eqref{eq:Z2infty}, and we obtain
\begin{equation*}
\left\Vert \Pi_{\infty} z^{(m)} \right\Vert_{\ell^1_\nu(C^0} \leq \vartheta_\epsilon^2 \rstar \left\Vert v^{(m)}\right\Vert_{\X_\nu}\left\Vert h^{(m)}\right\Vert_{\X_\nu} \sum_{j=0}^{2R-1}\chi^{(m,j)}_N  \left\vert \left(g^{(j)}\right)'' \right\vert \left( \left\Vert \bu^{(m)} \right\Vert_{\ell^1_\nu(C^0)} + \vartheta_\epsilon r^*\right),
\end{equation*}
for all $1\leq m\leq M$, which can be combined with the second part of Lemma~\ref{lem:GinvPi_infty} to obtain the $(W_{\infty})^{(m)}_{ij}$ estimates.

\section{Applications}
\label{sec:examples}

In this section, we present some applications of the validation procedure introduced in this paper, and discuss the obtained results. First, using the domain decomposition approach, we showcase the broad applicability of our method in Section~\ref{sec:ex_domain_decomposition}, by using it to rigorously integrate several PDEs of the form~\eqref{eq:PDE} having different properties and different dynamics. In Section~\ref{sec:ex_L}, we highlight the influence of one of the most significant aspects of this work, namely the solution-adapted choice of the operator $\L$ made in Section~\ref{sec:defL}, by reproducing some of the proofs from Section~\ref{sec:ex_domain_decomposition} but now with a trivial choice of $\L$, and by comparing the results. In Section~\ref{sec:ex_time_stepping}, we discuss the merits of time stepping as an alternative to domain decomposition.

\subsection{Different examples of equations}

\label{sec:ex_domain_decomposition}

We start with what is arguably the simplest nonlinear parabolic equation.
\begin{theorem}
\label{th:Fisher}
Consider the Fisher-KPP equation
	\begin{align}
		\label{eq:Fisher}
		\begin{cases}
			\dfrac{ \partial u }{ \partial t} = 
			\dfrac{ \partial^{2} u}{ \partial x^{2}} + u(1-u),
			 & (t,x) \in (0,\tend] \times [0,L], \\[2ex]
		\dfrac{ \partial^{j} u}{ \partial x^{j}}u(t,0)=\dfrac{ \partial^{j} u}{ \partial x^{j}}u(t,L),  & t \in [0,\tend], \ j=0,1, \\[2ex]			
			u(0,x) = f( x), & x \in [0,L],
	\end{cases}
	\end{align}
	with $L=4\pi$, $\tend=4$, and $f(x) = 0.5-0.5\sin\left(\frac{2\pi x}{L}\right)-\cos\left(\frac{4\pi x}{L}\right)+0.2\sin\left(\frac{8\pi x}{L}\right)$. Let $\bu=\bu(t,x)$ be the function represented in Figure~\ref{fig:Fisher}, and whose precise description in terms of Fourier-Chebyshev coefficients can be downloaded at~\cite{integratorcode}. Then, there exists a smooth solution $u$ of~\eqref{eq:Fisher} such that
\begin{align}
\label{eq:error_Fisher}
\sup_{t\in [0,\tend]} \sup_{x\in [0,L]} \vert u(t,x)-\bu(t,x) \vert \leq 5\times 10^{-2}.
\end{align}
\end{theorem}
\begin{figure}[h!]
\centering
\includegraphics[scale=0.5]{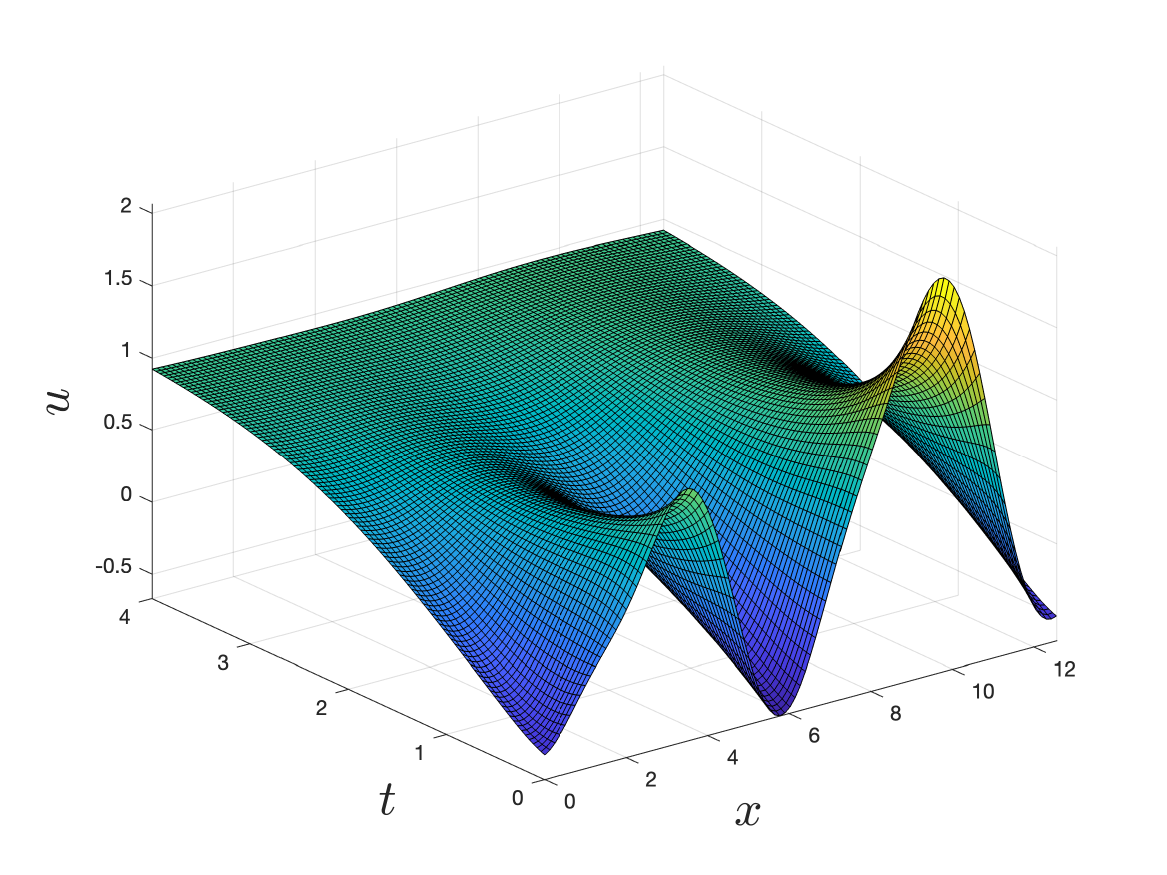}
\caption{The approximate solution $\bu$ of~\eqref{eq:Fisher}, which has been validated in Theorem~\ref{th:Fisher}.}
\label{fig:Fisher}
\end{figure}
This result is obtained by applying Theorem~\ref{thm:NewtonKantorovich} to the map $F$ of Definition~\ref{eq:zeroMap_domaindecomp} corresponding to~\eqref{eq:Fisher}, with $K=2$, $N=14$ and $M=25$. The lengths $\tau_m$ of the subdomains (which are not all the same), and the value of all the other parameters used in the proof can be found in the Matlab file \texttt{ivpdataFisher1.m}. The computational part of the proof, i.e., the evaluation of the bounds $Y$, $Z$ and $W$ derived in Section~\ref{sec:bounds} and Section~\ref{sec:domaindecomposition} for this specific approximate solution $\bu$ and the selected parameters, can be reproduced by running \runfile. 

\begin{remark}
\label{rem:Fisher}
The output of this procedure is actually slightly stronger than is stated in Theorem~\ref{th:Fisher}. Indeed, the successful application of Theorem~\ref{thm:NewtonKantorovich} implies the existence of a unique solution $u$ of~\eqref{eq:Fisher} in $\X_\nu^M$ such that $\left\Vert u-\bu\right\Vert_{\X_\nu^M}\leq 5\times 10^{-2}$. Lemma~\ref{lem:norms} then implies~\eqref{eq:error_Fisher}. The error bound that we obtain could be improved significantly. Indeed, in the proof of Theorem~\ref{th:Fisher} we tried to use a minimal amount of subdomains, namely $M=25$, but using $M=35$ instead we already get the error bound down to $1\times 10^{-3}$, and we could still get a much smaller error bound by further increasing $M$ as well as $K$.
\end{remark}

\medskip

Our method is also applicable to equations with higher order spatial derivatives, such as the Swift-Hohenberg equation, as was already highlighted in the introduction. 

\begin{theorem}
\label{th:SH_less_precise}
Consider the Swift-Hohenberg equation~\eqref{eq:SH}, with the same parameters and initial data as in Theorem~\ref{th:SH}. Let $\bu=\bu(t,x)$ be the function represented in Figure~\ref{fig:SH_less_precise}, and whose precise description in terms of Fourier-Chebyshev coefficients can be downloaded at~\cite{integratorcode}. Then, there exists a smooth solution $u$ of~\eqref{eq:SH} such that
\begin{align*}
\sup_{t\in [0,\tend]} \sup_{x\in [0,L]} \vert u(t,x)-\bu(t,x) \vert \leq 5\times 10^{-2}.
\end{align*}
\end{theorem}
\begin{figure}[h!]
\centering
\includegraphics[scale=0.5]{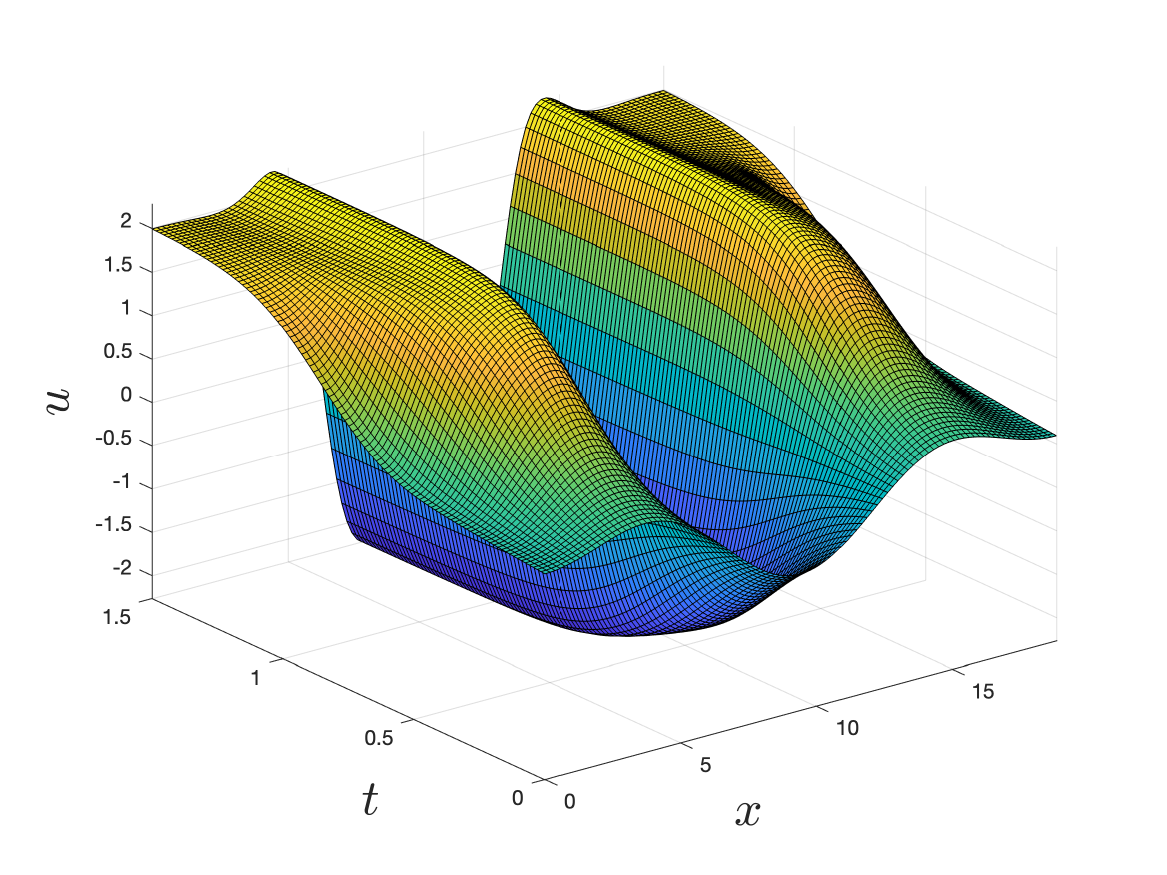}
\caption{The approximate solution $\bu$ of~\eqref{eq:SH}, which has been validated in Theorem~\ref{th:SH_less_precise}.}
\label{fig:SH_less_precise}
\end{figure}
The computational part of the proofs of Theorem~\ref{th:SH} and Theorem~\ref{th:SH_less_precise} can be reproduced by running \runfile. The value of all the parameters used in these proofs can be found in the Matlab files \texttt{ivpdataSwiftHohenberg1.m} and  \texttt{ivpdataSwiftHohenberg2.m}.

\begin{remark}
\label{rem:sharper_error_bound}
The only difference between Theorem~\ref{th:SH} and Theorem~\ref{th:SH_less_precise} is the size of the finite dimensional subspace $\left(\X_\nu^{KN}\right)^M$, that is used in each proof. For Theorem~\ref{th:SH_less_precise}, we tried to minimize the dimension of $\left(\X_\nu^{KN}\right)^M$ for which we could get a successful proof, and ended up taking $K=2$, $N=19$ and $M=108$. With these parameters, the proof can be run on a standard laptop in about 15 minutes. For Theorem~\ref{th:SH}, we used a larger approximation subspace $\left(\X_\nu^{KN}\right)^M$, with $K=5$, $N=30$ and $M=100$, and therefore got a sharper error bound, at the cost of more computation time and memory requirements.

The two approximate solutions $\bu$ used in Theorem~\ref{th:SH} and Theorem~\ref{th:SH_less_precise}, and depicted in Figure~\ref{fig:SH} and Figure~\ref{fig:SH_less_precise} respectively, are not exactly the same, but they are close enough to be indistinguishable to the naked eye. Yet, by using a higher dimensional subspace $\left(\X_\nu^{KN}\right)^M$ for the proof, we manage to obtain significantly sharper error bounds.
\end{remark}

\medskip 

We can also deal with nonlinear terms involving spatial derivatives.
\begin{theorem}
\label{th:KS}
Consider the Kuramoto–Sivashinsky equation
	\begin{align}
		\label{eq:KS}
		\begin{cases}
			\dfrac{ \partial u }{ \partial t} = 
			-\dfrac{ \partial^{4} u}{ \partial x^{4}} -\dfrac{ \partial^{2} u}{ \partial x^{2}} -\displaystyle\frac{1}{2}\dfrac{ \partial }{ \partial x} u^2
			, & (t,x) \in (0,\tend] \times [0,L], \\[2ex]
		\dfrac{ \partial^{j} u}{ \partial x^{j}}u(t,0)=\dfrac{ \partial^{j} u}{ \partial x^{j}}u(t,L),  & t \in [0,\tend], \ j=0,1,2,3, \\[2ex]			
			u (0, x) = f ( x ), & x \in [0,L],
	\end{cases}
	\end{align}
	with $L=5\pi$, $\tend=12$, and $f(x) = -\sin\left(\frac{2\pi x}{L}\right)$. Let $\bu=\bu(t,x)$ be the function represented in Figure~\ref{fig:KS}, and whose precise description in terms of Fourier-Chebyshev coefficients can be downloaded at~\cite{integratorcode}. Then, there exists a smooth solution $u$ of~\eqref{eq:KS} such that
\begin{align*}
\sup_{t\in [0,\tend]} \sup_{x\in [0,L]} \vert u(t,x)-\bu(t,x) \vert \leq 4\times 10^{-7}.
\end{align*}
\end{theorem}
The computational part of the proof of Theorem~\ref{th:KS} can be reproduced by running \runfile. This proof uses $K=5$, $N=30$ and $M=100$, while the value of all the other parameters used in the proof can be found in the Matlab file \texttt{ivpdataKuramoto1.m}.
\begin{figure}[h!]
\centering
\includegraphics[scale=0.5]{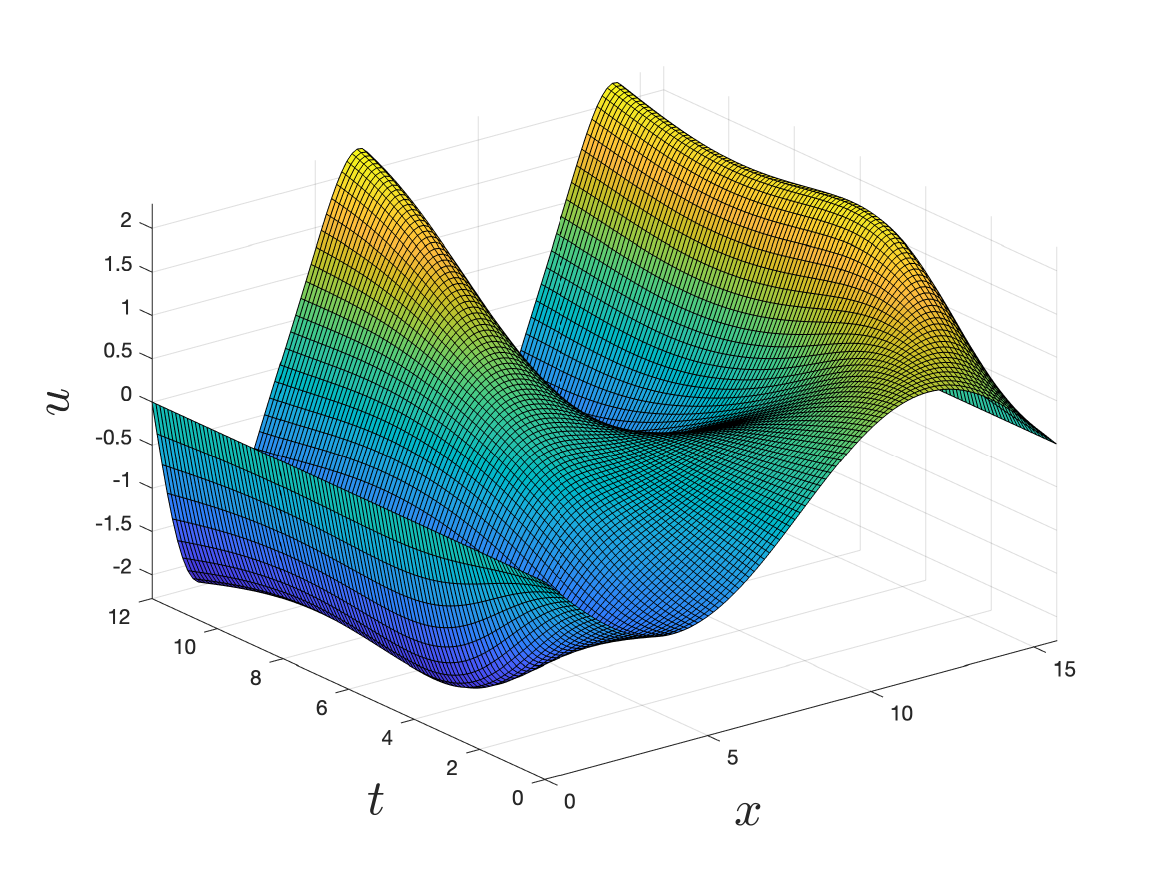}
\caption{The approximate solution $\bu$ of~\eqref{eq:KS}, which has been validated in Theorem~\ref{th:KS}.}
\label{fig:KS}
\end{figure}

\medskip

Our final example is the Ohta-Kawasaki equation, which models the evolution of diblock copolymer melts~\cite{OhtKaw86}.
\begin{theorem}
\label{th:OK}
Consider the Ohta-Kawasaki equation
	\begin{align}
		\label{eq:OK}
		\begin{cases}
			\dfrac{ \partial u }{ \partial t} = 
			-\displaystyle\frac{1}{\gamma^2}\dfrac{ \partial^{4} u}{ \partial x^{4}} -\dfrac{ \partial^{2} }{ \partial x^{2}}(u-u^3) -\sigma(u-m)
			, & (t,x) \in (0,\tend] \times [0,L], \\[2ex]
		\dfrac{ \partial^{j} u}{ \partial x^{j}}u(t,0)=\dfrac{ \partial^{j} u}{ \partial x^{j}}u(t,L),  & t \in [0,\tend], \ j=0,1,2,3, \\[2ex]			
			u (0, x ) = f ( x ), & x \in [0,L], 
	\end{cases}
	\end{align}
	with $\gamma = \sqrt{8}$, $\sigma=1/5$, $m=1/10$, $L=2\pi$, $\tend=5$, and $f(x) = m+\frac{1}{10}\cos\left(\frac{2\pi x}{L}\right)$. Let $\bu=\bu(t,x)$ be the function represented in Figure~\ref{fig:OK}, and whose precise description in terms of Fourier-Chebyshev coefficients can be downloaded at~\cite{integratorcode}. Then, there exists a smooth solution $u$ of~\eqref{eq:OK} such that
\begin{align*}
\sup_{t\in [0,\tend]} \sup_{x\in [0,L]} \vert u(t,x)-\bu(t,x) \vert \leq 3\times 10^{-3}.
\end{align*}
\end{theorem}
The computational part of the proof of Theorem~\ref{th:OK} can be reproduced by running \runfile. This proof uses $K=2$, $N=21$ and $M=395$, while the value of all the other parameters used in the proof can be found in the Matlab file \texttt{ivpdataOhtaKawasaki.m}. 
\begin{figure}[h!]
\centering
\includegraphics[scale=0.5]{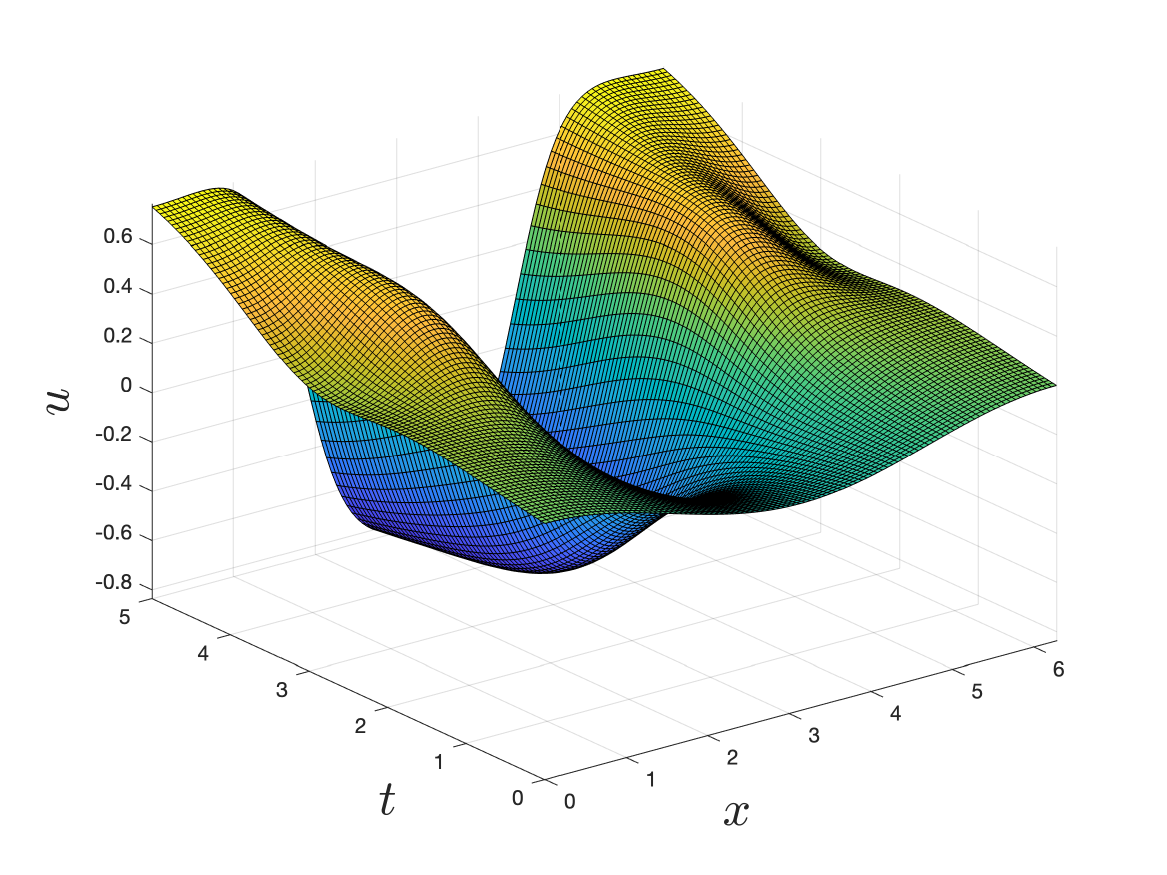}
\caption{The approximate solution $\bu$ of~\eqref{eq:OK}, which has been validated in Theorem~\ref{th:OK}.}
\label{fig:OK}
\end{figure}

\begin{remark}
\label{rem:many_parameters}
Whether computer-assisted proofs such as the ones presented in this section are successful or not (that is, whether the obtained values for the $Y$, $Z$ and $W$ bounds allow us to satisfy~\eqref{e:inequalities1}-\eqref{e:inequalities2}), depends strongly on the judicious choice of numerous computational parameters. The most obvious ones, which were already mentioned, are $K$, $N$ and $M$, which prescribe the finite dimensional subspace being used for the construction of $\L$ and of the approximate inverse $A$. However, there are other important parameters, such as the length $\tau_m$ of each subdomain, the weights  $\epsilon^{(m)}_{\infty N}$, $\epsilon^{(m)}_{\infty }$ and $\nu$, and (to a lesser extent) truncation parameters used to control interpolation errors, see Remark~\ref{rem:interp_error}, or quadrature errors, see Appendix~\ref{app:quadrature}.

At the moment, these parameters are mostly chosen by trial and error. A more automated way of selecting close to optimal parameters for a given problem would be very useful, but we did not investigate this question in this work.
\end{remark}

\subsection{The crucial influence of the operator $\L$}
\label{sec:ex_L}

In this work, we claimed that the sensible choice of the operator $\L$ described in Section~\ref{sec:defL}, which then impacts the definition of the zero finding map $F$ and of the subsequent fixed point reformulation $T$, is instrumental in getting more efficient computer-assisted proofs. We back up this claim by studying again equation~\eqref{eq:Fisher}, with the same parameters, initial data and final time, but now using a more naive choice 
\begin{align*}
	\L = (-1)^{R+1} \dfrac{ \partial^{2R} }{ \partial x^{2R}} . 
\end{align*}
We also keep the same values for $K$ and $N$ as in the proof of Theorem~\ref{th:Fisher} (namely $K=2$ and $N=14$), and then investigate the minimal number $M$ of subdomains that have to be used for the proof to be successful with this naive $\L$. We find that we then have to take at least $M=165$, whereas the proof of Theorem~\ref{th:Fisher} which used the more involved choice of $\L$ described in Section~\ref{sec:defL}, was already successful with $M=25$.

This example indeed showcases that our new choice of $\L$ allows us to use a finite dimensional subspace $\left(\X_\nu^{KN}\right)^M$ of much lower dimension, which then results in a proof that is significantly less expensive (for instance, we note that the definition of $A$ involves the numerical inversion of $DF^M_{KN}(\bu)$, which is a matrix of size $(K+1)(2N+1)M \times (K+1)(2N+1)M $).

The situation is similar for the other examples, but successful proofs with the naive choice of~$\L$ become even more costly so we did not even try to obtain one.

\subsection{Using time stepping}
\label{sec:ex_time_stepping}

For a given equation of the form~\eqref{eq:PDE}, when we try to use the approach presented in this paper with domain decomposition, in order to validate a solution for a larger and larger integration time $\tau$, we have to keep increasing the number $M$ of subdomains used, or alternatively to keep increasing the number $K$ of Chebyshev nodes used in each subdomain.

\begin{remark}
In practice, as soon as the $Y$ bound is no longer a limiting factor in getting a successful proof, it is usually more efficient to increase $M$ than to increase $K$. This was already noticed in the context of ODEs~\cite{BreLes18}, and is even more apparent in the current work, because increasing $M$ (i.e. taking more and thus smaller subdomains) means our operator $\L$ better approximates $DH(\bu)$ (with the notation of the introduction).
\end{remark}

However, $M$ cannot be taken arbitrarily large, because at some point the dimension of the subspace $\left(\X_\nu^{KN}\right)^M$ becomes too large for a computer to handle, memory-wise. This limitation is inherent to the domain-decomposition approach, where one has to deal with the entire orbit at once. 

On the other hand, with time stepping one can validate smaller portions of an the orbit one after the other, at the price of having to propagate the error bounds from one portion to the next. We can adapt our methodology to such a time stepping strategy. The details of controlling the propagating error bounds are given in Appendix~\ref{app:timestepping}. One setting in which these errors are relatively easy to manage is when the solution converges to an isolated equilibrium. Indeed, in that case we expect the semiflow to be contracting, and the errors to shrink between the beginning and the end of one step. Therefore, in principle one should be able to rigorously integrate the solution for as long as desired (see Remark~\ref{rem:integrating_forever}). Of course there is still a limitation related to computing power, but it is now predominantly related to speed (i.e., to the time taken by the computer to validate each step), rather than memory.

Below is an example of a validated solution of Fisher's equation~\eqref{eq:Fisher}, for which we used time stepping in order to go much further in time than in Section~\ref{sec:ex_domain_decomposition}.

\begin{theorem}
\label{th:Fisher_time_stepping}
Consider the Fisher-KPP equation~\eqref{eq:Fisher}, with the same parameters and initial data as in Theorem~\ref{th:Fisher}, except we now take $\tend=20$, and let $\bu=\bu(t,x)$ be the function represented in Figure~\ref{fig:Fisher_time_stepping}. Then, there exists a smooth solution $u$ of~\eqref{eq:Fisher} such that
\begin{align*}
\sup_{t\in [0,\tend]} \sup_{x\in [0,L]} \vert u(t,x)-\bu(t,x) \vert \leq 9\times 10^{-4}.
\end{align*}
\end{theorem}
\begin{figure}[h!]
\centering
\includegraphics[scale=0.5]{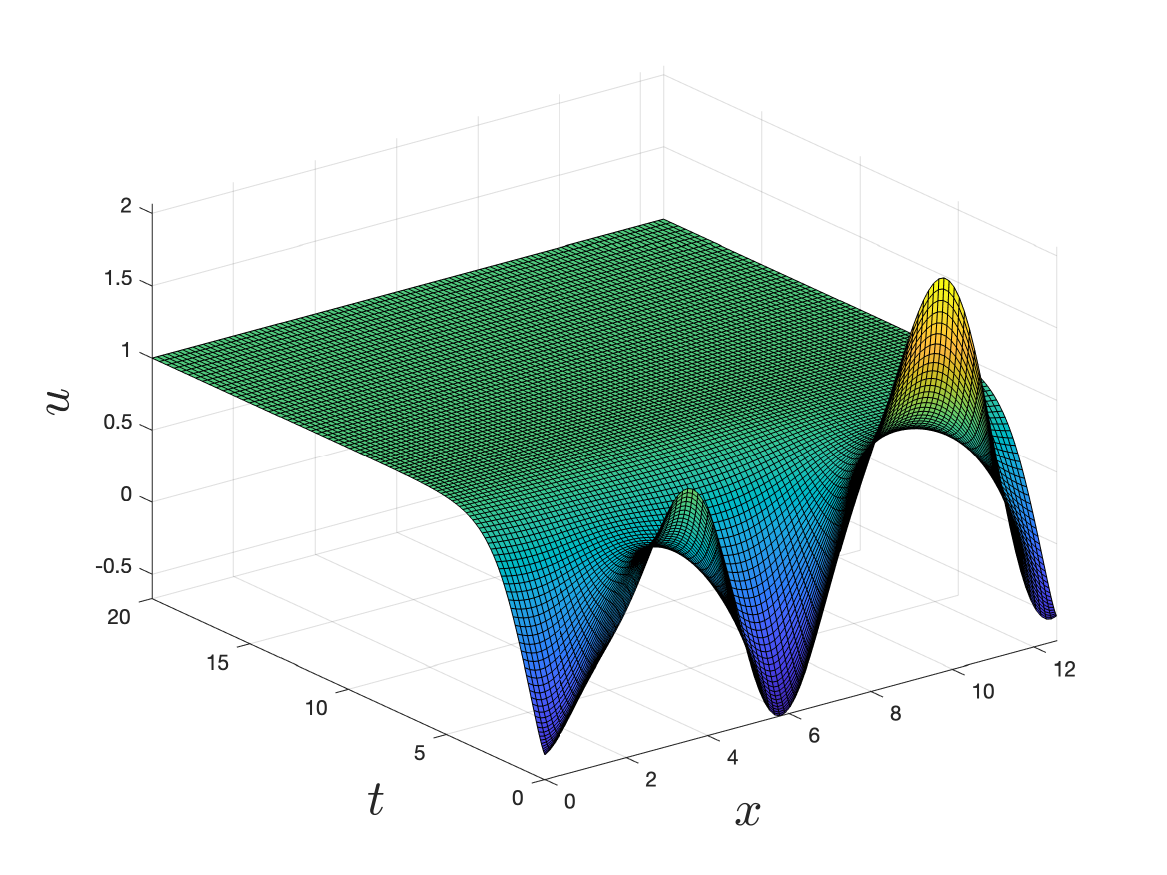}
\caption{The approximate solution $\bu$ of~\eqref{eq:Fisher}, which has been validated in Theorem~\ref{th:Fisher_time_stepping}.}
\label{fig:Fisher_time_stepping}
\end{figure}
The computational part of the proof of Theorem~\ref{th:Fisher_time_stepping} can be reproduced by running \runfile. The value of all the parameters used in this proof can be found in the Matlab file \texttt{ivpdataFisher4.m}. As opposed to the theorems presented in Section~\ref{sec:ex_domain_decomposition}, this one is proved using time stepping rather than domain decomposition. The final time was chosen to be $20$ so that the initial dynamics remains visible in the picture, but with time-stepping we can very easily obtain proofs for much larger finite times (especially because, once we are close to the equilibrium, much larger time steps can be used). Again, the precise error bounds we obtain in the proof are far from being uniform in time, and become much smaller once the solution is close to the equilibrium (for instance, in the last time step the error bound is of the order of $10^{-11}$).

\section*{Acknowledgments}

MB gratefully acknowledges the hospitality provided by the Department of Mathematics of the VU Amsterdam, which he visited several times during the preparation of this work.

\appendix
\section*{Appendix}
\section{Time stepping}
\label{app:timestepping}

When performing long time integration, an alternative for domain decomposition is to perform time stepping, i.e., to use the data at the final time of a step as the initial data for the next step. This may lead to a build-up of errors,
hence we need sharp enclosures for the endpoint of the orbit at each time step. We already have the one given by the validation radius, which is valid for all $t\in [-1,1]$, but we can also use the fact that we only need a bound for $t=1$ to get possibly better estimates. 
Indeed, suppose we have successfully proved the existence of a fixed 
point $u^{\ast} \in \bu + B_{r,\epsilon}(0)$ of $T$ for some $r >0$, then in the case of a single domain ($M=1$) we have
\begin{align}\label{eq:ustar}
	u^{\ast}(1) &= 
	e^{2\tau \L}f + \tau \int_{-1}^{1} e^{\tau\L(1-s)} \gamma \left( u^{\ast}(s) \right) \mbox{d}s  \\[1ex] &= 
	e^{2\tau \L}f + \tau \int_{-1}^{1} e^{\tau\L(1-s)} \gamma \left( \bu(s) \right) \mbox{d}s \nonumber  \\[1ex] & \quad + 
	\tau r \int_{-1}^{1} e^{\tau\L(1-s)} D\gamma \left( \bu \left( s \right) \right) h(s) \mbox{d}s  
	 \nonumber\\[1ex] & \quad + 
	\tau r^{2} \int_{-1}^{1} e^{\tau\L(1-s)} 
	\int_{0}^{1} \left( 1- \tilde{s} \right) D^{2}\gamma \left( \bu\left(s\right) + \tilde{s} r h\left(s\right) \right) \left[ h(s), h(s) \right] \mbox{d} \tilde{s} \ 			\mbox{d} s, \nonumber
\end{align}
for some $h \in B_{1, \epsilon}(0)$. We now see how we can derive an explicit enclosure of $u^{\ast}(1)$, based on this formula, studying separately the finite dimensional projection given by $\Pi_N$, and the higher order modes.

\begin{remark}
\label{rem:integrating_forever}
In some situations the error bound at $t=1$ may be tighter than the one at $t=-1$ (the initial data $f$). Observe that 
in the above expression, the only term which depends on the initial data is $e^{2\tau \L}f$. Suppose we started with an enclosure for $f$, for instance because $f$ is in fact the output of a previous time step, and assume that the eigenvalues of $\L$ are all negative, for example because the dynamics is close to a stable stationary state. We then get an enclosure for $e^{2\tau \L}f$ which is narrower than the one of $f$. Provided the error bounds we get for the remaining terms in~\eqref{eq:ustar} are small enough, it is thus possible that the total error bound we get for $u^*(1)$ is smaller than the one we started with for $f$, meaning that in such a situation errors do not grow from one time-step to the next. For the solution depicted in Figure~\ref{fig:Fisher_time_stepping}, we are seemingly getting very close to the homogeneous equilibrium, which is attractive, hence in that situation the eigenvalues of $\L$ should all be negative, and we can in principle integrate forever.

If we were to deal with an orbit which converges to a nonhomogeneous equilibrium, as is seemingly the case in Figure~\ref{fig:SH} for instance, there is one additional difficulty in rigorously integrating ``forever'' using time stepping, namely that the equilibrium is no longer isolated due to the spatial translation invariance of the problem. This then means that there will be a \emph{zero} eigenvalue in the linearized operator at the equilibrium, and that we cannot hope to shrink the errors in all directions. This issue is somewhat specific to our choice of periodic boundary conditions, and would for instance disappear with Dirichlet or Neumann boundary conditions.
\end{remark}

\subsection{Finite dimensional projection}

We can rewrite $\Pi_N u^{\ast}(1)$ as
\begin{align*}
	\Pi_N u^{\ast}(1) &=  
	Q e^{L \Lambda_N}Q^{-1}\Pi_N f + \tau Q \int_{-1}^{1} e^{\tau\Lambda_N(1-s)} Q^{-1} \Pi_N \gamma \left( \bu(s) \right) \mbox{d}s  \\[1ex] & \quad + 
	\tau r Q \int_{-1}^{1} e^{\tau\Lambda_N (1-s)} Q^{-1} \Pi_N D\gamma \left( \bu \left( s \right) \right) h(s) \mbox{d}s  
	\\[1ex] & \quad + 
	\tau r^{2} Q \int_{-1}^{1} e^{\tau\Lambda_N(1-s)}Q^{-1} 
	\int_{0}^{1} \left( 1- \tilde{s} \right) \Pi_N D^{2}\gamma \left( \bu\left(s\right) + \tilde{s} r h\left(s\right) \right) \left[ h(s), h(s) \right] \mbox{d} \tilde{s} \ 			\mbox{d} s, 
\end{align*}
where the first two terms can be evaluated explicitly (quadrature + error bounds), and we need to estimate the last two. A computation very similar to the one we made to get the $Z_{KN}$ bound yields, using Lemma~\ref{lem:convo_with_h_split},
\begin{align*}
&\left\vert \tau Q \int_{-1}^{1} e^{\tau\Lambda_N (1-s)} Q^{-1} \Pi_N D\gamma \left( \bu \left( s \right) \right) h(s) \mbox{d}s  \right\vert \leq \\
&\quad \vert Q \vert\, D_N(1)  \left(\vert Q^{-1} \vert \sum_{j=0}^{2R-1} \vert \D^j\vert \Pi_N \check\Phi\left(\left(g^{(j)}\right)'(\bu) - \bv^{(j)},\left(g^{(j)}\right)'(\bu)\right)  + \max(1,\epsilon_{\infty N}) \Upsilon\left(Q^{-1}\R_N \right) \right).
\end{align*}

We now turn our attention to the second order terms. Going back to the computation made for the $W_{KN}$ bound in~\eqref{eq:Z2_KN}, and applying Lemma~\ref{lem:norms}, \ref{lem:Upsilon} and \ref{lem:Theta} yields
\begin{align*}
& \left\vert \tau Q \int_{-1}^{1} e^{\tau\Lambda_N(1-s)}Q^{-1} 
	\int_{0}^{1} \left( 1- \tilde{s} \right) \Pi_N D^{2}\gamma \left( \bu\left(s\right) + \tilde{s} r h\left(s\right) \right) \left[ h(s), h(s) \right] \mbox{d} \tilde{s} \ \mbox{d} s \right\vert  \\
& \qquad \leq \frac{1}{2} \sum_{j=0}^{2R-1} \Upsilon \left( \vert Q\vert\, D_N(1) \,\vert Q^{-1}\vert\  \vert \D^j\vert  \right) \,
	  \left\vert \left(g^{(j)}\right)'' \right\vert \left( \left\Vert \bu \right\Vert_{\ell^1_\nu(C^0)} + \vartheta_\epsilon  r^*\right). 
\end{align*}

\subsection{Tail part (modes for $\vert n\vert > N$)}

For all $\vert n\vert >N$,
\begin{align*}
	u^{\ast}_n(1) &=  
	e^{2\tau \lambda_n}f_n + \tau \int_{-1}^{1} e^{\tau\lambda_n(1-s)} \gamma_n \left( \bu(s) \right) \mbox{d}s  \\[1ex] & \quad + 
	\tau r \int_{-1}^{1} e^{\tau\lambda_n(1-s)} D\gamma_n \left( \bu \left( s \right) \right) h(s) \mbox{d}s  
	\\[1ex] & \quad + 
	\tau r^{2} \int_{-1}^{1} e^{\tau\lambda_n(1-s)} 
	\int_{0}^{1} \left( 1- \tilde{s} \right) D^{2}\gamma_n \left( \bu\left(s\right) + \tilde{s} r h\left(s\right) \right) \left[ h(s), h(s) \right] \mbox{d} \tilde{s} \ 			\mbox{d} s.
\end{align*}
Therefore, this time looking at the computations made for $Z_\infty$ and $W_\infty$, we get
\begin{align*}
	\left\Vert \Pi_\infty u^{\ast}(1) \right\Vert_{\ell^1_\nu}  &\leq 
	\exp\left(2\tau\sup_{n>N} \Re(\lambda_n)\right) \left\Vert \Pi_\infty f \right\Vert_{\ell^1_\nu} + \sum_{\vert n\vert >N } \left\vert \tau \int_{-1}^{1} e^{\tau\lambda_n(1-s)} \gamma_n \left( \bu(s) \right) \mbox{d}s \right\vert \nu^{\vert n\vert} \\[1ex] & \quad + 
	r \vartheta_\epsilon  \sum_{j=0}^{2R-1}\chi_N^{(j)}  \left\Vert \tilde g^{(j)} \right\Vert_{\ell^1_\nu(C^0)} +
	\frac{r^2}{2}  \vartheta_\epsilon^2\sum_{j=0}^{2R-1}\chi^{(j)}_N \left\vert \left(g^{(j)}\right)'' \right\vert \left( \left\Vert \bu \right\Vert_{\ell^1_\nu(C^0)} + \vartheta_\epsilon  r^*\right),
\end{align*}
where the sum over $\vert n\vert >N$ is in fact finite and made of terms that can each be evaluated (quadrature + error bounds).

Combining this with the estimates obtained in the above sub-section, we can get an explicit enclosure of $u^{\ast}(1)$.

\begin{remark}
Here we assumed $M=1$, but time stepping can also be combined with the domain decomposition technique from Section~\ref{sec:domaindecomposition}, where we subdivide each time step into $M$ subdomains. Then, we can simply use the above estimates on each subinterval successively: we get an enclosure for $(u^*)^{(1)}(1)$, use~\eqref{eq:ustar} with $(u^*)^{(1)}(1)$ instead of $f$ to get an enclosure on $(u^*)^{(2)}(1)$, and so on, until we get an enclosure of $(u^*)^{(M)}(1)$ which is the solution at the end of one step using domain decomposition.
\end{remark}

\section{Interpolation error estimates}
\label{app:interp_err}

In this section, we provide explicit formulas for the constants $\sigma_{K,l}$ satisfying Definition~\ref{def:interp_cste}. Let us first show that, although most of the interpolation error estimates of this kind that can be found in the literature are stated for real-valued functions, they do also hold for complex valued-functions.

\begin{lemma}
\label{lem:interp_error_real}
Let $K\in\NN_{\geq 1}$, $l\in\{0,1,\ldots,K\}$ and $\sigma^\RR_{K,l}$ a real constant such that, for all $f\in C^{l+1}([-1,1],\RR)$,
\begin{align}
\label{eq:interp_error_real}
		\left \Vert f - P_{K}(f)  \right \Vert_{C^{0}} \leq 
		\sigma^\RR_{K,l} \left \Vert f^{(l+1)} \right \Vert_{C^{0}}.
	\end{align}
Then~\eqref{eq:interp_error_real} also holds for any $f\in C^{l+1}([-1,1],\CC)$, i.e., one can take $\sigma_{K,l}=\sigma^\RR_{K,l}$ in Definition~\ref{def:interp_cste}.
\end{lemma}
\begin{proof}
Let $f\in C^{l+1}([-1,1],\CC)$ and fix any $x\in[-1,1]$. Denoting by $\theta\in[0,2\pi)$ the argument of the complex number $f(x) - P_{K}(f)(x)$, we get
\begin{align*}
\left\vert f(x) - P_{K}(f)(x)  \right \vert &=  e^{-i\theta}\left( f(x) - P_{K}(f)(x)  \right)  \\
&=  \Re \left( e^{-i\theta} (f(x) - P_{K}(f)(x))  \right)  \\
&= \Re\left(e^{-i\theta} f\right)(x) - P_{K}(\Re\left(e^{-i\theta} f\right))(x)  \\
&\leq \sigma^\RR_{K,l} \left \Vert \Re\left(e^{-i\theta} f\right)^{(l+1)} \right \Vert_{C^{0}} \\
&\leq \sigma^\RR_{K,l} \left \Vert  f^{(l+1)} \right \Vert_{C^{0}},
\end{align*}
where the imaginary terms that should appear in the second line must vanish since we are computing a real quantity.
\end{proof}

\begin{theorem}
	\label{thm:interp_error}
	Let $K\in\NN_{\geq 1}$, $l \in \{0,1,\ldots K\}$, and $\Lambda_K$ be the Lebesgue constant associated to interpolation at the $K+1$ Chebyshev nodes (see, e.g.,~\cite{ApproximationTheory}, and~\cite{EhlZel66} for explicit values and tight bounds of $\Lambda_K$). Then,
	\begin{align}
		\label{eq:interp_error}
		\sigma_{K,l} &\bydef 
		\begin{cases}
			\min\left\{\left(1+\Lambda_K\right)\left(\frac{\pi}{2}\right)^{l+1}\frac{(K-l)!}{(K+1)!}\ ,\ \frac{1}{l+1}\sum_{q=0}^{\left[\frac{l}{2}\right]}\frac{1}{4^q\, (l-2q)!\, (q!)^2}\right\}, & 0 \leq l \leq K-1, \\[2ex]
			\dfrac{1}{ 2^{K-1} (K+1)!} & l=K. 	
		\end{cases}
	\end{align}		
	satisfies Definition~\ref{def:interp_cste}.	
\end{theorem}
For real-valued functions, the case $l=K$ is classical, as is the first part of the estimate for $l\leq K-1$, which follows from Jackson's theorem (see e.g.~\cite{Cheney}). Lemma~\ref{lem:interp_error_real} ensures these results carry over to complex-valued functions. However, the obtained estimates are far from sharp for small values of $K$. An alternative estimate, namely the second part of the minimum for $l\leq K-1$, was derived in~\cite{BreLes18}. In the important case $l=0$, this estimate gives $1$ and is thus smaller than $\left(1+\Lambda_K\right)\frac{\pi}{2}\frac{K!}{(K+1)!}$ for $K\leq 3$, but still relatively far from being sharp.

In the sequel, we introduce an optimization problem which is equivalent to finding the optimal value of $\sigma_{K,0}$ satisfying Definition~\ref{def:interp_cste}. According to Lemma~\ref{lem:interp_error_real}, we can restrict our attention to real-valued functions. Using interval arithmetic, we are then able to find a rigorous (and relatively tight) enclosure of the solution of the minimization problem, thereby providing an almost optimal value for $\sigma_{K,0}$.

\begin{definition}
Let $K\in\NN_{\geq 1}$ and $y\in\RR^{K+1}$. Let $P_K(y)$ denotes the polynomial of degree at most $K$ interpolating $y$ at the Chebyshev nodes, i.e. such that $P_K(y)(t_k) = y_k$ for $k=0 \dots K$. Whenever $y$ is such that
\begin{align}
\label{eq:hyp_y}
\vert y_k\vert \leq \vert t_k\vert \text{ for } k=0 \dots K,
\quad\text{and}\quad
\left\vert y_{k+1}-y_k \right\vert \leq \left\vert t_{k+1} - t_k \right\vert
\text{ for } k=0 \dots K-1,
\end{align}
we also introduce, for all $t\in[-1,1]$,
\begin{align*}
\overline{\varphi}_K(y)(t) = \min\bigl\{ y_{k^*} + t - t_{k^*},\, y_{{k^*}+1} -t + t_{{k^*}+1}, \, \vert t\vert \bigr\},
\end{align*}
and
\begin{align*}
\underline{\varphi}_K(y)(t) = \max\bigl\{ y_{k^*} - t + t_{k^*},\, y_{{k^*}+1} +t - t_{{k^*}+1}, \, -\vert t\vert \bigr\},
\end{align*}
where ${k^*}\in\{0,1,\ldots,K-1\}$ is such that $t\in[t_{k^*},t_{{k^*}+1}]$.
\end{definition}
Notice that~\eqref{eq:hyp_y} ensures that $\overline{\varphi}_K(y)$ and $\underline{\varphi}_K(y)$ are well defined: for $t=t_k$ we get $\overline{\varphi}_K(y)(t) = \underline{\varphi}_K(y)(t)=y_k$ regardless of whether we consider $k^*=k$ or $k^*=k-1$. Notice also that $\overline{\varphi}_K(y)$ (resp.\ $\underline{\varphi}_K(y)$) is the largest (resp.\ the smallest) piece-wise linear function $\varphi$ such that $\varphi(t_k)=y_k$ for all $k\in\{0,1,\ldots,K\}$, $\varphi(0)=0$, and $\vert\varphi'(t)\vert\leq 1$ for all $t\in[-1,1]\setminus\{0,t_0,\ldots,t_K\}$.

\begin{proposition}
\label{prop:sigma2optim}
Let $K\in\NN_{\geq 1}$, and
\begin{align*}
\Omega_K = \left\{ y\in\RR^{K+1},\ \vert y_k\vert \leq \vert t_k\vert \text{ for } k=0 \dots K,\ \vert y_{k+1}-y_k \vert \leq \vert t_{k+1} - t_k \vert \text{ for } k=0 \dots K-1 \right\}.
\end{align*} 
Then,
\begin{align}
\label{eq:sigmaK0_opt}
\sup_{\substack{ f\in C^1([-1,1],\RR)\\ f'\neq 0}} \frac{\left \Vert f - P_{K}(f)  \right \Vert_{C^{0}}}{\left \Vert f'  \right \Vert_{C^{0}}} = \max_{\substack{ t\in[-1,1] \\ y\in\Omega_K} } \max\left( \overline{\varphi}_K(y)(t) - P_K(y)(t),\, P_K(y)(t) - \underline{\varphi}_K(y)(t) \right).
\end{align}
\end{proposition}
\begin{proof}
Without loss of generality, we can assume $\left \Vert f'  \right \Vert_{C^{0}}=1$ and $f(0) = 0$, so that
\begin{align*}
\sup_{\substack{ f\in C^1([-1,1],\RR)\\ f'\neq 0}} \frac{\left \Vert f - P_{K}(f)  \right \Vert_{C^{0}}}{\left \Vert f'  \right \Vert_{C^{0}}} = \sup_{\substack{ f\in C^1([-1,1],\RR) \\ f(0) = 0,\ \left \Vert f'  \right \Vert_{C^{0}} = 1 }} \left \Vert f - P_{K}(f)  \right \Vert_{C^{0}}.
\end{align*}
For any $f\in C^1([-1,1],\RR)$ such that $f(0) = 0$ and $\left \Vert f'  \right \Vert_{C^{0}} = 1$, consider the vector $y\in\RR^{K+1}$ defined by $y_k = f(t_k)$ for $k=0 \dots K$, which belongs to $\Omega_K$. Since $\left \Vert f'  \right \Vert_{C^{0}} = 1$, we get
\begin{align*}
\underline{\varphi}_K(y) \leq f \leq \overline{\varphi}_K(y),
\end{align*}
and $P_K(f) = P_K(y)$. Therefore,
\begin{align*}
\left \Vert f - P_{K}(f)  \right \Vert_{C^{0}} \leq \max_{t\in[-1,1]} \max\left( \overline{\varphi}_K(y)(t) - P_K(y)(t),\, P_K(y)(t) - \underline{\varphi}_K(y)(t) \right),
\end{align*}
and hence
\begin{align*}
\sup_{\substack{ f\in C^1([-1,1],\RR)\\ f'\neq 0}} \frac{\left \Vert f - P_{K}(f)  \right \Vert_{C^{0}}}{\left \Vert f'  \right \Vert_{C^{0}}} \leq \max_{\substack{ t\in[-1,1] \\ y\in\Omega_K} } \max\left( \overline{\varphi}_K(y)(t) - P_K(y)(t),\, P_K(y)(t) - \underline{\varphi}_K(y)(t) \right).
\end{align*}
Conversely, for any $y\in\Omega_K$, the function $f=\overline{\varphi}_K(y)$ satisfies $f(0) = 0$ and $\left \vert f'(t)  \right \vert \leq 1$ for each $t$ where $f'$ is defined. Considering a sequence of $C^1$ functions $f_n$ such that $f_n(0) = 0$, $\left \Vert f'_n  \right \Vert_{C^{0}} = 1$, and $\left\Vert f_n-f\right\Vert_{C^{0}} \underset{n\to\infty}{\rightarrow} 0$, we get that
\begin{align*}
\sup_{\substack{ f\in C^1([-1,1],\RR) \\ f(0) = 0,\ \left \Vert f'  \right \Vert_{C^{0}} = 1 }} \left \Vert f - P_{K}(f)  \right \Vert_{C^{0}} \geq \max_{t\in[-1,1]}  \overline{\varphi}_K(y)(t) - P_K(y)(t).
\end{align*}
Similarly, 
\begin{align*}
\sup_{\substack{ f\in C^1([-1,1],\RR) \\ f(0) = 0,\ \left \Vert f'  \right \Vert_{C^{0}} = 1 }} \left \Vert f - P_{K}(f)  \right \Vert_{C^{0}} \geq \max_{t\in[-1,1]}  P_K(y)(t) - \underline{\varphi}_K(y)(t),
\end{align*}
and therefore
\begin{align*}
\sup_{\substack{ f\in C^1([-1,1],\RR) \\ f(0) = 0,\ \left \Vert f'  \right \Vert_{C^{0}} = 1 }} \left \Vert f - P_{K}(f)  \right \Vert_{C^{0}} \geq \max_{\substack{ t\in[-1,1] \\ y\in\Omega_K} } \max\left( \overline{\varphi}_K(y)(t) - P_K(y)(t),\, P_K(y)(t) - \underline{\varphi}_K(y)(t) \right),
\end{align*}
which concludes the proof.
\end{proof}
Proposition~\ref{prop:sigma2optim} gives us a tractable way of obtaining the optimal constant $\sigma_{K,0}$, the r.h.s. of~\eqref{eq:sigmaK0_opt} being an optimization problem over a compact set of $\RR^{K+2}$. For $K=2$, this optimization problem can be solved by hand, and we get
\begin{theorem}
\begin{align*}
\max_{\substack{ t\in[-1,1] \\ y\in\Omega_2} } \max\left( \overline{\varphi}_2(y)(t) - P_2(y)(t),\, P_2(y)(t) - \underline{\varphi}_2(y)(t) \right) = \frac{14\sqrt{7}-20}{27}\quad (\approx 0.6311),
\end{align*}
which is the optimal value of $\sigma_{2,0}$.
\end{theorem}
For $K=3$, solving the optimization problem by hand may also be within reach, but for larger values of $K$ it seems a formidable task. However, we can use interval arithmetic and adaptively subdivide the compact set $[-1,1]\times\Omega_K$ in order to get tight rigorous enclosure of the optimal value of $\sigma_{K,0}$.

\begin{theorem}
For $K\in\{2,\ldots,6\}$, the optimal value of $\sigma_{K,0}$ is contained in the enclosure given in the second column of Table~\ref{tab:sigmaK0} below.
\begin{table}[h!]
\begin{center}
\begin{tabular}{ c||c|c } 
 $K$ & optimal $\sigma_{K,0}$ (computed via Prop.~\ref{prop:sigma2optim}) & $\sigma_{K,0}$ computed via Th.~\ref{thm:interp_error} \\ \hline
 \hline
 2 & [0.6311,0.6312] & 1 \\ \hline 
 3 & [0.6666,0.6667] & 1  \\ \hline 
 4 & [0.5253,0.5254] & 0.8793  \\ \hline 
 5 & [0.4944,0.4945] & 0.7825  \\ \hline 
 6 & [0.4261,0.4265] & 0.6966  
\end{tabular}
\caption{Enclosures for the optimal value of $\sigma_{K,0}$, in the second column. For the sake of comparison, we give in the third column the value of $\sigma_{K,0}$ provided by Theorem~\ref{thm:interp_error}.}
\label{tab:sigmaK0}
\end{center}
\end{table}
\end{theorem}
\begin{proof}
The proof can be reproduced by running \texttt{script\_InterpConstant.m} available at~\cite{integratorcode}, together with \textsc{Intlab}~\cite{Intlab}.
\end{proof}	

\begin{remark}
We only provide enclosures up to $K=6$ because larger values of $K$ did not prove useful in this work: in order to get a smaller contraction rate (i.e. $Z$ bound), increasing the number of subdomains is more efficient than increasing $K$. This was already observed in~\cite{BreLes18} when using interpolation in time for ODEs, and becomes even more important in this work, as increasing the number of subdomains also improves our piece-wise constant in time approximation $\L$ of $DH(\bu)$ (with the notations of the introduction), which in turns improves the $Z$ bound. The only reason to increase $K$ is to get a smaller defect (the $Y$ bound).

However, it should also be noted that the computational cost of the algorithm we used to obtain rigorous enclosures of $\sigma_{K,0}$ increases dramatically with $K$, and that this approach does not seem suitable without algorithmic improvements for much larger values of $K$.
\end{remark}

\begin{remark}
Since the constant $\sigma_{K,0}$ plays a crucial role in our validation estimates, as it appears in the $Z_{\infty N}$ bound, see~\eqref{eq:Z1inftyN}, a possible way to improve upon the setup presented in this paper would be, in the interpolation in time, to replace the Chebyshev nodes by interpolation nodes chosen to minimize $\sigma_{K,0}$, but we did not explore this option in this work.
\end{remark}
\section{Quadrature for the exponential integrals}
\label{app:quadrature}

In this section we explain how to compute rigorous enclosures for 
\begin{align}	
	\int_{-1}^{t} e^{ \lambda \left( t - s \right) } \psi(s) \ \mbox{d}s 
	&= 
	\frac{t+1}{2} \int_{-1}^{1} \exp \left( \frac{\lambda}{2} \left( t + 1 \right) \left( 1- s \right) \right)
	\psi \left( \frac{t+1}{2} \left( s + 1 \right) -1 \right) \mbox{d}s,
	\label{eq:var_int}
\end{align}
where $\lambda\in \CC$, $t\in (-1,1]$ and $\psi: [-1,1] \rightarrow \CC$ is a polynomial function, written in the Chebyshev basis
\begin{align}
\label{eq:psi_coeffs}
	\psi = \psi_{0} + 2 \sum_{k=1}^{K} \psi_{k} T_{k},
\end{align}
for some $K \in \NN_{0}$.
These enclosures are needed to 
rigorously evaluate the map $F_{KN}$ on the computer. The idea is to approximate \eqref{eq:var_int} 
by using a quadrature rule based at the Chebyshev points and to compute a bound for the associated error. 

\paragraph{Clenshaw-Curtis quadrature.}
Let $( t^{K_{1}}_{k} )_{k=0}^{K_{1}}$ denote the Chebyshev points of order $K_{1} \in \NN_{\geq 2}$
(see Definition \ref{def:chebpoints}). The parameter $K_{1}$ is the order of the quadrature
rule and need not be equal to $K$. If $f : \left[-1,1\right] \rightarrow \CC$ is sufficiently smooth and $K_{1}$
is sufficiently large, then 
\begin{align*}
	\int_{-1}^{1} f (s) \ \mbox{d} s \approx \int_{-1}^{1} f_{K_{1}} (s) \ \mbox{d}s,
\end{align*}
where $f_{K_{1}}$ is the Chebyshev interpolant of $f$ of order $K_{1}$. In particular, note that
\begin{align*}
	f_{K_{1}} = \sum_{k=0}^{K_{1}} f \left( t^{K_{1}}_{k} \right) \phi_{k} , 
\end{align*}
where $( \phi_{k} )_{k=0}^{K_{1}}$ are the \emph{Lagrange polynomials} associated to 
$( t^{K_{1}}_{k} )_{k=0}^{K_{1}}$, i.e., $\phi_{k} ( t^{K_{1}}_{l} ) = \delta_{kl}$
for $0 \leq k,l \leq K_{1}$. Hence
\begin{align}
	\label{eq:quadrature}
	\int_{-1}^{1} f_{K_{1}} (s) \ \mbox{d}s = \sum_{k=0}^{K_{1}} f ( t_{k}^{K_1} ) w_{k}, 
	\quad
		w_{k} \bydef  \int_{-1}^{1} \phi_{k}(s) \ \mbox{d}s, \quad 0 \leq k \leq K_{1}.
\end{align}
\begin{remark}
	The weights $\left( w_{k} \right)_{k=0}^{K_{1}}$ are independent of the objective function $f$ and are commonly referred 
	to as the Chebyshev quadrature weights. These weights can be efficiently computed by using the Discrete Fourier
	Transform as explained in \cite{Quadrature}. 
\end{remark}
In conclusion, the integral of $f$ can be approximated by 
\begin{align*}
  	\int_{-1}^{1} f (s) \ \mbox{d} s \approx	\sum_{k=0}^{K_{1}} w_{k} f ( t_{k}^{K_1} ) .
\end{align*}
This particular quadrature rule is referred to as \emph{Clenshaw-Curtis quadrature}. Furthermore, 
the latter approximation is exact whenever $f$ is a polynomial of at most order $K_{1}$. 
The reader is referred to \cite{ApproximationTheory, Quadrature} for a more detailed treatment
of Clenshaw-Curtis quadrature. 

The next theorem, presented in \cite{ApproximationTheory}, provides a bound for the error
associated to Clenshaw-Curtis quadrature:
\begin{theorem}
	\label{thm:Clenshaw}
	Suppose $f: \left[ -1, 1\right] \rightarrow \CC$ can be analytically extended to the open ellipse $\mathcal{E}_{\rho}$,
	for some $\rho >1$. If $f$ is bounded on $\mathcal{E}_{\rho}$, and $K_1$ is even then 
	\normalfont
	\begin{align*}
		\left \vert \int_{-1}^{1} f(s) \ \mbox{d}s - \sum_{k=0}^{K_{1}} w_{k} f ( t_{k}^{K_1} ) \right \vert \leq
		\frac{64}{15} \frac{\rho^{-K_{1}}}{\rho^{2}-1} \sup_{z \in \mathcal{E}_{\rho}} \left \vert f(z) \right \vert.
	\end{align*}
	\begin{proof}
		See \cite[Theorem $19.3$]{ApproximationTheory}. The statement there is only made for real valued functions, but the proof readily extends to complex valued functions.
	\end{proof}
\end{theorem}

As mentioned before, the strategy is to use Clenshaw-Curtis quadrature to approximate \eqref{eq:var_int} and to 
bound the associated error with the aid of the Theorem \ref{thm:Clenshaw}. To use this theorem, however, we
need one final estimate to bound the integrand in \eqref{eq:var_int} on $\mathcal{E}_{\rho}$, and this is where it proves convenient to have $\psi$ written in the Chebyshev basis:

\begin{lemma}
Let $\rho\geq 1$, $t \in [-1,1]$, $\psi$ as in~\eqref{eq:psi_coeffs}, and write $\lambda=\alpha+i\beta$, $\alpha,\beta\in\RR$. Then 
	\begin{align*}
		& \sup_{z \in \mathcal{E}_{\rho}} \left \vert 
		\frac{t+1}{2} \exp \left( \frac{\lambda}{2} \left( t + 1 \right) \left( 1- z \right) \right)
		\psi \left( \frac{t+1}{2} \left( z + 1 \right) -1 \right) \normalfont \right \vert 	\\
		 & \qquad\leq 
		\frac{ t+1}{2} 
		\exp \left( \alpha + \frac{1}{2}\sqrt{\alpha^2(\rho+\rho^{-1})^2 + \beta^2 (\rho-\rho^{-1})^2} \right)
		\left( \left \vert \psi_{0} \right \vert + \sum_{k=1}^{K} \left \vert \psi_{k} \right \vert \left( \rho^{k} + \rho^{-k} \right) \right) .
	\end{align*} 
\end{lemma}
\begin{proof}
Since $\frac{t+1}{2} \left(\mathcal{E}_{\rho}  + 1 \right) -1 \subset \mathcal{E}_{\rho}$, it follows from Lemma~\ref{lem:upperbounds_norms} that
\begin{align*} 
		& \sup_{z \in \mathcal{E}_{\rho}} \left \vert 
		\psi \left( \frac{t+1}{2} \left( z + 1 \right) -1 \right) \normalfont \right \vert \leq \left \vert \psi_{0} \right \vert + \sum_{k=1}^{K} \left \vert \psi_{k} \right \vert \left( \rho^{k} + \rho^{-k} \right).
	\end{align*}
Besides, we have
\begin{align*}
		& \sup_{z \in \mathcal{E}_{\rho}} \left \vert 
		\frac{t+1}{2} \exp \left( \frac{\lambda}{2} \left( t + 1 \right) \left( 1- z \right) \right) \right \vert 	=
		\frac{ t+1}{2} 
		\exp \left( \frac{t + 1}{2} \sup_{z \in \mathcal{E}_{\rho}} \Re\left(\lambda\left( 1- z \right)\right) \right).
	\end{align*} 
It only remains to be proven that
\begin{equation}
\label{eq:max_exp}
\sup_{z \in \mathcal{E}_{\rho}} \Re\left((\alpha+i\beta)\left( 1- z \right)\right) = \alpha + \frac{1}{2}\sqrt{\alpha^2(\rho+\rho^{-1})^2 + \beta^2 (\rho-\rho^{-1})^2}.
\end{equation}
Any $z\in\mathcal{E}_{\rho}$ can be written $z=\frac{1}{2}\left(re^{i\theta}+r^{-1}e^{-i\theta}\right)$, for some $(r,\theta)\in[1,\rho]\times[0,2\pi]$. We then have
\begin{align*}
\Re\left((\alpha+i\beta)\left( 1- z \right)\right) &= \alpha\left(1-\frac{1}{2}(r+r^{-1})\cos\theta\right) + \beta\frac{1}{2}(r-r^{-1})\sin\theta \\
&= \alpha + \frac{1}{2}\sqrt{\alpha^2(r+r^{-1})^2 + \beta^2 (r-r^{-1})^2}\cos(\theta-\theta_r),
\end{align*}
for some $\theta_r\in[0,2\pi]$, which yields~\eqref{eq:max_exp}.
\end{proof}

\section{Computing the $\chi_N$ bounds}
\label{app:chi}

We provide here an explicit procedure for computing the bounds $\chi_N^{(j)}$, $j=0,\ldots,2R-1$, introduced in Lemma~\ref{lem:tailEstimate} and used in the $Z_\infty$ and $W_\infty$ bounds.

It will be useful to consider
\begin{equation*}
\alpha_l = (-1)^l \left(\bv^{(2l)}\right)_0,\qquad l=0,\ldots,R-1,
\end{equation*}
and the polynomial
\begin{equation*}
P(X) = -X^{R}+\sum_{l=0}^{R-1}\alpha_l X^{l},
\end{equation*}
so that
\begin{equation*}
\Re(\lambda_n) = P(n^2).
\end{equation*}

\begin{enumerate}
\item Define
\begin{equation*}
N^{(j)} \bydef \max\left(\sqrt{\sum_{l=0}^{R-1}\frac{\vert 2l-j\vert}{2R-j}\vert\alpha_l\vert},N+1\right),
\end{equation*}
for which (since $N\geq 0$)
\begin{equation*}
\left(N^{(j)}\right)^2 \geq \max\left(\sum_{l=0}^{R-1}\frac{\vert 2l-j\vert}{2R-j}\vert\alpha_l\vert,1\right).
\end{equation*}
We have
\begin{align*}
-jP(X) + 2XP'(X) = (j-2R)X^R + \sum_{l=0}^{R-1} (2l-j)\alpha_jX^l,
\end{align*}
hence, for $n\geq N^{(j)}$, 
\begin{align*}
-jP(n^2) + 2n^2P'(n^2) \leq 0,
\end{align*}
and therefore the map $n\mapsto \frac{n^j}{-\Re(\lambda_n)}$ is decreasing. Since this map goes to $0$ when $n$ goes to infinity, it means it is also positive for $n\geq N^{(j)}$, in which case $\Re(\lambda_n)$ must be negative, therefore
\begin{align*}
\sup_{n\geq N^{(j)}} n^j \frac{ 1- e^{2\tau\Re(\lambda_{n})} }{ -\Re(\lambda_{n})} \leq 
\sup_{n\geq N^{(j)}}  \frac{ n^j}{ -\Re(\lambda_{n})} \leq 
\frac{ \left(N^{(j)}\right)^j}{ -\Re\left(\lambda_{N^{(j)}}\right)}.
\end{align*}
\item Compute
\begin{align*}
\tilde \chi_N^{(j)} \bydef \max_{N < n \leq  N^{(j)}} n^j \frac{ 1- e^{2\tau\Re(\lambda_{n})} }{ -\Re(\lambda_{n})}.
\end{align*}
\begin{itemize}
\item If 
\begin{align*}
\frac{ \left(N^{(j)}\right)^j}{ -\Re\left(\lambda_{N^{(j)}}\right)} \leq \tilde \chi_N^{(j)},
\end{align*}
then $\chi_N^{(j)} = \tilde \chi_N^{(j)}$ and we are done.
\item Otherwise, compute $\tilde{N}^{(j)}>N^{(j)}$ so that 
\begin{align*}
\frac{ \left(\tilde N^{(j)}\right)^j}{ -\Re\left(\lambda_{\tilde N^{(j)}}\right)} \leq \tilde \chi_N^{(j)},
\end{align*}
and then we get
\begin{align*}
\chi_N^{(j)} = \max_{N < n \leq  \tilde N^{(j)}} n^j \frac{ 1- e^{2\tau\Re(\lambda_{n})} }{ -\Re(\lambda_{n})}.
\end{align*}
\end{itemize}
\end{enumerate}

\bibliographystyle{abbrv}
\bibliography{Bibliography} 

\end{document}